\documentclass{article}
\usepackage{graphicx} 
\usepackage{amsthm}
\usepackage{amsfonts, amsmath, amssymb, amsthm}
\usepackage{hyperref}
\usepackage{cleveref}
\usepackage[utf8x]{inputenc}
\usepackage{url}
\usepackage{multicol}
\usepackage{enumitem}
\usepackage{tikz}

\usetikzlibrary{shapes,arrows}
\usetikzlibrary{decorations.markings}
\usetikzlibrary{arrows.meta}

\def\msf{\mathsf}

\def\mcal{\mathcal}
\def\h{\hat}
\def\t{\tilde}

\def\uhr{\upharpoonright}

\def\w{\mathsf{WSub}}
\def\wsamelong{\mathsf{WSub}^{=}}
\def\wshorter{\mathsf{WSub}^{*}}
\def\vw{\mathsf{OSub}}
\def\vwsamelong{\mathsf{OSub}^=}
\def\vwshorter{\mathsf{OSub}^{*}}
\def\sectionallymonochromatic{level-homogeneous}
\def\sufficient{matching}
\def\hjsufficient{$A^{<\omega}\rightarrow\ell$-matching}
\def\potentiallyforces{\operatorname{{?}{\vdash}}}

\def\forces{\Vdash}

\def\nonprogressable{non-progressable}
\def\progressable{progressable}
\def\wc{WitCol}

\def\wHJ{$\exists^\infty$-$\msf{IHJ}$}

\def\GRtree{Graham-Rothschild tree}
\def\msflow{\mathsf{LOW}}

\def\wOVW{$\exists^\infty$-$\msf{OVW}^0$}

\def\PP{\mathbb{P}}

\def\QQ{\mathbb{Q}}

\def\HH{\mathbb{H}}
\def\RR{\mathbb{R}}
\def\GG{\mathbb{G}}

\newcommand{\Psf}{\mathsf{P}}
\newcommand{\Qsf}{\mathsf{Q}}

\def\F{\mathcal{F}}

\def\C{\mathcal{C}}
\def\I{\mathcal{I}}

\def\L{\mathcal{L}}

\def\T{\mathcal{T}}
\def\Sc{\mathcal{S}}
\def\M{\mathcal{M}}
\def\Nc{\mathcal{N}}

\def\R{\mathcal{R}}

\newcommand{\uh}{\upharpoonright}

\newcommand{\qvdash}{\operatorname{{?}{\vdash}}}
\newcommand{\nqvdash}{\operatorname{{?}{\nvdash}}}

\newcommand{\BSig}{\mathsf{B}\Sigma^0}

\newcommand{\RCA}[0]{\mathsf{RCA}}
\newcommand{\WKL}[0]{\mathsf{WKL}}
\newcommand{\ACA}[0]{\mathsf{ACA}}

\newcommand{\PIOOCA}[0]{\Pi^1_1\mathsf{-CA}}

\newcommand{\RT}[0]{\mathsf{RT}}

\newcommand{\CCSL}[0]{\mathsf{CCSL}}

\newcommand{\SADS}[0]{\mathsf{SADS}}

\newcommand{\CSL}[0]{\mathsf{CSL}}
\newcommand{\OVW}[0]{\mathsf{OVW}}
\newcommand{\DRT}[0]{\mathsf{DRT}}
\newcommand{\ODRT}[0]{\mathsf{ODRT}}
\newcommand{\CDRT}[0]{\mathsf{CDRT}}

\newcommand{\IOOVW}[0]{\exists^\infty\mbox{-}\mathsf{OVW}}
\newcommand{\IOIHJ}[0]{\exists^\infty\mbox{-}\mathsf{IHJ}}
\newcommand{\LOVW}[0]{\mathsf{LOVW}}
\newcommand{\LCSL}[0]{\mathsf{LCSL}}

\newcommand{\NN}[0]{\mathbb{N}}

\newcommand{\card}{\operatorname{card}} 

\newcommand{\OSub}{\mathsf{OSub}}
\newcommand{\USub}{\mathsf{USub}}
\newcommand{\WSub}{\mathsf{WSub}}
\newcommand{\SubW}{\mathsf{Sub}}

\def\qt#1{``#1''}%

\title{The reverse mathematics of the Ordered Variable Word theorem}
\date{\today}

\newtheorem{theorem}{Theorem}
\numberwithin{theorem}{section}
\newtheorem{maintheorem}[theorem]{Main Theorem}
\newtheorem{lemma}[theorem]{Lemma}

\newtheorem{proposition}[theorem]{Proposition}

\newtheorem{corollary}[theorem]{Corollary}

\newtheorem{claim}{Claim}
\newtheorem{subclaim}{Subclaim}[claim] 

\AtBeginEnvironment{theorem}{\setcounter{claim}{0}}
\AtBeginEnvironment{maintheorem}{\setcounter{claim}{0}}
\AtBeginEnvironment{lemma}{\setcounter{claim}{0}}
\AtBeginEnvironment{proposition}{\setcounter{claim}{0}}
\AtBeginEnvironment{corollary}{\setcounter{claim}{0}}
\AtBeginEnvironment{question}{\setcounter{claim}{0}}
\AtBeginEnvironment{example}{\setcounter{claim}{0}}

\theoremstyle{definition}
\newtheorem{remark}[theorem]{Remark}
\newtheorem{definition}[theorem]{Definition}
\makeatletter
\newtheorem*{rep@theorem}{\rep@title}
\newcommand{\newreptheorem}[2]{%
  \newenvironment{rep#1}[2][]{%
    \def\optarg{##1}%
    \ifx\optarg\@empty
      \def\rep@title{#2 \ref{##2}}%
    \else
      \def\rep@title{#2 \ref{##2} (##1)}%
    \fi
    \begin{rep@theorem}}%
    {\end{rep@theorem}}}
\makeatother

\makeatletter
\newtheorem*{rep@proposition}{\rep@title}

\newcommand{\newrepproposition}[2]{%
  \newenvironment{rep#1}[1]{%
    \def\rep@title{#2 \ref{##1}}%
    \begin{rep@proposition}}%
  {\end{rep@proposition}}%
}
\makeatother

\newcommand{\upupharpoonright}{%
  \upharpoonright
  \mathrel{\mkern-4.5mu}%
  \upharpoonright
}
\newcommand{\uuh}{\upupharpoonright}

\newreptheorem{theorem}{Theorem}
\newreptheorem{maintheorem}{Main Theorem}
\newreptheorem{lemma}{Lemma}

\usepackage{xcolor}

\author{Lu Liu \and Ludovic Patey}

\begin{document}

\maketitle

\begin{abstract}
In this article, we study the reverse mathematics of variable word theorems used in the proof of the Dual Ramsey theorem. We prove that the Ordered Variable Word theorem does not imply Ramsey's theorem for pairs and that every computable instance admits a solution of low${}_2$ degree. This is used to prove the Carlson-Simpson Lemma and the Open Dual Ramsey theorem over~$\ACA_0$, thereby answering some 40-years old open questions. These results have consequences in the reverse mathematics of structural Ramsey theory.
\end{abstract}

\section{Introduction}

Celebrated Ramsey's theorem~\cite{ramsey1929problem} ($\RT^n$) states the existence, for every $k$-coloring of $[\NN]^n$,
of an infinite set $H \subseteq \NN$ on which $[H]^n$ is monochromatic, for $k, n \geq 1$. Here, we write $[X]^n$ for the collection of all $n$-subsets of~$X$. In 1984, Carlson and Simpson~\cite{carlson1984dual} proved a dual version of Ramsey's theorem, in which $k$-colorings of $\NN$ are colored instead of $[\NN]^n$.
Given $\alpha \leq \omega$, let $(\NN)^\alpha$ denotes the class of partitions of $\NN$ into exactly $\alpha$ non-empty parts, and given $p \in (\NN)^\omega$, let $(p)^\alpha$ denotes the class of coarsenings of $p$ into exactly $\alpha$ many parts, that is, given $q \in (p)^\alpha$ and $x, y \in \NN$, if $p(x) = p(y)$, then $q(x) = q(y)$.

\begin{theorem}[Dual Ramsey Theorem~\cite{carlson1984dual}]
For all finite $k, \ell \geq 1$, for every Borel $\ell$-coloring of $(\NN)^k$,
there exists some $p \in (\NN)^\omega$ such that $(p)^k$ is monochromatic.
\end{theorem}

We write $\DRT^k$ and $\ODRT^k$ for the Dual Ramsey theorem restricted to Borel and open colorings of $(\NN)^k$, respectively. The core combinatorial content of the Open Dual Ramsey theorem is known as the \emph{Carlson-Simpson Lemma} and can be formulated in terms of variable word theorems.

A finite or infinite \emph{word} over a finite alphabet~$A$ is a map $w : \alpha \to A$ for some~$\alpha \leq \omega$.
We write $|w|$ for its \emph{length}~$\alpha$. Given $\beta \leq \omega$, a $\beta$-variable word over~$A$ is a word $w$ over the alphabet $A \sqcup \{ x_i : i < \beta \}$ where each $x_i$ appears at least once, and the first occurrence of $x_i$ appears before the first occurrence of $x_{i+1}$. We write $A^{\alpha, \beta}$ and $A^{<\alpha, \beta}$ for the sets of all $\beta$-variable word of length~$\alpha$, and smaller than~$\alpha$, respectively. Given $\gamma \leq \beta \leq \omega$, a $\beta$-variable word~$w$ over~$A$ and a $\gamma$-variable word~$u$ of length~$|u| \leq \beta$, we write $w[u]$ for the $\gamma$-variable word over~$A$ where each occurrence of $x_j$ is replaced with $u(j)$, and cut before the first occurrence of $x_{|u|}$. If $|u| = \beta$, then $|w[u]| = |w|$.

\begin{lemma}[Carlson-Simpson Lemma~\cite{carlson1984dual}]
Fix a finite alphabet~$A$ and some $n, \ell \in \NN$. For every $\ell$-coloring of $A^{<\omega, n}$,
there exists an $\omega$-variable word~$W$ over~$A$ such that $\{ W[u] : u \in A^{<\omega, n} \}$ is monochromatic.
\end{lemma}

We write $\CSL^n$ for the Carlson-Simpson Lemma restricted to $\ell$-colorings of $A^{<\omega, n}$ for any finite alphabet~$A$ and any $\ell \geq 1$. Since then, the Carlson-Simpson Lemma found other applications in proving the existence of finite big Ramsey degree in structural Ramsey theory~\cite{hubivcka2020big}, such as reproving a result by Dobrinen~\cite{dobrinen2020ramsey} for the triangle-free Henson graph.
The Carlson-Simpson Lemma admits an inductive proof along~$n$ which is similar to the proof of Ramsey's theorem. For the base case, Carlson and Simpson~\cite[Theorem 6.3]{carlson1984dual} proved a stronger statement in terms of \emph{ordered} variable words. Given $\beta \leq \omega$, an ordered $\beta$-variable word over~$A$ is a word $w$ over the alphabet $A \sqcup \{ x_i : i < \beta \}$ where each $x_i$ appears at least once, and the \emph{last} occurrence of $x_i$ appears before the first occurrence of $x_{i+1}$. Note that contrary to the unordered version, each variable kind appears only finitely often in an ordered $\omega$-variable word. 

\begin{theorem}[Ordered Variable Word theorem~\cite{carlson1984dual}]
Fix a finite alphabet~$A$ and some $\ell \in \NN$. For every $\ell$-coloring of $A^{<\omega, 0}$,
there exists an ordered $\omega$-variable word~$W$ over~$A$ such that $\{ W[u] : u \in A^{<\omega, 0} \}$ is monochromatic.
\end{theorem}

We write $\OVW^0$ for the Ordered Variable Word theorem restricted to $\ell$-colorings of $A^{<\omega, 0}$ for any finite alphabet~$A$ and any $\ell \geq 1$. There exists actually an ordered counterpart of the Carlson-Simpson Lemma proven by Carlson~\cite{carlson1988unifying} but for which no elementary proof is known. 

We are interested in the meta-mathematics of the Dual Ramsey theorem and its core lemmas from the viewpoint of computability theory and reverse mathematics.

\subsection{Reverse mathematics}

Reverse mathematics is a foundational program started by Harvey Friedman, whose goal is to find optimal axioms to prove ordinary theorems. It uses the framework of subsystems of second-order arithmetic, with a base theory, $\RCA_0$, capturing \emph{computable mathematics}. More precisely, $\RCA_0$ consists of the axioms of a discrete ordered commutative semiring, together with the $\Sigma^0_1$-induction scheme and the $\Delta^0_1$-comprehension scheme. See Simpson~\cite{simpson2009subsystems} or Dzhafarov and Mummert~\cite{dzhafarov2022reverse} for a reference in reverse mathematics. Overall this article, we shall be interested in the following subsystems, which are listed in increasing axiomatic strength:
\begin{itemize}
    \item $\WKL_0$, standing for \emph{Weak K\"onig's lemma}, is $\RCA_0$ augmented with the statement \qt{Every infinite binary tree admits an infinite path}. It can be thought of as capturing compactness arguments.
    \item $\ACA_0$, standing for \emph{Arithmetic Comprehension Axiom}, is $\RCA_0$ together with the comprehension scheme for arithmetic formulas (where set parameters are allowed). It is a $\Pi^1_1$-conservative extension of  Peano arithmetic. From a computability-theoretic viewpoint, it can be equivalently defined as $\RCA_0$ augmented with the statement \qt{$\forall X \exists Y (Y = X')$}, where $X' = \{ e : \Phi^X_e(e)\downarrow \}$ is the Turing jump of~$X$.
    \item $\ACA_0'$ is $\RCA_0$ augmented with the statement $\forall n \forall X \exists Y (Y = X^{(n)})$,
    where $X^{(n)}$ is the $n$th iterate of the Turing jump, defined by $X^{(0)} = X$ and $X^{(n+1)} = (X^{(n)})'$.
    By Yokoyama~\cite{yokoyama2023paris} (see also~\cite{pacheco2022determinacy}), $\ACA_0'$ is equivalent to the $\Pi^1_2$-reflection principle for $\ACA_0$, which is the statement \qt{every $\Pi^1_2$-formula provable in~$\ACA_0$ is true}.
    \item $\ACA_0^+$ is $\RCA_0$ augmented with the statement $\forall X \exists Y (Y = X^{(\omega)})$, where $X^{(\omega)} = \bigoplus_n X^{(n)}$ is the $\omega$-jump of~$X$. Given the Cantor pairing function $\langle \cdot, \cdot \rangle : \NN^2 \to \NN$, the operator $\bigoplus_n X_n = \{ \langle a, b \rangle : a \in X_n \rangle \}$ is the effective union of the sequence $(X_n)_{n \in \NN}$. The system $\ACA_0^+$ was popularized as being the weakest system known to prove Hindman's theorem.
    \item $\PIOOCA_0$ is $\RCA_0$ augmented with the comprehension scheme for $\Pi^1_1$-formulas. As such, it is an impredicative system of extremely high axiomatic strength compared to the systems above.
\end{itemize}

\noindent
We shall mostly study \emph{$\Pi^1_2$-problems}, which are $\Pi^1_2$-sentences of the form
$$
\Psf = \forall X[\Phi(X) \to \exists Y \Psi(X,Y)]
$$
where $\Phi$ and $\Psi$ are arithmetic formulas. A \emph{$\Psf$-instance} is any set~$X$ such that $\Phi(X)$ holds,
and a \emph{$\Psf$-solution to~$X$} is a set~$Y$ such that $\Psi(X,Y)$ holds. For example, an \emph{$\OVW^0$-instance} is a coloring $f : A^{<\omega, 0} \to \ell$ for a finite alphabet~$A$ and some~$\ell \geq 1$. An \emph{$\OVW^0$-solution} to~$f$ is an ordered $\omega$-variable word~$W$ such that $\{ W[u] : u \in A^{<\omega,0} \}$ is monochromatic.
\smallskip

In their article, Carlson and Simpson~\cite{carlson1984dual} asked about the proof-theoretic and computability-theoretic strength of the Dual Ramsey theorem and its combinatorial lemmas. In an unpublished paper, Slaman~\cite{slamanNoteDualRamsey1997} noticed that $\DRT^k_\ell$ holds over~$\PIOOCA_0$. Miller and Solomon~\cite{miller2004effectiveness} proved that $\ODRT^{k+1}$ implies $\RT^k$ over~$\RCA_0$, hence that $\ODRT^4$ implies $\ACA_0$ over~$\RCA_0$. They also proved that $\OVW^0$ is not provable over~$\WKL_0$, even for binary alphabets and 2-colorings. Erhard~\cite{erhard2013carlson} proved that $\OVW^0$ is not provable by Ramsey's theorem for pairs restricted to stable colorings over~$\RCA_0$. Dzhafarov, Flood, Solomon and Westrick \cite{dzhafarov2021effectiveness} studied extensively the Dual Ramsey theorem and its variable word lemmas.
They showed in particular that the full statement of the Dual Ramsey theorem is sensitive to representation, in terms of Borel or Baire codes. Liu, Monin and Patey~\cite{liu2019computable} proved $\OVW^0$ over~$\ACA_0$ for binary alphabets and constructed a computable 2-coloring of $\{0,1\}^{<\omega,0}$ with no $\Delta^0_2$ $\OVW^0$-solution. More recently, Anglès d’Auriac, Mignoty, Liu and Patey~\cite{angles2023carlson} proved two versions of $\OVW^0$: a piecewise syndetic and a density version, the latter being based on Dodos, Kanellopoulos and Tyros~\cite{dodos2014density}. They deduced that the Carlson-Simpson Lemma holds over~$\ACA_0^+$.

\subsection{Main contributions}

In this article, we give a new proof of $\OVW^0$ in the style of Towsner~\cite{towsner2012simple}, and derive from its combinatorics a new parametric notion of effective forcing for building solutions to $\OVW^0$ with a good first-jump control. Thanks to this notion of forcing, we prove the following theorem:

\begin{maintheorem}\label[maintheorem]{mthm:ovw0-rt22}
$\OVW^0$ implies neither $\RT^2_2$ nor $\ACA_0$ over~$\RCA_0$.
\end{maintheorem}

Moreover, such a separation is witnessed by an \emph{$\omega$-model}, that is, a structure $\M = (M, S, 0, 1, +, \cdots, <)$ in which the first-order part consists of the standard integers $\omega$, equipped with the natural operations and order. This answers partially~\cite[Question 11]{montalban2011open} and \cite[Problem 2]{miller2004effectiveness} and answers negatively \cite[Problem 1]{miller2004effectiveness} and \cite[Question 4.9]{liu2019computable}.

Second, we prove that every computable instance of $\OVW^0$ admits a solution of low${}_2$ degree
by constructing a $\emptyset'$-computable 1-generic filter for this notion of forcing. Formalizing the construction over $\ACA_0$, we use this as the base case of an inductive proof $\CSL^n$ over~$\ACA_0$. We therefore obtain an exact characterization of the strength of the Carlson-Simpson Lemma and the Open Dual Ramsey Theorem.

\begin{maintheorem}\label[maintheorem]{mthm:csl-aca}
For every~$n \in \omega$, $\ACA_0 \vdash \CSL^n$ and $\ACA_0 \vdash \ODRT^{n+2}$.
Moreover, a reversal holds for every~$n \geq 2$.
\end{maintheorem}

This therefore gives a partial answer to the original question of Carlson and Simpson~\cite{carlson1984dual}
and a complete answer to~\cite[Problem 3]{miller2004effectiveness}. Based on the work of
Hubi\v{c}ka~\cite{hubivcka2020big},
formalized in reverse mathematics by Anglès d’Auriac et al.~\cite{angles2023carlson}, \Cref{mthm:csl-aca} shows over $\ACA_0$ that the triangle-free Henson graph has finite big Ramsey degree. Note that this fact was recently proven over $\ACA_0$ by Cholak, Dobrinen, and Towsner~\cite{cholak2026hensonupper} for every Henson graph, using a Milliken-style proof for coding trees, yielding exact bounds for finite big Ramsey degree contrary to the approach of
Hubi\v{c}ka~\cite{hubivcka2020big}.

Last, we complete the picture by showing that $\CSL^1$ and $\ODRT^3$ are logically weak, with the following theorem:

\begin{maintheorem}\label[maintheorem]{mthm:csl1-avoidance}
Neither $\CSL^1$, nor $\ODRT^3$ implies $\ACA_0$ over~$\RCA_0$.
\end{maintheorem}

\Cref{mthm:csl1-avoidance} is by far the most complicated part of the article, involving three non-trivial steps: we first prove that an infinite variant of the Hales-Jewett theorem admits so-called strong preservation of hyperimmunity, then use it to prove that the Ordered Variable Word theorem admits the same property, and last, we prove that a cohesive version of $\CSL^1$ admits preservation of hyperimmunity. 

\Cref{mthm:csl1-avoidance} shows that the statement of the indivisibility of the triangle-free Henson graph does not imply $\ACA_0$, answering a question of Gill~\cite{gill2023note} and Cholak, Dobrinen and McCoy~\cite[Question 48]{cholak2026henson}.

\subsection{Organization of the paper}

We start with \Cref{sec:towsner-ovw0}, which provides a simple proof of $\OVW^0$ over $\ACA_0$ in the style of Towsner~\cite{towsner2012simple}. This settles the combinatorial objects which are then used in \Cref{sec:weakly-low} to design a parametric notion of effective forcing with good first-jump control. In particular, we prove a low${}_2$ basis theorem for $\OVW^0$. Then, in \Cref{sec:preservation-hyp}, we prove that $\OVW^0$ admits so-called preservation of $\omega$ hyperimmunities, and deduce \Cref{mthm:ovw0-rt22}. In \Cref{scaOVW-subsec1}, we prove that an infinite version of the hales-jewett theorem admits so-called strong preservation of $\omega$ hyperimmunities. We use this result in \Cref{scaOVW-sec3} to prove that $\OVW^0$ admits strong preservation of 1 hyperimmunity.  In \Cref{sec:csl-aca}, a formalization of the low${}_2$ basis theorem for $\OVW^0$ from \Cref{sec:weakly-low} is used to prove \Cref{mthm:csl-aca}. In \Cref{sec:weakness-csl1}, we prove that a cohesive version of $\CSL^1$ admits preservation of $\omega$ hyperimmunities and, combined with \Cref{scaOVW-sec3}, we deduce \Cref{mthm:csl1-avoidance}. Last, in \Cref{sec:applications}, we derive consequences of \Cref{mthm:csl-aca} and \Cref{mthm:csl1-avoidance} to structural Ramsey theory.

\subsection{Notation}

\emph{Computability theory.} We shall use the standard notations of computability theory. See Soare~\cite{soareTuringComputability2016} for a reference. In particular, we let $\Phi_0, \Phi_1, \dots$ be an enumeration of all Turing functionals, and write $\Phi^A_e(x)\downarrow = y$ if the $e$th functional with oracle~$A$ halts on input~$x$ and outputs~$y$. Otherwise, we say it diverges on input~$x$, and write $\Phi^A_e(x)\uparrow$.
\smallskip

\noindent
\emph{Words.} Given two (finite or infinite) words~$w_0, w_1$, we write $w_0 \preceq w_1$ if $w_0$ is a prefix of~$w_1$.
Given a finite or infinite word~$W$ and some length~$t \leq |W|$, we write $W \uh_t$ for the initial segment of~$W$ of length~$t$. More generally, given $t_0 \leq t_1 \leq |W|$, we write $W \uh_{t_0}^{t_1}$ for the sequence of length~$t_1 - t_0$ such that $W \uh_{t_0} \cdot\ W \uh_{t_0}^{t_1} = W \uh_{t_1}$. We write $\epsilon$ for the empty word. In particular, given a variable word~$w$, $w[\epsilon]$ is the initial segment of~$w$ cut before the first occurrence of~$x_0$.
\smallskip

\noindent
\emph{Combinatorial spaces.} 
Throughout this article, we shall work with two different combinatorial spaces: ordered variable words (in \Cref{sec:towsner-ovw0,sec:weakly-low,sec:preservation-hyp,scaOVW-subsec1,scaOVW-sec3}) for $\OVW^0$, and unordered variable words (in \Cref{sec:csl-aca,sec:weakness-csl1}) for $\CSL$. To avoid introducing too many notations at once, we shall introduce the specific definitions for the corresponding spaces in \Cref{sec:notation-ovw} and \Cref{sec:notation-csl}, respectively, when they are first used.


\section{A simple proof of $\OVW^0$}\label[section]{sec:towsner-ovw0}

Anglès d'Auriac et al.~\cite{angles2023carlson} gave proofs of the piecewise syndetic and the density versions of the zero-dimensional Ordered Variable Word theorem, respectively, yielding two proofs of~$\OVW^0$ over~$\ACA_0$. Both proofs were based on finitary versions of the respective statements, one of which was proven by Dodos et al~\cite{dodos2014density}. In this section, we give an arguably more natural proof of $\OVW^0$ over~$\ACA_0$, using the Hales-Jewett theorem below as a bootstrap. The combinatorics will be then reused in later sections to show that $\OVW^0$ admits a weakly low basis (\Cref{sec:weakly-low}) and preservation of $\omega$ hyperimmunities (\Cref{sec:preservation-hyp}).

Towsner~\cite{towsner2012simple} gave a simple combinatorial proof of the Finite Union Theorem over~$\ACA_0^+$ by introducing the concepts of half-match and full-match. His technique was then used to give a combinatorial proof  Carlson's theorem for words~\cite{bompardReverseMathematicsCarlson2022} over~$\ACA_0^+$. We will give a Towsner-style proof of~$\OVW^0$. Before this, we introduce some notation.

\subsection{Notation}\label[section]{sec:notation-ovw}

Given $\beta \leq \alpha \leq \omega$, we write $A^{\alpha, \beta}_{<}$ and $A^{\leq\alpha, \beta}_{<}$ for the sets of all ordered $\beta$-variable word of length~$\alpha$, or at most~$\alpha$, respectively. When $\beta = 0$, we simply write $A^\alpha$ and $A^{\leq\alpha}$. Given ordered variable words~$u$ and $v$, when doing the concatenation $u \cdot v$, we will always assume that the variable kinds of~$u$ and $v$ are from disjoint spaces, so that there is no variable kind clash. In particular, if $u$ is an ordered $n$-variable word and $v$ is a ordered $\beta$-variable word, then $u \cdot v$ is an ordered $\beta+n$-variable word.
Given $\gamma \leq \beta \leq \omega$ and an ordered $\beta$-variable word~$w$, we let $\OSub^\gamma_{A}(w)$ be the set of all ordered $\gamma$-variable words generated by~$w$, that is,
$$
\OSub^\gamma_{A}(w) =  \{ w[u] : u \in A^{\leq\beta, \gamma}_{<} \}
$$
Note that if $\beta = \omega$ and $\gamma < \omega$, then $\OSub^\gamma_{A}(w)$ contains ordered $\gamma$-variable words of finite and of infinite length. On the other hand, $\OSub^\omega_A(w)$ contains only ordered $\omega$-variable words, hence of infinite length. Most of the time, the alphabet~$A$ will be clear from context, and we shall simply write $\OSub^\gamma(w)$ for $\OSub^\gamma_A(w)$.
Let
\begin{align}
 \vw(w)&:= \bigcup\limits_{0\leq \gamma\leq \beta}\vw^\gamma(w).
 \end{align}
Note that $\vw(w)$ includes in particular $\vw^0(w)$, that is, the set of words generated by~$w$. Since most theorems we shall consider are colorings of words, we will write $\w(w)$ for $\vw^0(w)$.
Last, we define
\begin{align}\nonumber
\vwsamelong(w) &:= \big\{ v\in \vw(w):|v| = |w|  \big\}\\ \nonumber
\vwshorter(w) &:=
\big\{
v\in\vw(w): |v|<|w|
\big\}.
\end{align}
The sets $\OSub^{\gamma,=}(w)$, $\OSub^{\gamma,\star}(w)$, $\wsamelong(w)$ and $\wshorter(w)$ are defined accordingly.

\subsection{Hales-Jewett theorem}

The following theorem is known as the Hales-Jewett theorem~\cite{hales1963regularity}. The original proof used double induction and yielded a bound $HJ$ growing as fast as the Ackermann function. Shelah~\cite{shelah1988primitive} gave a new proof yielding primitive recursive bounds:

\begin{theorem}[Hales-Jewett, $\RCA_0$]\label[theorem]{thm:hales-jewett}
There exists a primitive recursive function~$HJ(k, \ell)$ such that for every $n \geq HJ(k, \ell)$, every alphabet~$A$ of size~$k$ and every coloring $f : A^n \to \ell$, there exists some 1-variable word $v$ over~$A$ of length~$n$ such that $\WSub^{=}(v) = \{ v[a] : a \in A \}$ is $f$-homogeneous.
\end{theorem}

For simplicity, given a finite alphabet~$A$ of size~$k$ and some~$\ell \geq 1$, we might simply write $HJ(A,\ell)$ for $HJ(\card A, \ell)$.

\subsection{Proof of $\OVW^0$}

We now turn to the proof of $\OVW^0$ over~$\ACA_0$. In what follows, fix an finite alphabet~$A$ of size~$k \geq 1$ and a number of colors~$\ell \geq 1$. Unless mentioned, the ordered variable words will be over~$A$.

\begin{definition}
Let $w$ be an ordered $n$-variable word, $W$ be an ordered $\omega$-variable word,
and $f : \WSub^{\star}(w \cdot W) \to \ell$ be a coloring.
\begin{itemize}
    \item[1.] $w$ \emph{half-matches} $W$ if for every~$u \in \WSub^{\star}(W)$, there is some~$v \in \OSub^{1,=}(w)$ such that $\WSub^{=}(v \cdot u)$ is $f$-monochromatic (in other words, there is some color~$c < \ell$ such that for every~$a \in A$, $f(v[a] \cdot u) = c$).
    \item[2.] $w$ \emph{full-matches} $W$ if for every~$u \in \WSub^{\star}(W)$, there is some~$v \in \OSub^{1,=}(w)$ such that $\WSub(v \cdot u)$ is $f$-monochromatic (in other words, for every~$a \in A$, $f(v[a] \cdot u) = f(v[\epsilon])$, where $v[\epsilon]$ is the initial segment of~$v$ up to the first occurrence of the variable).
\end{itemize}
\end{definition}

The idea behind full-matches is to be able, given an $\OVW^0$-solution within~$W$, to merge it with an element of~$\OSub^{1,=}(w)$ to obtain a larger $\OVW^0$-solution. It follows that if $w$ full-matches~$W$, then at least one element of $\OSub^{1,=}(w)$ can be extended into an infinite $\OVW^0$-solution. The following lemma will not be used in our proof of $\OVW^0$ as it already assumes that $\OVW^0$ holds, but is useful to give an intuition:

\begin{lemma}[$\RCA_0 + \OVW^0$]\label[lemma]{lem:full-match-is-extensible}
Let $w$ be an ordered $n$-variable word, $W$ be an ordered $\omega$-variable word, and $f : \WSub^{\star}(w \cdot W) \to \ell$ be a coloring.
Suppose that $w$ full-matches $W$, then there is some~$u \in \OSub^{1,=}(w)$ extensible into an infinite $\OVW^0$-solution to~$f$.
\end{lemma}
\begin{proof}
Let $g : A^{<\omega} \to \ell \times \card \OSub^{1,=}(w)$ be defined by $g(v) = \langle c, u\rangle$ where $u \in \OSub^{1,=}(w)$ is such that $\WSub(u \cdot W[v])$ is $f$-monochromatic for color~$c$.
By $\OVW^0$, there is some ordered $\omega$-variable word~$U$ such that $\WSub^{\star}(U)$ is $g$-monochromatic for some color~$\langle c, u\rangle$. By definition of~$g$, for every~$v \in A^{<\omega}$ and every~$a \in A$,
\begin{equation}\label{eq:full-match-eqn0}
    f(u[\epsilon]) = f(u[a] \cdot W[U[v]]) = c
\end{equation}

We claim that $u \cdot W[U]$ is an $\OVW^0$-solution to~$f$. Indeed, given~$v \in \WSub^{\star}(u \cdot W[U])$, we have two cases. Case 1: $|v| = 0$. Then, $f((u \cdot W[U])[v]) = f(u[\epsilon]) = c$ by \Cref{eq:full-match-eqn0}.
Case 2: $|v| > 1$. Say $v = a \cdot v'$ with $a \in A$. Then $f((u \cdot W[U])[v]) = f(u[a] \cdot W[U[v']]) = c$ still by \Cref{eq:full-match-eqn0}.
\end{proof}

Once proven the existence, for every coloring, of a full-match, then one can define an $\omega$-sequence of nested full-matches to produce a tree containing at least one path forming an infinite solution. Towsner proved the existence of full-matches over~$\ACA_0$ by defining an $\omega$-sequence of nested half-matches. The existence of half-matches itself was proven over~$\RCA$, an extension of $\RCA_0$ will the full induction scheme. However, the proof was non-uniform. Because of this, the $\omega$-sequence of nested half-matches was only shown to exist over~$\ACA_0$, and the $\omega$-sequence of nested full-matches needed $\ACA_0^+$. In our case, the existence of half-matches is an immediate consequence of the Hales-Jewett theorem, and yields a uniform bound. Indeed, Shelah~\cite{shelah1988primitive} proved that the function $HJ$ is primitive recursive, and therefore provably total over~$\RCA_0$.

\begin{lemma}[$\RCA_0$]
Fix~$A$ and $\ell \geq 1$.
Let $w$ be an ordered $HJ(A, \ell)$-variable word, $W$ be an ordered $\omega$-variable word,
and $f : \WSub^{\star}(w \cdot W) \to \ell$ be a coloring for some~$\ell \geq 1$.
Then $w$ half-matches $W$.
\end{lemma}
\begin{proof}
Fix some~$u \in \WSub^{\star}(W)$. Let $g : A^{HJ(A, \ell),0} \to \ell$ be defined by $g(v) = f(w[v] \cdot u)$.
By \Cref{thm:hales-jewett}, there is some 1-variable word~$v \in A^{HJ(A,\ell),1}_{<}$ such that $\WSub^{=}(v)$ is $g$-homogeneous for some color~$c$. In particular, $w[v] \in \OSub^{1,=}(w)$ is such that $\WSub^{=}(w[v] \cdot u)$ is $f$-homogeneous for color~$c$.
\end{proof}

Before proving the existence of a full-match over~$\WKL$, we define the notion of \emph{Graham-Rothschild tree} which corresponds to an $\omega$-sequence of nested half-matches, and will be of fundamental importance over the article.
Let $(\ell_n)_{n \in \NN}$ be inductively defined by
$$\ell_0 = \ell \mbox{ and } \ell_{n+1} = \ell_n \times \card A^{HJ(A, \ell_n),1}_<$$
Here, $\card A^{HJ(A, \ell_n),1}_<$ is the number of ordered 1-variable words of length $HJ(A, \ell_n)$ over~$A$. Then, let $(t_n)_{n \in \NN}$ be inductively defined by
$$t_0 = 0 \mbox{ and } t_{n+1} = t_n + HJ(A, \ell_n)$$

The \emph{Graham-Rothschild tree for $\ell$-colorings over alphabet~$A$} is the tree $\T_{A,\ell}$ whose level~$n$ consists of all ordered $n$-variable words~$w$ of length~$t_n$, such that for every~$s < n$, the variable~$x_s$ appears only within the positions $[t_n, t_{n+1})$. This tree is partially ordered by the prefix relation. Note that this tree is computably bounded, and every infinite path is an ordered $\omega$-variable word over~$A$. The fundamental property of this Graham-Rothschild tree is the following:

\begin{lemma}[$\RCA_0$]\label[lemma]{lem:graham-rothschild-tree}
Fix~$A$ and $\ell \geq 1$.
Let $f : A^{<\omega} \to \ell$ be a coloring.
For every~$u \in A^{<\omega}$ and $n \in \NN$, there is some ordered $n$-variable word~$w$ at level~$n$ in~$\T_{A,\ell}$ such that $\WSub^{=}(w \cdot u)$ is $f$-monochromatic.
\end{lemma}
\begin{proof}
For every~$n \in \NN$, let $f_n : A^{<\omega} \to \ell_n$ be defined inductively as follows:
First, $f_0 : A^{<\omega} \to \ell_0$ is simply $f$.
Having defined $f_n$, let $f_{n+1} : A^{<\omega} \to \ell_{n+1}$ be defined by $f_{n+1}(u) = \langle f_n(w[a] \cdot u), w \rangle$ for the least (in any fixed order) $w \in A^{HJ(A, \ell_n),1}_{<}$ such that $f_n$ is monochromatic on $\{ w[a] \cdot u : a \in A \}$. Note that such a $w$ exists by definition of $HJ(A, \ell_n)$.

By $\Pi^0_1$-induction, for any~$n$ and $u \in A^{<\omega}$, $f_n(u)$ can be written of the form $\langle c, w_0, \dots, w_{n-1} \rangle$, where $c < \ell$, $w_s \in A^{HJ(A, \ell_s),1}_{<}$, and
$$
    S^u_n = \{ w_0[a_0] \cdots w_{n-1}[a_{n-1}] \cdot u : a_0, \dots, a_{n-1} \in A \}
$$
is $f$-monochromatic for color~$c$. Let $v_n$ be the ordered $n$-variable word
$$w_0[x_0] \cdots w_{n-1}[x_{n-1}]$$
Note that $S^u_n = \WSub^{=}(v_n \cdot u)$ and $v_n$ is at level~$n$ in~$\T_{A,\ell}$.
\end{proof}

Note that by the uniformity of the notion of half-match for $\OVW^0$, the Graham-Rothschild tree does not depend on the coloring $f : A^{<\omega} \to \ell$. This has to be put in contrast with Towsner's proof of Hindman's theorem, where the nested sequence of half-matches depends on the coloring. However, we shall not use this uniformity in this article, except for the computability of the Graham-Rothschild tree, while its Hindman counterpart is computable in~$\emptyset''$.

\begin{remark}
The Graham-Rothschild tree~$\T_{A,\ell}$ actually satisfies the following stronger universality property, provable using only $\Delta^0_0$-induction on~$n$:
\begin{quote}
For every~$n \in \NN$ and every coloring $g : A^{t_n} \to \ell$, there is an ordered $n$-variable word~$w$ at level~$n$ in $\T_{A,\ell}$ such that $\WSub^=(w)$ is $g$-monochromatic.
\end{quote}
The statement of \Cref{lem:graham-rothschild-tree} is the restriction of this universality property to the sub-class $(f^n_u)_{n \in \NN, u \in A^{<\omega}}$ of \emph{shifts} of~$f$, defined for each $n \in \NN$ and $u \in A^{<\omega}$ by
$$
f^n_u : A^{t_n} \ni v \mapsto f(v \cdot u).
$$
However, to emphasize the general translation from Towsner-style proofs to notions of forcing with good first-jump control, we shall only use the statement of \Cref{lem:graham-rothschild-tree}.
\end{remark}

We now turn to the proof of the existence of full-matches over~$\WKL$, an extension of~$\WKL_0$ with the full induction scheme. Given a function $h : \NN \to \NN$, we write $h^{<\omega}$ for the set of finite words $w$ such that for every~$i < |w|$, $w(i) < h(i)$. A tree is \emph{bounded} if it is a sub-tree of $h^{<\omega}$ for some function $h : \NN \to \NN$. $\WKL_0$ proves K\"onig"s lemma for bounded infinite trees (see Simpson~\cite[Lemma IV.1.4]{simpson2009subsystems}). As in Towsner's proof, we need the following intermediary step:

\begin{lemma}[$\WKL_0$]\label[lemma]{lem:ovw0-towsner-full-step}
Fix~$A$ and $\ell \geq 1$.
Let $f : A^{<\omega} \to \ell$ be a coloring. One of the following holds:
\begin{itemize}
    \item[(1)] There is some~$n \in \NN$ such that $x_0\cdots x_n$ full-matches $x_{n+1}x_{n+2}\cdots$;
    \item[(2)] There is some ordered $\omega$-variable word~$W$ such that $\WSub^{\star}(W)$ $f$-avoids a color.
\end{itemize}
\end{lemma}
\begin{proof}
Suppose that (1) does not hold.
Let $\T_{A,\ell}$ be the Graham-Rothschild tree for $\ell$-colorings and alphabet~$A$.
\smallskip

\textbf{Claim 1}: For every $n \in \NN$, there is some~$w$ at level~$n$ in~$\T_{A,\ell}$ such that
$\WSub^{\star}(w)$ $f$-avoids some color. Because (1) does not hold, there is some~$u_0 \in A^{<\omega}$ such that for every~$v \in A^{t_n,1}_{<}$, $\WSub(v \cdot  u_0 )$ is not $f$-monochromatic. By \Cref{lem:graham-rothschild-tree} applied to~$f$ and $u_0$, there is some~$w$ at level~$n$ in~$\T_{A,\ell}$ such that
$\WSub^{=}(w \cdot u_0)$ is $f$-monochromatic for some color~$c$. Given $u \in \WSub^{\star}(w)$, there is some $v \in \OSub^{1,=}(w)$ such that $v[\epsilon] = u$. In particular, $v \in A^{t_n,1}_{<}$, so $\WSub(v \cdot u_0)$ is not $f$-monochromatic. So $c \neq f(v[\epsilon]) = f(u)$. So $\WSub^{\star}(w)$ $f$-avoids~$c$. This proves our Claim~1.

Let $\Sc$ be the sub-tree of~$\T_{A,\ell}$ of all~$w$ witnessing Claim~1.
By $\WKL_0$, there is an infinite path~$U$ through~$\Sc$. In particular, $U$ is an ordered $\omega$-variable word such that $\WSub^{\star}(U)$ $f$-avoids some color~$c$.
\end{proof}

\begin{lemma}[$\WKL$]\label[lemma]{lem:ovw0-towsner-full}
Fix~$A$ and $\ell \geq 1$.
Let $W$ be an ordered $\omega$-variable word and $f : \WSub^{\star}(W) \to \ell$ be a coloring.
There is some ordered $n$-variable word~$w_1$ for some~$n \geq 1$, and some ordered $\omega$-variable word~$W_1$ such that
$w_1 \cdot W_1 \in \OSub^\omega(W)$ and $w_1$ full-matches~$W_1$.
\end{lemma}
\begin{proof}
By $\Pi^1_2$-induction over~$\ell$.
Let $g : A^{<\omega} \to \ell$ be defined by $g(u) = f(W[u])$.
By \Cref{lem:ovw0-towsner-full-step}, there are two cases:
Case 1: there is some~$n \in \NN$ such that $x_0\cdots x_n$ full-matches $x_{n+1}x_{n+2}\cdots$ for~$g$. Let $w_1 = W[x_0\dots x_n]$ and $W_1$ be such that $W = w_1 \cdot W_1$. Then $w_1$ full-matches $W_1$ for~$f$.
Case 2: there is some ordered $\omega$-variable word~$U$ such that $\WSub^{\star}(U)$ $g$-avoids a color. Unfolding the definition, $\WSub^{\star}(W[U])$ $f$-avoids a color. By induction hypothesis, there is some ordered $n$-variable word~$w_1$ for some~$n \geq 1$ and some ordered $\omega$-variable word $W_1$ such that $w_1 \cdot W_1 \in \OSub^\omega(W[U])$ and $w_1$ full-matches~$W_1$.
\end{proof}

Note that \Cref{lem:ovw0-towsner-full} is proven over $\WKL_0$ augmented with the full induction scheme. The proof shows that actually, $\Pi^1_2$-induction suffices. The following conservation theorems states that not only $\WKL$ is $\Pi^1_2$-conservative over~$\ACA_0$, but one can even bound the computational complexity of the existential witness of the statement. It appears to be part of the folklore.

\begin{proposition}\label[proposition]{prop:wkl-aca0-conservation}
Let $\Phi(X, Y)$ be an arithmetic statement.
Suppose that $\WKL \vdash \forall X \exists Y \Phi(X,Y)$,
then $\ACA_0 \vdash \forall X \exists Y \mbox{ low over } X \mbox{ such that } \Phi(X,Y)$.
\end{proposition}
\begin{proof}
Fix~$\Phi(X,Y)$. By the compactness theorem, since $\WKL \vdash \forall X \exists Y \Phi(X,Y)$, then there is a finite set of axioms~$T$ of $\WKL$ such that $T \vdash \forall X \exists Y \Phi(X,Y)$.
Let $\M = (M, S)$ be a model of~$\ACA_0$, and $X \in S$. By Simpson~\cite[Chapter VIII]{simpson2009subsystems} and the formalized low basis theorem, there exists a set $Z = \bigoplus_n Z_n$ with $Z' \leq_T X'$, such that $\Nc = (M, \{ Z_n : n \in M \}) \models \WKL_0$ and $X = Z_0$. Note that any second-order formula over~$\Nc$ can be transformed into a first-order formula over~$\M$, so since $\ACA_0$ contains arithmetic induction, letting $\Gamma$ be the finite fragment of formulas such that $T \subseteq \WKL_0 + \mathsf{I}\Gamma$, we have $\Nc \models \WKL_0 + \mathsf{I}\Gamma$. It follows that $\Nc \models \exists Y \Phi(X,Y)$, so there exists some~$n \in Z_n$ such that $\Nc \models \Phi(X, Z_n)$. Since $\Phi$ is arithmetic, then $\M \models \Phi(X,Z_n)$. Note that $(X \oplus Z_n)' \leq_T Z' \leq_T X'$.
\end{proof}

Combining \Cref{prop:wkl-aca0-conservation} with \Cref{lem:ovw0-towsner-full}, we obtain:

\begin{lemma}[$\ACA_0$]\label[lemma]{lem:ovw0-towsner-full-low}
Fix~$A$ and $\ell \geq 1$.
Let $W$ be an ordered $\omega$-variable word and $f : \WSub^{\star}(W) \to \ell$ be a coloring.
There is some ordered $n$-variable word~$w_1$ for some~$n \geq 1$, and some ordered $\omega$-variable word~$W_1$ such that
$w_1 \cdot W_1 \in \OSub^\omega(W)$, $w_1$ full-matches~$W_1$ and $(W_1 \oplus W \oplus f)' \leq_T (W \oplus f)'$.
\end{lemma}

We are now ready to prove $\OVW^0$ over~$\ACA_0$.

\begin{theorem}[$\ACA_0$]
Fix~$A$ and $\ell \geq 1$.
For every coloring $f : A^{<\omega} \to \ell$, there is an ordered $\omega$-variable word~$W$ over~$A$
such that $\WSub^{\star}(W)$ is $f$-monochromatic.
\end{theorem}
\begin{proof}
We will define a sequence of tuples $(\ell_n, f_n, w_{n+1}, W_n)_{n \in \omega}$ such that
\begin{itemize}
    \item[(1)] $f_n : \WSub^{\star}(W_n) \to \ell_n$;
    \item[(2)] $w_{n+1}$ full-matches $W_{n+1}$ for $f_n$;
    \item[(3)] $w_{n+1} \cdot W_{n+1} \in \OSub^\omega(W_n)$;
    \item[(4)] $(\bigoplus_{s \leq n} f_n \oplus W_n)' \leq_T f'$.
\end{itemize}
First, $\ell_0 = \ell$, $f_0 = f$ and $W_0 = x_0x_1\cdots$.
Assume we have defined~$\ell_n, f_n$ and $W_n$ satisfying items (1) and (4).
By \Cref{lem:ovw0-towsner-full-low}, there is some pair $(w_{n+1}, W_{n+1})$ satisfying items (2-4).
Let $\ell_{n+1} = \ell_n \times \card \OSub^{1,=}(w_{n+1})$ and let $f_{n+1} : \WSub^{\star}(W_{n+1}) \to \ell_{n+1}$ be defined by $f_{n+1}(u) = \langle c, v \rangle$ for some~$v \in \OSub^{1,=}(w_{n+1})$ such that $\WSub(v \cdot u)$ is $f_n$-monochromatic for some color~$c < \ell_n$.

By $\Pi^0_1$-induction, for any~$n \in \NN$ and $u \in A^{<\omega}$, $f_n(u)$ can be written in the form $\langle c, v_0, \dots, v_{n-1} \rangle$, where $c < \ell$, $v_s \in \OSub^{1,=}(w_s)$, and
$\WSub(v_0 \cdots v_{s-1} \cdot u)$ is $f$-monochromatic for color~$c$.

Let $\Sc$ be the tree whose level~$n$ consists of all ordered $n$-variable words~$w \in \OSub^{n,=}(w_1 \cdots w_n)$ such that $\WSub^{\star}(w)$ is $f$-monochromatic. Note that $\Sc$ is $f'$-computably bounded. Moreover, $\Sc$ is infinite. By $\WKL_0$, there is an infinite path $U$ through~$\Sc$. Such a path~$U$ is an ordered $\omega$-variable word such that $\WSub^{\star}(U)$ is $f$-monochromatic.
\end{proof}

\section{Weakly low basis theorem for $\msf{OVW}$}\label[section]{sec:weakly-low}

The levels of the arithmetic hierarchy are not closed under relativization, in the sense that for $n \geq 2$, the relation\qt{$X$ is $\Delta^0_n(Y)$} is not transitive. It is therefore useful to work with notions of lowness.
A set~$X$ is of \emph{low${}_n$ degree} over~$Y$ if $(X \oplus Y)^{(n)} \leq_T Y^{(n)}$. In particular, the relation \qt{$X$ is of low${}_n$ degree over $Y$} is transitive. When $n = 1$, we simply say low for low${}_1$.

\begin{definition}
A $\Pi^1_2$-problem~$\Psf$ admits a \emph{low${}_n$ basis} if for every set~$Z$ and every $Z$-computable $\Psf$-instance~$X$,
there is a $\Psf$-solution~$Y$ to $X$ such that $(Y \oplus Z)^{(n)} \leq Z^{(n)}$.
\end{definition}

For instance, Jockusch and Soare~\cite{jockusch1972pi} proved that $\WKL$ admits a low basis. Some $\Pi^1_2$-problems are \qt{close} to admit a low basis, in that any PA degree over~$\emptyset'$ can decide the jump of a solution. We then say that it admits a \emph{weakly low basis}.

\begin{definition}
A $\Pi^1_2$-problem~$\Psf$ admits a \emph{weakly low${}_n$ basis} if for every set~$Z$, every set~$P$ of PA degree over~$Z^{(n)}$ and every $Z$-computable $\Psf$-instance~$X$,
there is a $\Psf$-solution~$Y$ to $X$ such that $(Y \oplus Z)^{(n)} \leq P$.
\end{definition}

For any~$Z$, by the low basis theorem~\cite{jockusch1972pi} relative to $Z^{(n)}$, there is a set~$P$ of PA degree over $Z^{(n)}$ such that $P' \leq_T Z^{(n+1)}$. It follows that if $(Y \oplus Z)^{(n)} \leq P$, then $(Y \oplus Z)^{(n+1)} \leq_T P' \leq_T Z^{(n+1)}$, so if $\Psf$ admits a weakly low${}_n$ basis, then it admits a low${}_{n+1}$ basis.
The most famous theorem admitting a weakly low basis~\cite{cholak2001strength}, but no low basis~\cite{jockusch1972ramsey}, is Ramsey's theorem for pairs.

Liu, Monin and Patey~\cite[Corollary 4.5]{liu2019computable} constructed a computable instance of~$\OVW^0$ with no $\Delta^0_2$ solutions, so \emph{a fortiori}, $\OVW^0$ does not admit a low basis.
The goal of this section is to prove the following theorem:

\begin{theorem}\label[theorem]{thm:ovw0-weakly-low}
$\OVW^0$ admits a weakly low basis.
\end{theorem}
As mentioned, if a $\Pi^1_2$-problem admits a weakly low basis, then its admits a low${}_2$ basis.

\begin{corollary}
 $\OVW^0$ admits a low${}_2$ basis theorem.
\end{corollary}
\begin{proof}
Immediate by \Cref{thm:ovw0-weakly-low} and by the low basis theorem~\cite{jockusch1972pi}.
\end{proof}
Last, note that the construction only uses arithmetical induction, and can be entirely formalized within $\ACA_0$:

\begin{theorem}[$\ACA_0$]\label[theorem]{ovw0-low2-formalized}
Fix a finite alphabet~$A$ and $\ell \geq 1$.
For every set~$Z$ and every $Z$-computable coloring $f : A^{<\omega} \to \ell$,
there is an ordered $\omega$-variable word~$W$ such that $\WSub^{\star}(W)$ is $f$-monochromatic,
and $(W \oplus Z)'' \leq_T Z''$.
\end{theorem}

The rest of this section is dedicated to the proof of \Cref{thm:ovw0-weakly-low}. 
For simplicity, instead of constructing an $\OVW^0$-solution, we shall work with a weakening of the Ordered Variable Word theorem where only infinitely levels are required to be monochromatic.
 
\begin{definition}
Let $\IOOVW^0$ be the $\Pi^1_2$-problem whose instances are colorings $f : A^{<\omega} \to \ell$ for a finite alphabet~$A$ and some number of colors~$\ell \geq 1$. A \emph{$\IOOVW^0$-solution} to~$f$ is an ordered $\omega$-variable word~$W$ such that for infinitely many $n\in\omega$, $W[A^n] = \{ W[u] : u \in A^n \}$ is monochromatic for $f$.
\end{definition}

This weaker variant is actually computably equivalent to $\OVW^0$, and this provably over~$\RCA_0 + \BSig_2$:

\begin{lemma}[$\RCA_0 + \BSig_2$]
Let $f : A^{<\omega} \to \ell$ be a coloring.
For every $\IOOVW^0$-solution~$W$ to~$f$, $f \oplus W$ computes an $\OVW^0$-solution to~$f$.
\end{lemma}
\begin{proof}
By the infinite pigeonhole principle (equivalent to $\BSig_2$ over~$\RCA_0$),
there is a color~$c < \ell$ such that the $f\oplus W$-computable set 
$$
L_c = \{ n \in \NN : W[A^n] \mbox{ is } f\mbox{-monochromatic for color } c \} 
$$
is infinite. Say $L_c = \{ n_0 < n_1 < \dots \}$.
Let $\hat W$ be obtained from~$W$ by merging the variable kinds $\{ x_m : n_i \leq m < n_{i+1} \}$ into one variable kind, for each~$i \in \NN$. Then, $\hat W \in \OSub^\omega(W)$, and for every~$u \in \WSub^\star(\hat W)$, $|u| \in L_c$, hence $f(u) = c$. It follows that $\WSub^\star(\hat W)$ is $f$-monochromatic for color~$c$, so $\hat W$ is an $\OVW^0$-solution to~$f$.
\end{proof}

For notation convenience, we shall prove \Cref{thm:ovw0-weakly-low} in an unrelativized form.
Throughout this section, fix a computable coloring $f:A^{<\omega}\rightarrow\ell$
and a set $P$ of PA degree over~$\emptyset'$.
We will construct an $\IOOVW^0$-solution~$G$ such that $G' \leq_T P$ using an effectivization of a notion of forcing.

The idea is to construct an infinite $\emptyset'$-computable sequence of conditions
whose upward-closure forms a 1-generic filter. The set of stems of this decreasing sequence will form an infinite $\emptyset'$-computable, $\emptyset'$-computably bounded tree~$T$, and any PA degree~$P$ over~$\emptyset'$ will compute an infinite path~$U$. By 1-genericity of the filter, any path through~$T$ will be a generalized low $\IOOVW^0$-solution to~$f$, so $P \geq_T U'$.



\subsection{Notion of forcing}\label[section]{sec:forcing-notion}

We now use the combinatorics of \Cref{sec:towsner-ovw0} to design a parametric notion of forcing to build solutions to instances of~$\IOOVW^0$, with a good first-jump control.

\begin{definition}\label[definition]{caOVW-def-precondition}
A \emph{$\PP$-branch} is a pair $(v,W)$ where $v$ is a finite ordered variable word and
$W$ an ordered $\omega$-variable word.
\end{definition}

One can think of a $\PP$-branch as a Mathias-like condition, where the \emph{stem} $v$ is an initial segment of the ordered variable word we are building, and the \emph{reservoir} $W$ puts some constraints of the future extensions of the stem. We can associate to a $\PP$-branch~$(v, W)$ its \emph{denotation} $[v, W]$, representing the class of candidates of the solution~$G$ we are constructing, namely, 
\begin{align}\label{caOVW-eq14}
[v, W] = \big\{
G\in \vw(v\cdot W) :\ & v\preceq G
\big\}
\end{align}
Note that, although an $\IOOVW^0$-solution to~$f$ needs to contain infinitely many variable kinds, and have infinitely many levels of $f$-monochromaticity, this is not necessarily the case for every~$G \in [v, W]$. In particular, some~$G \in [v, W]$ will only contain the variable kinds of~$v$, or will have no level of $f$-monochromaticity. We will therefore need to choose an appropriate $G \in [v, W]$ by satisfying some positive requirements infinitely often (see Lemma \ref{caOVW-lem-positiveext}).
We also implicitly require that the variable kinds in $G\uhr_{|v|}^{|G|-1}$ and $v$ are disjoint.

We now define the notion of $\PP$-branch extension, such that $(\h v,\h W)$ extends $(v,W)$ if $(\h v,\h W)$ is \qt{more precise} than $(v,W)$. In particular, it will satisfy that if $(\h v,\h W)$ extends $(v,W)$, then
$[\h v,\h W] \subseteq [v,W]$.

\begin{definition}[$\mathbb P$-branch extension]\label[definition]{caOVW-def-preconditionextension}
A $\PP$-branch\ $(\h v,\h W)$ \emph{extends}
a $\PP$-branch\ $(v,W)$
(written $(\h v,\h W)\leq (v,W)$)
iff:
$\h v\succeq v$
and $\h v\cdot \h W\in \OSub^\omega(v\cdot W)$.
Or equivalently, for some ordered $\omega$-variable word
 $U\in \OSub^\omega(W)$ and some $t\in\omega$,
 $\h v = v\cdot (U\uhr t)$ and $\h W = U\uhr_t^\infty$
 where $t$ must be in between variables of $U$.

\end{definition}
Clearly the extension relation is transitive.
We emphasize that it is not necessary the case
 that $\h v \in \vwshorter(v\cdot W)$ (therefore not necessary
 $\h v\in [v,W]$)
 since the truncating point $t$ is somewhat more flexible.
 But $\h v$ must be a prefix of some member in $[v,W]$.

Last, note that a $\PP$-branch\ may not
even contain a \wOVW\ solution to $f$.
Therefore, we work with a finite set of $\PP$-branches satisfying some universality constraint.

\begin{definition}[Forcing condition]\label[definition]{caOVW-def-condition}
 A \emph{$\mathbb P$-precondition}
is a pair $(I,W)$ where
\begin{itemize}
\item $I$ is a finite set of ordered variable word
that is length-homogeneous. i.e., For some $t\in\omega$ (denoted as $t(I)$)
 $|v| = t$ for all $v\in I$;
\item For every $v\in I$, $(v,W)$
is a $\PP$-branch.
\end{itemize}
A $\PP$-precondition $(I, W)$ is \emph{$f$-matching} if for every word $w\in \wshorter(W)$,
there exists $v\in I$
such that $\wsamelong(v\cdot w)$ is monochromatic for $f$.
A \emph{$\PP$-condition} is an $f$-matching $\PP$-precondition.
\end{definition}

One can think of a $\PP$-precondition as a finite set of $\PP$-branches sharing the same reservoir. Accordingly, any $\PP$-precondition $(I, W)$ is given a denotation $[I,W] = \cup_{v\in I} [v,W]$

\begin{definition}[$\PP$-precondition extension]
A $\PP$-precondition $(\h I,\h W)$ \emph{extends} $\PP$-precondition
$(I,W)$ iff for every $\h v\in \h I$,
there exists a $v\in I$ such that
$(\h v,\h W)\leq (v,W)$.
\end{definition}

Intuitively, $f$-matching is a universality constraint which ensures that every $\PP$-condition contains an \wOVW-solution to $f$. To support this intuition, as we did in \Cref{lem:full-match-is-extensible}, we shall assume $\OVW^0$, and prove the following fact:

\begin{lemma}[$\RCA_0 + \OVW^0$]\label[lemma]{rem:ovw0-matching-has-solutions}
Let $(I, W)$ be a $\PP$-condition. Then $[I,W]$ contains an \wOVW-solution to $f$.
\end{lemma}
\begin{proof}
Since $(I, W)$ is $f$-matching, let $\bar f : A^{<\omega} \to \ell \times I$ be defined as follows:
\begin{align}\label{eq:ovw-matching-sol-eq0}
\bar f: &\ A^{<\omega}\ni w \mapsto \langle c, v\rangle
\text{ such that }\\ \nonumber
&\text{$\wsamelong(v\cdot W[w])$ is $f$-monochromatic for color~$c$}.
\end{align}
Let $U$ be an $\OVW^0$-solution to $\bar f$, that is, $U$ is an ordered $\omega$-variable word such that 
\begin{align}\label{eq:ovw-matching-sol-eq1}
\text{$\WSub^\star(U)$ is monochromatic for some color~$\langle c, v\rangle \in \ell \times I$.}
\end{align}
Let $\hat W = v \cdot W[U]$. Note that $\hat W \in [v, W] \subseteq [I,W]$.
We claim that $\hat W$ is an $\IOOVW^0$-solution to~$f$. We shall actually prove the following stronger property:

\begin{claim}
For every $w \in \WSub^{\star}(\hat W)$ with $|w| > |v|$, $f(w) = c$.
\end{claim}
\begin{proof}
Since $w \in \WSub^{\star}(\hat W) = \WSub^{\star}(v \cdot W[U])$,
there is some~$\hat v \in \WSub^=(v)$ and $\hat w \in \WSub^{\star}(U)$ such that
\begin{align}\label{eq:ovw-matching-sol-eq2}
w = \hat v \cdot W[\hat w]
\end{align}
By (\ref{eq:ovw-matching-sol-eq1}), $\bar f(\hat w) = \langle c, v\rangle$,
so by (\ref{eq:ovw-matching-sol-eq0}), $\WSub^=(v \cdot W[\hat w])$ is $f$-monochromatic for color~$c$.
Since $w \overset{(\ref{eq:ovw-matching-sol-eq2})}= \hat v \cdot W[\hat w] \in \WSub^=(v \cdot W[\hat w])$, we have
$f(w) = c$. This completes our claim and the proof of \Cref{rem:ovw0-matching-has-solutions}.
\end{proof}

\end{proof}

\subsection{Effective forcing}

It is common in effective forcing to work with a restricted set of $\PP$-conditions with computability-theoretic constraints. In our situation, we shall restrict ourselves to $\PP$-preconditions $(I, W)$ such that $W$ belongs to a non-empty class of sets~$\C$ satisfying the following closure properties:
\begin{itemize}
    \item[(1)] $\C$ is downward-closed under Turing reducibility: $\forall X \in \C \forall Y \leq_T X, Y \in \C$;
    \item[(2)] $\C$ is closed under PA degrees: $\forall X \in \C \exists Y \in \C, Y $ is of PA degree over~$X$.
\end{itemize}
Any such class is called a \emph{Scott class}. If one furthermore requires that $\C$ is closed under the effective join, then we obtain the well-known notion of Scott ideal.
Any basis theorem for $\Pi^0_1$ classes induces a Scott class. For instance, by the computably dominated and the low basis theorems~\cite{jockusch1972pi}, the classes of all sets of computably dominated degree or of all low degree form Scott classes, respectively. None of the above examples form Scott ideals.

\begin{definition}
Given a Scott class~$\C$, a $\PP^\C$-(pre)condition is a $\PP$-(pre)condition $(I, W)$ such that $W \in \C$.
\end{definition}

To effectivize the construction of a filter, we shall be particularly interested in the class $\mathsf{LOW}$ of all sets of low degree. Any $\PP^{\mathsf{LOW}}$-precondition can be represented as a finite object as follows:

\begin{definition}
A \emph{lowness index} of a set~$X$ of low degree is an integer~$e$ such that $\Phi_e^{\emptyset'} = X'$.
A \emph{code} for a $\PP^{\mathsf{LOW}}$-precondition $(I, W)$ is a pair $\langle I, e\rangle$ where~$e$ is a lowness index of~$W$.
\end{definition}

Note that $\langle I, e\rangle$ is a finitary object, hence can be coded by an integer using any reasonable G\"odel numbering. We shall prove the core lemmas of the notion of $\PP$-forcing in the general case, but one should keep in mind that they will all hold even when restricting $\PP^\C$ for any fixed Scott class. Furthermore, when proving a density lemma, we will always specify the amount of computational power needed to transform a code of a $\PP^{\mathsf{LOW}}$-precondition into a code of the desired $\PP^{\mathsf{LOW}}$-extension.

\subsection{First-jump control}\label[section]{sec:ovw0-forcing-first-jump}

Given a notion of forcing, there exists a general forcing relation defined as follows: $p$ \emph{forces} a property $\varphi(G)$ if $\varphi(G_\F)$ holds for every sufficiently generic filter~$\F$ containing~$p$, where $G_\F$ is the set induced by the filter~$\F$. Although this forcing relation also admits a simple inductive and syntactical definition, the complexity of the forcing relation often remains too high with respect to the complexity of the forced formula.

For our purpose, it will be sufficient to force $\Sigma^0_1$ and $\Pi^0_1$ formulas. At this level of the arithmetic hierarchy, there is often a stronger forcing relation which does not only hold for any sufficiently generic filter, but for \emph{every} filter containing~$p$. We now give a formal definition of a forcing relation for $\Sigma^0_1$ and $\Pi^0_1$-formulas.

By Kleene's normal form theorem, every $\Sigma^0_1$-formula $\varphi(G)$ can be written of the form $\exists x \psi(G \uh_x)$, where $\psi(\sigma)$ is a \emph{monotonous} $\Delta^0_0$-formula, that is, if $\psi(\sigma)$ holds and $\sigma \preceq \tau$, then $\psi(\tau)$ also holds. We will always assume that a $\Sigma^0_1$-formula is in normal form.

\begin{definition}[Forcing relation]\label[definition]{caOVW-def-forcing}
Let $(v,W)$ be a $\PP$-branch and $\varphi(G):= \exists n\psi(G\uhr n)$ be a $\Sigma^0_1$-formula. Let
\begin{enumerate}[label=(\arabic*)]
\item $(v,W)\Vdash \varphi(G)$
iff $\psi(v\uhr n)$ holds for some $n$;

\item $(v,W)\Vdash \neg\varphi(G)$ iff
for every $\h v\in [v,W]$, every $n\in\omega$, $\neg\psi(\h v\uhr n)$ holds.
\end{enumerate}

We also say $(v, W)$ \emph{forces} $\varphi(G)$ for $(v, W) \Vdash \varphi(G)$.
A $\PP$-precondition $(I,W)$ \emph{decides} a $\Sigma^0_1$-formula $\varphi(G)$ iff
for every $v\in I$, either $(v,W)\Vdash\varphi(G)$ or $(v,W)\Vdash\neg\varphi(G)$.
\end{definition}

Suppose $(v,W)\Vdash \varphi(G)$ for some $\Sigma^0_1$ or $\Pi^0_1$-formula $\varphi(G)$, then $\varphi(G)$ holds for every $G\in [v,W]$.
It is also clear that if $(v,W) \Vdash \varphi(G)$, then for every $(\h v,\h W) \leq (v, W)$, we also have $(\h v, \h W) \Vdash \varphi(G)$. Indeed, $[\h v,\h W]\subseteq [v,W]$.

A $\PP$-filter~$\F$ is \emph{1-generic} if for every $\Sigma^0_1$-formula $\varphi(G)$, there is a condition~$p \in \F$ deciding~$\varphi(G)$. It is not clear at first sight that every sufficiently generic $\PP$-filter is 1-generic. We shall however see that this is the case. For this, we define the so-called \emph{forcing question} $\qvdash$ for $\Sigma^0_1$-formulas, which is a central concept in effective forcing.

In general, a forcing question is a relation $p \qvdash \varphi(G)$ between a condition~$p$ and a formula~$\varphi(G)$, such that if $p \qvdash \varphi(G)$, then there is an extension~$q \leq p$ forcing $\varphi(G)$, and if $p \nqvdash \varphi(G)$, then there is an extension~$q \leq p$ forcing $\neg \varphi(G)$. The computability-theoretic features of a generic set are closely related to the existence of a forcing question with appropriate definability and combinatorial properties~\cite{patey2025lowness}. In our case, since we work with a tree-like notion of forcing, we need to adapt the definition.

Recall that $\mcal T_{A, \ell \times I}$ is the \GRtree\
for $\ell \times I$-colorings over alphabet~$A$ (see \Cref{sec:towsner-ovw0}).

\begin{definition}[Forcing question]\label[definition]{caOVW-def-forcingquestion}
Let $(I,W)$ be a $\PP$-precondition and $v\in I$.
We write $(I,W)\qvdash_v\varphi(G)$
iff there exists a level~$n \in \NN$ such that for every $u$ at level~$n$ in~$\T_{A,\ell\times I}$,
 there exists some $\hat u\in \OSub^=(u)$ such that
$\varphi(v\cdot W[\hat u])$ holds.
\end{definition}

The idea is that if $(I,W)\qvdash_v\varphi(G)$, then there is an extension $(\hat I, \hat W)$ such that $(\hat v, \hat W) \Vdash \varphi(G)$ for all the stems $\hat v \in I$ extending~$v$. If on the other hand, $(I,W)\nqvdash_v \varphi(G)$, then there is an extension $(I, \hat W)$ with the same set of stems~$I$, such that $(v, \hat W) \Vdash \neg \varphi(G)$. The issue is that in the positive case, $\hat I$ might increase the number of stems incompatible with~$v$, and therefore one might never obtain an extension deciding $\varphi(G)$ on all its stems. We shall therefore force $\neg \varphi(G)$ on as many stems as possible, and then force $\varphi(G)$ on all the remaining stems at once.

The easiest case to consider is a negative answer:

\begin{lemma}[$\Pi_1^0$-Extension]\label[lemma]{caOVW-lemextPi1}
Let $(I,W)$ be a $\PP$-condition
and $v\in I$.
If $(I,W)\nqvdash_v \varphi(G)$,
then there exists a $\PP$-condition
$(I,\h W)\leq (I,W)$ such that
$(v,\h W)\Vdash\neg\varphi(G)$.
Moreover, we can choose $\h W$ so that $\h W'\leq_T W'$
and the index of $\h W'\leq_T W'$ can be obtained $W'$-computably.
\end{lemma}
\begin{proof}
By definition of $(I,W)\nqvdash_v\varphi(G)$, for every
level~$n \in \NN$ there is some $u$ at level~$n$ in~$\T_{A, \ell \times I}$,
such that for every~$\hat u \in \OSub^=(u)$, $\varphi(v \cdot W[\hat u])$ does not hold.
Let $Q$ be the $\Pi^{0,W}_1$ class of all~$U \in [\T_{A, \ell \times I}]$ such that
\begin{align}\label{caOVW-eq2}
&\text{for every $u\in\vwshorter(U) $,
we have $\neg\varphi(v\cdot W[u])$}.
\end{align}
By compactness, the class~$Q$ is non-empty.
By construction of the \GRtree, every $U\in Q$ is an ordered $\omega$-variable word.
By the low basis theorem \cite{jockusch1972pi},
there is some $U\in Q$ be such that $(W\oplus U)'\leq_T W'$.
Clearly, $$(I,W[U])\leq (I,W)$$ is a $\PP$-precondition. Furthermore, it is $f$-matching since only the reservoir is restricted, so $(I,W[U])$ is a $\PP$-condition.
It remains to show that $(v,W[U])\Vdash\neg\varphi(G)$,
which follows directly from (\ref{caOVW-eq2}).
Thus, we are done.

\end{proof}

Repeatedly applying Lemma \ref{caOVW-lemextPi1}, we obtain:

\begin{lemma}[$\Pi_1^0$-Extension rephrased]\label[lemma]{caOVW-lemextPi1-2}
Let $(I,W)$ be a $\PP$-condition.
There exists a $\PP$-condition $(I,U)\leq (I,W)$ and a $J\subseteq I$
such that
\begin{enumerate}[label=(\arabic*)]
\item  For every $v\in I\setminus J$,
$(v,U)\Vdash\neg\varphi(G)$.
\item For every $v\in J$,
$(I,U)\qvdash_v \varphi(G)$.
\end{enumerate}
Moreover, we can choose $U$ so that $U'\leq_T W'$
and the index of  $U'\leq_T W'$ and $J$ can be obtained $W'$-computably.
\end{lemma}
\begin{proof}
This is done by applying \Cref{caOVW-lemextPi1} successively on each~$v \in I$ for which the forcing question has a negative answer, until we achieve a $\PP$-condition~$(I, U) \leq (I, W)$ such that $(I, U) \qvdash_v \varphi(G)$ on every $v \in I$ such that $(v, U) \not \Vdash \neg \varphi(G)$. Let $J$ be the set of such stems.
\end{proof}

Now comes the hard case, namely forcing $\varphi(G)$:

\begin{lemma}[$\Sigma_1^0$-Extension]\label[lemma]{caOVW-lemextSigma1}
Let $(I,W)$ be a $\PP$-condition and $J\subseteq I$.
Suppose for every $v\in J$,
$(I,W)\qvdash_v \varphi$.
Then there exists a $\PP$-condition $(\h I,\h W)\leq (I,W)$
such that for every $\h v\in \h I$ with $\h v$ extending
some $v\in J$, we have
$(\h v,\h W)\Vdash\varphi$.
Moreover, we can choose $(\h I,\h W)$
so that $\h W = W\uhr_{t(\h I)-t(I)}^\infty$  and the index of  $\h I $ can be obtained $W'$-computably.
\end{lemma}
\begin{proof}
 The main difficulty is to preserve
 the $f$-matching property while extending stems in $I$.
 This is possible since each $v\in J$ is extended
 to sufficiently many $\h v$. Indeed, we concatenate each such $v$ with some
 ordered variable word generated by nodes in some level of $\mcal T_{A, \ell \times I}$
 for \emph{every} node in that level.
 For the remaining stems $v\in I\setminus J$, we extend $v$ to some
 $v\cdot w$ where $w$ ranges over a set of words.
 The concrete proof is as follows.

Unfolding the definition of
$(I,W)\qvdash_v\varphi(G)$, there is a level $n \in \NN$
such that
for every $v\in J$,
every $u$ in level $n$ of  $\mcal T_{A, \ell \times I}$,
\begin{align}
\label{caOVW-eq16}
&\text{there exists some $\h u\in \vwsamelong(u) $
}\\ \nonumber
&\text{such that
$\varphi(v\cdot W[\h u])$ holds.}
\end{align}

In this proof,
let $witness(v)$ denote the set of all  $\h u$ satisfying (\ref{caOVW-eq16}).
Recall that by construction of the Graham-Rothschild tree, all the nodes at level $n$ in $\mcal T_{A, \ell \times I}$
are of length $t_n$.
Let
\begin{align}\label{caOVW-eq17}
\h I & := \left\{ v\cdot W[\hat u] : \begin{array}{l}
    \text{ either }v\in J\text{ and } \h u\in witness(v)\\
  \text{ or } v\in I\setminus J\text{ and }\hat u\in A^{t_n}
  \end{array} \right\} \\
\h W & :=  W \uhr_{t(\h I)-t(I)}^\omega. \nonumber
\end{align}
We show that $(\h I,\h W)$ is the desired extension.
It is obviously a $\PP$-precondition.
By definition of $\h I$ (see also (\ref{caOVW-eq16})),
for every $\h v\in \h I$ with $\h v$ extending
some member of $J$, $\varphi(\h v)
$ holds (so $(\h v,\h W)\Vdash\varphi(G)$).
Thus, it remains to show that
\begin{claim}\label[claim]{caOVW-claim4}
$(\h I,\h W)$ is $f$-\sufficient.
\end{claim}
\begin{proof}

Let $$w\in \wshorter(W\uhr_{t(\h I)-t(I)}^\infty).$$
It suffices to find $\h v\in \h I$ such that
\begin{align}\label{caOVW-eq20}
\wsamelong(\h v\cdot w)\text{ is monochromatic for $f$.}
\end{align}
Since $(I, W)$ is $f$-matching, let $\bar f : A^{<\omega} \to \ell \times I$ be defined as follows:
\begin{align}\label{newcaOVW-eq1}
\bar f: &\ A^{<\omega}\ni w \mapsto \langle c, v\rangle
\text{ such that }\\ \nonumber
&\text{$\wsamelong(v\cdot W[w])$ is $f$-monochromatic for color~$c$}.
\end{align}
Let $\hat w \in A^{<\omega}$ be such that
$$w = W\uhr_{t(\h I)-t(I)}^\omega[\hat w].$$
By the fundamental property of the Graham-Rothschild tree $\T_{A,\ell \times I}$ (\Cref{lem:graham-rothschild-tree}) applied to the coloring~$\bar f$ and the word $\hat w$, there is some~$u$ at level~$n$ in $\T_{A,\ell \times I}$ such that
\begin{align}\label{newcaOVW-eq2}
\text{$\WSub^=( u \cdot \hat w)$ is $\bar f$-monochromatic for some color $\langle c, v\rangle \in \ell \times I$} 
\end{align}
By (\ref{newcaOVW-eq1}),
\begin{align}\label{newcaOVW-eq3}
\text{$\WSub^=(v \cdot W[u \cdot \hat w])$ is $f$-monochromatic for color~$c$} 
\end{align}
Note that
$$
W[u \cdot \hat w] = W[u] \cdot W\uhr_{t(\h I)-t(I)}^\omega[\hat w] = W[u] \cdot w
$$
So by (\ref{newcaOVW-eq3}), we have
\begin{align}\label{newcaOVW-eq4}
\text{$\WSub^=(v \cdot W[u] \cdot w)$ is $f$-monochromatic for color~$c$} 
\end{align}
Suppose first that $v \in I \setminus J$. In particular, letting $\hat u \in \SubW^=(u)$,
we have in particular $\hat u \in A^{t_n}$, so by (\ref{caOVW-eq17}), $v \cdot W[\hat u] \in \hat I$. Moreover, by (\ref{newcaOVW-eq4}), we have
$$
\text{$\WSub^=(v \cdot W[\hat u] \cdot w)$ is $f$-monochromatic for color~$c$.} 
$$
Suppose now that $v \in J$. By (\ref{caOVW-eq16}), there is some $\hat u \in witness(v)$ such that $\hat u \in \OSub^=(u)$. By (\ref{caOVW-eq17}), $v \cdot W[\hat u] \in \hat I$. Moreover, by (\ref{newcaOVW-eq4}), we have
$$
\text{$\WSub^=(v \cdot W[\hat u] \cdot w)$ is $f$-monochromatic for color~$c$.} 
$$
In both cases we are done. This completes the proof of \Cref{caOVW-claim4}.

\end{proof}
The ``Moreover" part follows by analyzing  the effectiveness of above proof.
This completes the proof of \Cref{caOVW-lemextSigma1}.
\end{proof}

In summary of $\Pi_1^0$ extension \Cref{caOVW-lemextPi1-2}
and $\Sigma_1^0$ extension \Cref{caOVW-lemextSigma1},
we have
\begin{lemma}\label[lemma]{caOVW-lemextsummary}
Let $(I,W)$ be a $\PP$-condition
and $\varphi(G)$ a $\Sigma_1^0$ formula.
There exists a $\PP$-condition $(\h I,\h W)\leq (I,W)$
deciding $\varphi(G)$.
Moreover, we can choose $(\h I,\h W)$ so that $\h W'\leq_T W'$ and the index of $\h W'\leq_T W'$
and $\h I $ can be obtained $W'$-computably.
\end{lemma}
\begin{proof}
Immediate by \Cref{caOVW-lemextPi1-2} and \Cref{caOVW-lemextSigma1}.
The ``Moreover" part follows by also noting that the predicate
$(I,W)\qvdash_v\varphi(G)$ is $W'$-decidable (uniformly in $I,v$).
\end{proof}

\subsection{Positive requirements}\label[section]{sec:ovw0-forcing-positive}

Recall that the notion of forcing does not structurally ensures that the resulting set~$G$
will contain infinitely many variable kinds, and have infinitely many levels of $f$-monochromaticity.
We therefore need to force these requirements explicitly.
For this, given~$t \in \NN$, consider the following $\Sigma_1^0$-formula $\varphi_t(G)$:
\begin{align}\label{caOVW-defvarphi}
&\text{ for some $n$ with $t\leq |G[0^n]|<|G|$, $G[A^n]$ is $f$-monochromatic.}
\end{align}

Note that $\varphi_t(G)$ implies that $G$ contains more variable kinds than $G\uhr t$.
In order to ensure that $G$ is an $\IOOVW^0$-solution to~$f$, it suffices to force $\varphi_t(G)$ for infinitely many $t$. In case $(I,W)\qvdash_v \varphi_{t(I)}$ for all $v\in I$,
we can simply apply $\Sigma_1^0$-extension \Cref{caOVW-lemextSigma1}
to extend $(I,W)$ to force $\varphi_{t(I)}(G)$.
If on the other hand, $(I,W)\nqvdash_v \varphi_{t(I)}$, we will show that we can delete $v$ without sabotaging the $f$-matching property.

\begin{lemma}
[Positive requirement]
\label[lemma]{caOVW-lem-positiveext}
Let $(I,W)$ be a $\PP$-condition
and $v\in I$.
If $(v,W)\Vdash\neg\varphi_{t(I)}(G)$,
then $(I\setminus\{v\},W)$ is a $\PP$-condition.
\end{lemma}
\begin{proof}
$(I\setminus\{v\}, W)$ is obviously a $\PP$-precondition.
To verify it is $f$-matching,
 let $w\in \wshorter( W)$.
By $f$-\sufficient\ of $(I,W)$,
there exists $\t v\in  I$ such that
\begin{align}\label{caOVW-weak-eq21}
\wsamelong(\t v\cdot w)\text{ is monochromatic for $f$.}
\end{align}
It suffices to show that $\t v\ne v$.
Suppose otherwise. Substitute $\t v $ by $v$
in (\ref{caOVW-weak-eq21}), we get
\begin{align}\label{caOVW-weak-eq211}
\wsamelong( v\cdot w)\text{ is monochromatic for $f$.}
\end{align}
But obviously, we can extend $  v\cdot w$ to some
$  v_1\in (v,W)$
with
\begin{align}\label{caOVW-weak-eq51}
\text{$ v_1$ containing more variable kind
(than $v$) and}
\end{align}
$| v\cdot w|$ is the first occurrence of a new variable kind
in $v_1$,
which means: let $n$ be the number of variables in $v$
\begin{align}\label{caOVW-weak-eq5}
 v_1[A^n] = \wsamelong(v\cdot w)\text{
 (by (\ref{caOVW-weak-eq211})) is monochromatic for $f$.}
\end{align}
Clearly
$$
t(I)= |v|\leq |v\cdot w|= |v_1[A^n]|\overset{(\ref{caOVW-weak-eq51})}{<}|v_1|.
$$
Thus, combining with (\ref{caOVW-weak-eq5}),
we have $\varphi_{t(I)}(v_1)$ holds.
This is a contradiction to $(v,W)\Vdash \neg\varphi_{t(I)}(G)$.
Thus, we are done.

 \end{proof}

 Repeatedly applying Lemma \ref{caOVW-lem-positiveext},
 we have:
 \begin{lemma}[Positive requirement rephrased]
 \label[lemma]{caOVW-lem-positiveextrephrase}
Let $(I,W)$ be a $\PP$-condition.
There is a $\PP$-condition
$(\h I,\h W)\leq (I,W)$ such that
for every $\h v\in \h I $, $\varphi_{t(I)}(\h v)$ holds.
Moreover, we can choose $(\h I,\h W)$
so that  $\h W = W\uhr_m^\omega$ for some $m$ and the index of
 $\h I $ can be obtained $W'$-computably.
 \end{lemma}
 \begin{proof}
 By $\Pi_1^0$-extension \Cref{caOVW-lemextPi1-2},
there is a set $J\subseteq I$ and a $\PP$-condition $(I,U)\leq (I,W)$
 such that
\begin{align}\nonumber
&\text{for every
$v\in I\setminus J$,
$(v,U)\Vdash \neg\varphi_{t(I)}(G);$} \\ \nonumber
&\text{for every $v\in J$,
$(I,U)\qvdash_v \varphi_{t(I)}(G)$.}
\end{align}
By Lemma \ref{caOVW-lem-positiveext},
$(J,U)$ is a $\PP$-condition.
By the $\Sigma_1^0$-extension \Cref{caOVW-lemextSigma1},
there is a $\PP$-condition  $(\h I,\h W)\leq (J,U)$
such that
\begin{align}
\nonumber
&\text{for every $\h v\in \h I$,
$\varphi_{t(I)}(v)$ holds.}
\end{align}
Thus, we are done.
 \end{proof}

Now we are ready to prove Theorem \ref{thm:ovw0-weakly-low}.

\subsection{Proof of \Cref{thm:ovw0-weakly-low}}

Recall that a $\PP^{\msflow}$-condition is a $\PP$-condition $(I,W)$ such that $W$ is low.
Start with the following  $\PP^{\msflow}$-condition $(I_0,W_0)$ where
$I_0:= \{\varepsilon\}$, $W_0:= x_0x_1x_2\cdots$
By Lemma \ref{caOVW-lemextsummary} and Lemma \ref{caOVW-lem-positiveext},
we can $\emptyset'$-compute a sequence
of codes of $\PP^{\msflow}$-conditions $$
(I_0,W_0)\geq (I_1,W_1)\geq\cdots
$$
such that
\begin{enumerate}[label=(\arabic*)]
\item for every $\Sigma_1^0$ formula $\varphi$,
there is some $s\in\omega$ such that $(I_s,W_s)$ decides $\varphi$;

\item for infinitely many $t\in\omega$, there exists $s\in\omega$
such that $t(I_s)\geq t$ and for every $v\in I_s$, $\varphi_t(v)$ holds
(recall (\ref{caOVW-defvarphi}) for definition of $\varphi_t$).
\end{enumerate}
Note that the set of stems $I = \cup_s I_s$ forms an infinite, finitely  branching
tree. To see that $I$ is infinite, note that by \Cref{rem:ovw0-matching-has-solutions}, $I_s \neq \emptyset$ for every~$s \in \NN$. Since this tree is $\emptyset'$-computable, for any set~$P$ of PA degree over~$\emptyset'$, $P$ computes an infinite path through this tree,
 say $G = \cup_s v_s$ where $v_s\in I_s$.

By item (2),
\begin{align}\nonumber
\text{$G$ is a  \wOVW\ solution to $f$.}
\end{align}

To see $P$ computes $G'$, for each $\Sigma_1^0$ formula $\varphi$,
$\emptyset'$-computably find $s$ so that $(I_s,W_s)$ decides $\varphi(G)$
and $P$-computably
decide whether $(v_s,W_s)\Vdash\varphi(G)$ or $(v_s,W_s)\Vdash \neg\varphi(G)$.
Note that we need $P$ to know which stem in $I_s$ is a prefix of~$G$.
Also note that $G\in [v_s,W_s]$ for all $s$.
Thus,
$(v_s,W_s)\Vdash\varphi(G)$ ($(v_s,W_s)\Vdash \neg\varphi(G)$ respectively)
implies $\varphi(G)$ ($\neg\varphi(G)$ respectively) holds.
i.e.,
\begin{align}\nonumber
\text{ whether $\varphi(G)$
holds is $P$-decidable (so $G'\leq_T P$).}
\end{align}
 Thus, $G$ is as desired and we are done.
This completes the proof of \Cref{thm:ovw0-weakly-low}.

\section{Preservation of hyperimmunities}\label[section]{sec:preservation-hyp}

The halting set being $\Sigma^0_1$-definable, every $\omega$-model of~$\ACA_0$ contains~$\emptyset'$.
It follows that to prove that a $\Pi^1_2$-problem~$\Psf$ does not imply $\ACA_0$ over~$\RCA_0$, it suffices
to show the existence of an $\omega$-model of~$\RCA_0 + \Psf$ which does not contain~$\emptyset'$.
Actually, most statements $\Psf$ such that $\RCA_0 + \Psf \nvdash \ACA_0$ are shown to satisfy a stronger property, called \emph{cone avoidance}.

\begin{definition}
A $\Pi^1_2$-problem~$\Psf$ admits \emph{cone avoidance} if for every set~$Z$, every set $C \not \leq_T Z$ and every $Z$-computable $\Psf$-instance~$X$, there is a $\Psf$-solution~$Y$ such that $C \not \leq_T Y \oplus Z$.

\end{definition}

If a $\Pi^1_2$-problem~$\Psf$ admits cone avoidance, then for every non-computable set~$C$, there is an $\omega$-model of $\RCA_0 + \Psf$ which does not contain~$C$. This technique was used, for instance, by Seetapun and Slaman~\cite{seetapun1995strength} to prove that Ramsey's theorem for pairs is strictly weaker than~$\ACA_0$.

Later, the author~\cite{patey2017iterative} introduced the notion of \emph{preservation of hyperimmunities}, as a refinement of cone avoidance to prove separations from other statements such as $\RT^2_2$, where $\RT^n_k$ is the restriction of Ramsey's theorem for $k$-colorings of the $n$-tuples. A function $h : \NN \to \NN$ is \emph{$Z$-hyperimmune} if it is not dominated by any $Z$-computable function, where $g$ \emph{dominates} $h$ means that $\forall x g(x) \geq h(x)$.

\begin{definition}
Let $k \leq \omega$.
A $\Pi^1_2$-problem~$\Psf$ admits \emph{preservation of $k$ hyperimmunities} if for every set~$Z$, every sequence $\langle h_i : i < k \rangle$ of $Z$-hyperimmune functions and every $Z$-computable $\Psf$-instance~$X$, there is a $\Psf$-solution~$Y$ such that for every~$i < k$, $h_i$ is $Y \oplus Z$-hyperimmune.
\end{definition}

Both notions are related: Downey et al~\cite{downey2022relationships} proved that cone avoidance and preservation of 1 hyperimmunity coincide. On the other hand, for every $k$, there is a $\Pi^1_2$-problem which admits preservation of $k$, but not $k+1$ hyperimmunities. In particular, $\RT^2_2$ admits preservation of 1 hyperimmunity, but the following classical argument shows that the bound is optimal:

\begin{lemma}[\cite{jockusch1972ramsey}]\label[lemma]{lem:rt22-not-2-hyp}
$\RT^2_2$ does not admit preservation of 2 hyperimmunities.
\end{lemma}
\begin{proof}
The \emph{principal function} $p_X$ of an infinite set~$X = \{ x_0 < x_1 \dots \}$ is the map $n \mapsto x_n$.
By \cite[Theorem 5.2]{jockusch1968semirecursive}, there exists a $\Delta^0_2$ set $A$ such that $p_A$ and $p_{\overline{A}}$ are both hyperimmune. By Shoenfield's limit lemma~\cite[Theorem 2]{shoenfield1959degrees}, there exists a computable coloring $f : [\NN]^2 \to 2$ such that for every~$x$, $\lim_y f(x,y)$ exists and equals $A(x)$.
Seeing $f$ as a computable instance of~$\RT^2_2$, every infinite $f$-homogeneous set~$H$ is either a subset of~$A$ or of~$\overline{A}$, hence $p_H$ dominates either $p_A$ or $p_{\overline{A}}$.
\end{proof}

Preservation of $k$ hyperimmunities is a particular case of a general notion of preservation of a weakness property.
One can show that if a $\Pi^1_2$-problem $\Psf$ preserves some weakness property, but another $\Pi^1_2$-problem~$\Qsf$ does not, then $\RCA_0 + \Psf \nvdash \Qsf$, as witnessed by an $\omega$-model (see \cite[Section 3.4]{patey2016reverse}). We reprove the separation in the case of preservation of hyperimmunities for the sake of containment:

\begin{lemma}[\cite{patey2017iterative}]\label[lemma]{lem:pres-hyp-sep}
Suppose that $\Psf$ admits preservation of $k$ hyperimmunities, but $\Qsf$ does not.
Then there is an $\omega$-model of $\RCA_0 + \Psf$ which is not a model of~$\Qsf$
\end{lemma}
\begin{proof}
Since $\Qsf$ does not admit preservation of $k$ hyperimmunities, there is some set~$Z$, some sequence $(h_i : i < k)$ of $Z$-hyperimmune functions and a $Z$-computable $\Qsf$-instance~$X$ such that for every $\Qsf$-solution~$Y$ to~$X$,
at least one $h_i$ is not $Y \oplus Z$-hyperimmune. Since $\Psf$ admits preservation of $k$ hyperimmunities, one can define an infinite sequence of sets $Z_0, Z_1, \dots$ such that
\begin{itemize}
    \item[(1)] for every $n \in \omega$ and every $X_0 \oplus \cdots \oplus X_n$-computable $\Psf$-instance~$I$,
    $Z_m$ is a $\Psf$-solution to~$I$ for some~$m \in \omega$.
    \item[(2)] for every~$i < k$ and $n \in \omega$, $h_i$ is $Z_0 \oplus \dots \oplus Z_n$-hyperimmune.
\end{itemize}
Let $\M$ be the $\omega$-structure whose second-order part is the Turing ideal $\I = \{ A : \exists n (A \leq_T Z_0 \oplus \cdots \oplus Z_n) \}$. Since $\I$ is a Turing ideal, then by Friedman~\cite{friedmanSystemsSecondOrder1975}, $\M \models \RCA_0$.
By (1), $\M \models \Psf$, and by (2), $\M \not \models \Qsf$.
\end{proof}

It follows in particular from \Cref{lem:pres-hyp-sep} that if $\Psf$ admits preservation of $k$ hyperimmunities, and
$\Psf$ implies $\Qsf$ over $\omega$-models, then so does~$\Qsf$.
The goal of this section is to prove the following theorem:

\begin{theorem}\label[theorem]{thm:ovw0-omega-hyp}
$\OVW^0$ admits preservation of $\omega$ hyperimmunities.
\end{theorem}

Before proving \Cref{thm:ovw0-omega-hyp}, we derive multiple consequences:
Miller and Solomon~\cite[Problem 1]{miller2004effectiveness} asked whether $\OVW^0$ implies~$\ACA_0$ over~$\RCA_0$.
We answer negatively:

\begin{repmaintheorem}{mthm:ovw0-rt22}
$\OVW^0$ implies neither $\RT^2_2$ nor $\ACA_0$ over $\RCA_0$.
\end{repmaintheorem}
\begin{proof}
By \Cref{thm:ovw0-omega-hyp}, $\OVW^0$ admits preservation of $\omega$ (and a fortiori 2) hyperimmunities,
while by \Cref{lem:rt22-not-2-hyp}, $\RT^2_2$ does not. By \Cref{lem:pres-hyp-sep}, it follows that $\RCA_0 + \OVW^0 \nvdash \RT^2_2$. Since $\ACA_0 \vdash \RT^2_2$ (see Simpson~\cite[Lemma III.7.4]{simpson2009subsystems}), then $\RCA_0 + \OVW^0 \nvdash \RT^2_2$.
\end{proof}

Weaker consequences of $\RT^2_2$ are known not to admit preservation of 2 hyperimmunities.
For instance, the \emph{Ascending Descending Sequence} for linear orders of order type $\omega + \omega^*$ ($\SADS$)
states that every such linear orders admit an infinite ascending or descending sequence (see Hirschfeldt and Shore~\cite{hirschfeldt2007combinatorial}).

\begin{lemma}[\cite{hirschfeldt2009atomic}]\label[lemma]{lem:sads-not-2-hyp}
$\SADS$ does not admit preservation of 2 hyperimmunities.
\end{lemma}
\begin{proof}
Tennenbaum (see Rosenstein~\cite{rosenstein1982linear}) and Denisov (see Goncharov and Nurtazin~\cite{goncharov1973constructive}) constructed a computable linear ordering $\L = (\omega, <_\L)$ of order type $\omega+\omega^*$ with no computable infinite ascending or descending sequence. Let $U$ and $V$ be the $\omega$ and $\omega^*$ parts, respectively. By assumption, $U \sqcup V = \omega$. We claim that their principal functions $p_U$ and $p_V$ are hyperimmune (see \Cref{lem:rt22-not-2-hyp} for the notion of principal function). Indeed, suppose that a computable function $f : \omega \to \omega$ dominates $p_U$. Then, one can compute an infinite $\L$-ascending sequence by letting $x_0$ be the $<_\L$-least element $<_\omega$-smaller than $f(0)$, and having defined $x_n$, letting $x_{n+1} >_\L x_n$ be the $<_\L$-least element $<_\omega$-smaller than~$f(n+1)$. Similarly, if $f$ dominates $p_V$, one can compute an infinite $\L$-descending sequence. On the other hand, any infinite ascending or descending sequence $H$ is a subset of~$U$ or $V$, in which case $p_H$ dominates $p_U$ or $p_V$.
\end{proof}

\begin{corollary}
$\OVW^0$ does not imply $\SADS$ over $\RCA_0$.
\end{corollary}
\begin{proof}
Immediate by \Cref{thm:ovw0-omega-hyp}, \Cref{lem:sads-not-2-hyp} and \Cref{lem:pres-hyp-sep}.
\end{proof}

We now turn to the proof of \Cref{thm:ovw0-omega-hyp}.
Again, we will prove the theorem in its unrelativized form. Fix a  computable coloring $f : A^{<\omega} \to \ell$ and a countable sequence $\vec{h} = (h_i)_{i \in \NN}$ of hyperimmune functions.

A set $X$ is of \emph{computably dominated} (or \emph{hyperimmune-free}) degree if it does not compute any hyperimmune function. In other words, $X$ is of computably dominated degree if every $X$-computable function is dominated by a computable function. Note that any hyperimmune function is $X$-hyperimmune for every set~$X$ of computably dominated degree.
By the computably dominated basis theorem~\cite[Theorem 2.4]{jockusch1972pi}, the class $\mathsf{HIF}$ of all computably dominated degrees is a Scott class.
Recall that a $\PP^{\mathsf{HIF}}$-condition
is a $\PP$-condition\ $(v,W)$ such that $W$ is of hyperimmune-free degree.

\subsection{Diagonalization requirements}\label[section]{sec:ovw0-omega-hyp-diag}

In order to force preservation of $\omega$ hyperimmunities, one needs to satisfy the following requirements, for every hyperimmune function $h \in \vec{h}$ and Turing functional~$\Psi$:
\begin{quote}
$\R_\Psi^h$: Either $\Psi^G$ is non-total or $\Psi^G $ does not dominate $h$.
\end{quote}

The following density lemma shows how to satisfy the $\R_\Psi^h$-requirements.
Since $\mathsf{HIF}$ forms a Scott class, all the abstract lemmas from \Cref{sec:ovw0-forcing-first-jump} and \Cref{sec:ovw0-forcing-positive} hold when restricting to $\PP^{\mathsf{HIF}}$-conditions.

\begin{lemma}[Diagonalization requirement]
\label[lemma]{lem:ovw0-omega-hyp-diag}
Let $h : \NN \to \NN$ be a hyperimmune function, $(I, W)$ be a $\PP^{\mathsf{HIF}}$-condition
and $\Psi$ be a Turing functional. There is a $\PP^{\mathsf{HIF}}$-condition $(\h I,\h W)$ of $ (I, W)$ such that
for every $v \in \h I$,
\begin{enumerate}[label=(\arabic*)]
\item either $(v, \h W)\forces \Psi^G(n)\uparrow$ for some~$n \in \NN$;
\item or $(v, \h W)\forces \Psi^G(n) < h(n)$ for some~$n \in \NN$.
\end{enumerate}
\end{lemma}
\begin{proof}
By \Cref{caOVW-lemextPi1-2}, there is a $\PP^{\mathsf{HIF}}$-condition $(I,U)\leq (I,W)$ and $J\subseteq I$
such that
\begin{enumerate}[label=(\arabic*)]
\item for every $v\in J$ and every $n\in\NN$,
$(I,U)\potentiallyforces_v \Psi^G(n)\downarrow$;
\item for every $v\in I\setminus J$, there is some $n\in\NN$ such that
$(v,U)\forces \Psi^G(n)\uparrow$.
\end{enumerate}
Note that, unfolding the definition of the forcing question, by item (1), for every $n\in\NN$,
there exists $b_n\in\NN$ such that
\begin{align}
\text{for every $v\in J$,
$(I,U)\potentiallyforces_v \Psi^G(n)\downarrow<b_n$.}
\end{align}
Moreover, since the forcing question is $\Sigma^0_1(U)$, the sequence $(b_n)_{n\in\NN}$ can be chosen to be
$U$-computable. Since $U$ is of hyperimmune-free degree and $h$ is hyperimmune, $h$ is $U$-hyperimmune, so
there is some $n\in\NN$ such that $b_n<h(n)$.
By \Cref{caOVW-lemextSigma1},
there is a $\PP^{\mathsf{HIF}}$-condition
$(\h I,\h W)\leq (I,U)$ such that
\begin{align}\nonumber
\text{for every $\h v\in \h I$ extending some stem in $ J$,
$(\h v,\h W)\forces \Psi^G(n)\downarrow<b_n$.}
\end{align}
This completes the proof of \Cref{lem:ovw0-omega-hyp-diag}.
\end{proof}

We are now ready to prove \Cref{thm:ovw0-omega-hyp}.

\subsection{Proof of \Cref{thm:ovw0-omega-hyp}}

Let $f : A^{<\omega} \to \ell$  be a computable coloring, for a finite alphabet~$A$ and some~$\ell \geq 1$.
Let $\vec{h} = (h_i)_{i \in \NN}$ be a countable sequence of hyperimmune functions.
Let 
$$
(I_0,W_0)\geq (I_1,W_1)\geq\cdots
$$
be a sufficiently generic decreasing sequence of $\PP^{\mathsf{HIF}}$-conditions.
By \Cref{lem:ovw0-omega-hyp-diag}, for every $h \in \vec{h}$
 and every Turing functional $\Psi$, there is some $s\in\NN$
 such that for every $v\in I_s$,
\begin{enumerate}[label=(\arabic*)]
\item either $(v,W_s)\forces\Psi^G(n)\uparrow$ for some $n$;
\item or $(v,W_s)\forces\Psi^G(n)\downarrow<h(n)$ for some $n$.
\end{enumerate}
As in the proof of \Cref{thm:ovw0-weakly-low}, let $G\in \cap_s [I_s,W_s]$.
By \Cref{caOVW-lem-positiveext}, $G[A^n]$ is $f$-monochromatic for infinitely many $n$.
By items (1,2), $h$ is $G$-hyperimmune for each $h \in \vec{h}$.
This completes the proof of \Cref{thm:ovw0-omega-hyp}.

\section{Strong weakness of \wHJ\ }\label[section]{scaOVW-subsec1}

Preservation of 1 hyperimmunity shows the inability, for a $\Pi^1_2$-problem~$\Psf$, to encode in its computable instances sufficiently fast-growing functions that dominate some fixed hyperimmune function. There are some $\Pi^1_2$-problems that are so weak that even arbitrary instances have this inability. This yields the notion of strong preservation of 1 hyperimmunity.

\begin{definition}
Let $k \leq \omega$.
A $\Pi^1_2$-problem~$\Psf$ admits \emph{strong preservation of $k$ hyperimmunities} if for every set~$Z$, every sequence $\langle h_i : i < k \rangle$ of $Z$-hyperimmune functions and every \emph{arbitrary} $\Psf$-instance~$X$, there is a $\Psf$-solution~$Y$ such that for every~$i < k$, $h_i$ is $Y \oplus Z$-hyperimmune.
\end{definition}

The equivalence from Downey et al~\cite{downey2022relationships} between preservation of 1 hyperimmunity and cone avoidance extends to their strong counterparts. The most famous example of $\Pi^1_2$-problem admitting strong cone avoidance is Ramsey's theorem for singletons ($\RT^1$), also known as the infinite pigeonhole principle (see Dzhafarov and Jockusch~\cite{Dzhafarov2009Ramseys}). For many hierarchies of partition theorems $(\Psf^n)_{n \in \NN}$, there exists a deep link between strong preservation of properties for $\Psf^n$ and preservation of properties for $\Psf^{n+1}$. This will be exemplified in \Cref{sec:weakness-csl1} to show that $\CSL^1$ admits preservation of 1 hyperimmunity, and therefore does not imply $\ACA_0$ over~$\RCA_0$.

It is natural to ask whether $\OVW^0$ admits strong preservation of 1 hyperimmunity. While it is the case, the proof happens to be surprisingly intricate, and will be decomposed into two important steps. In this section, we will prove that an infinite version of the Hales-Jewett theorem admits strong preservation of $\omega$ hyperimmunities. Then, in \Cref{scaOVW-sec3}, this fact will be used to prove that $\OVW^0$ admits strong preservation of 1 hyperimmunity.

\subsection{Infinite Hales-Jewett theorem}

A \emph{left-variable word} over an alphabet~$A$ is a 1-variable word~$v$ over~$A$ such that $v(0)$ is the first occurrence of the unique variable kind. $\OVW^0$ can be reformulated as a theorem about left-variable words~\cite{hindman2004partition} as follows:

\begin{theorem}[Ordered Variable Word theorem, rephrased]
Fix an alphabet~$A$ and some~$\ell \in \NN$. For every $\ell$-coloring of $A^{<\omega}$, there exists an infinite sequence $(v_n)_{n \in \NN}$ of 1-variable words over~$A$ such that for every~$n > 0$, $v_n$ is a left-variable word,
and the following set is monochromatic:
$$\{ v_0[a_0] \cdot v_1[a_1] \cdots v_s[a_s] : s \in \NN, a_0, \dots, a_s \in A \}
$$
\end{theorem}

Under this formulation, it is a natural strengthening the infinite version of the Hales-Jewett theorem where none of $v_n$ are required to be left-variable words. We shall actually consider an even weaker statement, which can be thought of as a Hales-Jewett counterpart of $\IOOVW$:

\begin{definition}
Let $\IOIHJ$ be the $\Pi^1_2$-problem whose instances are colorings $f : A^{<\omega} \to \ell$ for a finite alphabet~$A$ and some~$\ell \geq 1$. A \emph{$\IOIHJ^0$-solution} to~$f$ is an ordered $\omega$-variable word~$W$ such that for infinitely many $n\in\NN$ in between variable kinds of~$W$, $\WSub^{=}(W \uh_n)$ is $f$-monochromatic.
\end{definition}

Here, an integer~$n \in \NN$ is \emph{in between variable kinds of~$W$} if no variable kind occurs both in ~$W \uh_n$ and in $W \uh_n^\omega$. Note that an $\IOIHJ^0$-solution gives no information about the positions of monochromaticity.
The remainder of this section is devoted to the proof of the following theorem:

\begin{theorem}\label[theorem]{scaOVW-prop3}
$\IOIHJ$ admits strong preservation of $\omega$ hyperimmunities.
\end{theorem}

The study of this seemingly artificial statement is justified by its use in the proof of strong preservation of 1 hyperimmunity of $\OVW^0$ in \Cref{scaOVW-sec3}.

\subsection{Notion of forcing}

Fix a finite alphabet~$A$ and an arbitrary coloring $f : A^{<\omega} \to \ell$ for some number of colors $\ell \geq 1$. We shall use the same notion of forcing as in \Cref{sec:preservation-hyp}, that is, we shall work with $\PP^{\mathsf{HIF}}$-conditions, where $\mathsf{HIF}$ denotes the class of all sets of hyperimmune-free degree.

The only difficulty comes from the fact that, the coloring~$f$ being arbitrary, it might be of hyperimmune degree, and therefore too complex to be \qt{directly accessed}. Thankfully, the core lemmas about first-jump control rely on the so-called Graham-Rotschild tree (see \Cref{sec:towsner-ovw0}), whose combinatorial features are universal, in the sense that they hold for every coloring, and therefore do not need to access~$f$ in particular. It follows that \Cref{caOVW-lemextPi1-2} and \Cref{caOVW-lemextSigma1} about the forcing question still hold, except for the \qt{Moreover} part about uniformity. The main diagonalization lemma of \Cref{sec:preservation-hyp} uses abstractly the lemmas about the forcing question, so it also holds:

\begin{replemma}[Diagonalization requirement]{lem:ovw0-omega-hyp-diag}
Let $h : \NN \to \NN$ be a hyperimmune function, $(I, W)$ be a $\PP^{\mathsf{HIF}}$-condition
and $\Psi$ be a Turing functional. There is an extension $(\h I,\h W)$ of $ (I, W)$ such that
for every $v \in \h I$,
\begin{enumerate}[label=(\arabic*)]
\item either $(v, \h W)\forces \Psi^G(n)\uparrow$ for some~$n \in \NN$;
\item or $(v, \h W)\forces \Psi^G(n) < h(n)$ for some~$n \in \NN$.
\end{enumerate}
\end{replemma}

\subsection{Positive requirements}

 For positive requirements, the formula $\varphi_t$ being $\Sigma_1^{0,f}$, it is no longer $\Sigma^0_1$ since $f$ is not computable, so the $\Pi_1^0$ extension \Cref{caOVW-lemextPi1} and \Cref{caOVW-lem-positiveext} do not hold. 
 We therefore split $\varphi_t$ into two formulas, one focusing on adding more variable kinds, and the 
 other on adding more monochromatic levels.
 For the ``more variable" part, it is still the same as \Cref{caOVW-lem-positiveext}
 (since the ``more variable" predicate is still $\Sigma_1^0$).
 
 \begin{lemma}[Positive requirement: variable]
 \label[lemma]{scaOVW-lemext-Pvariable}
  Let $(I,W)$ be a $\PP^{\mathsf{HIF}}$-condition.
 There exists a $\PP^{\mathsf{HIF}}$-condition
$(\h I,\h W)\leq (I,W)$ such that
 for every $\h v\in \h I$ and $v\in I$ with
$\h v\succeq v$, $\h v$ contains more
variable kinds than $v$.
 \end{lemma}
\begin{proof}
Same as \Cref{caOVW-lem-positiveext}.
 \end{proof}

For the ``more monochromatic level" part, 
for an ordered variable word $v$,
let $p_f(v)$ denotes the number of $t\leq |v|$ in between variable kinds of~$v$
such that $\WSub^=(v\uhr t)$ is $f$-monochromatic.

\begin{lemma}[Positive requirement: monochromaticity]
\label[lemma]{scaOVW-lemext-P}
 Let $(I,W)$ be a $\PP^{\mathsf{HIF}}$-condition
 and $v\in I$.
There exists a $\PP^{\mathsf{HIF}}$-condition
$(J,U)\leq (I,W)$ such that
for every $\h v\in J$
with $\h v\succeq v$,
we have $p_f( \h v)>p_f(v)$.
Moreover, every $v\in I$ admits at most one extension in $J$.
\end{lemma}
\begin{proof}
We first try to find a $w\in \wshorter(W)$ such that
\begin{align}\label{scaOVW-eq8}
\wsamelong(v\cdot w)\text{ is monochromatic for $f$.}
\end{align}
We divide into two cases depending on whether such a $w$ exists.
\bigskip

\noindent\textbf{Case 1:} The word $w$ as (\ref{scaOVW-eq8}) does not exist.
We show that $(I\setminus\{v\}, W)$ is the desired extension.
In this case, there is no $\h v\in I\setminus\{v\}$ such that $\h v\succeq v$.
Therefore, it remains to verify that $(I\setminus\{v\}, W)$
is a $\PP^{\mathsf{HIF}}$-condition. It is obviously a $\PP^{\mathsf{HIF}}$-precondition.
To see that $(I\setminus \{v\},W)$ is $f$-matching,
let $w\in\wshorter(W)$. By $f$-matching of $(I,W)$,
there exists $\t v\in I$ such that
\begin{align}\label{scaOVW-eq9}
\wsamelong(\t v\cdot w)\text{ is monochromatic for $f$.}
\end{align}
Since (\ref{scaOVW-eq8}) does not hold, $\t v \in I \setminus \{v\}$.
The ``Moreover" part follows obviously. Thus we are done for Case 1.
\bigskip

\noindent\textbf{Case 2:} The word $w$ as (\ref{scaOVW-eq8}) exists.
Consider the set $$ J:= \{\h v\cdot w: \h v\in I \}\text{ and ordered variable word }
U:=W\uhr_{|w|}^\infty.$$
We show that $(J,U)$ is the desired extension.
Clearly $(J,U)$ is a $\PP^{\mathsf{HIF}}$-precondition extending $(I,W)$
and (by (\ref{scaOVW-eq8}))
 $$p_f(v\cdot w)> p_f(v).$$
We need to verify that $(J,U)$ is $f$-matching.
Let $u\in \wshorter(U)$, clearly $$w\cdot u\in \wshorter(W).$$
Therefore, by $f$-matching of $(I,W)$,
there exists $v\in I $ such that
\begin{align}
\label{scaOVW-eq10}
\wsamelong(v\cdot w\cdot u)\text{ is monochromatic for $f$.}
\end{align}
But $v\cdot w\in J$. Thus, combine with (\ref{scaOVW-eq10}),
we have $(J,U)$ is $f$-matching.
The ``Moreover" part follows obviously. Thus we are done for Case 2.

\end{proof}

We are now ready to prove  \Cref{scaOVW-prop3}.

\subsection{Proof of \Cref{scaOVW-prop3}}

The proof is the same as \Cref{thm:ovw0-weakly-low}, except that the positive requirements are forced separately:
Let $\vec{h} = (h_n)_{n \in \NN}$ be a countable sequence of hyperimmune functions.
Let 
$$
(I_0,W_0)\geq (I_1,W_1)\geq\cdots
$$
be a sufficiently generic decreasing sequence of $\PP^{\mathsf{HIF}}$-conditions.
By \Cref{lem:ovw0-omega-hyp-diag}, for every hyperimmune function $h \in \vec{h}$
 and Turing functional $\Psi$, there is a $s\in\NN$
 such that for every $v\in I_s$,
\begin{enumerate}[label=(\arabic*)]
\item either $(v,W_s)\forces\Psi^G(n)\uparrow$ for some $n$;
\item or $(v,W_s)\forces\Psi^G(n)\downarrow<h(n)$ for some $n$.
\end{enumerate}
As in the proof of \Cref{thm:ovw0-weakly-low},
let $G\in \cap_s [I_s,W_s]$. 
By items (1,2), each $h \in \vec{h}$ is $G$-hyperimmune.
By \Cref{scaOVW-lemext-Pvariable}, $G$ contains infinitely many variable kinds, and by \Cref{scaOVW-lemext-P}, there are infinitely many~$n \in \NN$ in between variable kinds of~$G$ such that $\WSub^{=}(G \uh_n)$ is $f$-monochromatic.
This completes the proof of \Cref{scaOVW-prop3}.

\section{Strong weakness of $\OVW^0$}\label[section]{scaOVW-sec3}


The goal of this section is to prove the following theorem:

\begin{theorem}\label[theorem]{sca-th1}
$\msf{OVW}^0$ admits strong preservation of 1 hyperimmunity.
\end{theorem}

Actually, given a coloring $f: A^{<\omega}\rightarrow\ell$, we shall construct an ordered $\omega$-variable word~$G$ with slightly weaker properties: each level of~$G$ will be $f$-monochromatic, but the color of monochromaticity will depend on the level:

\begin{definition}\label[definition]{def:ovw-level-homogeneous}
An ordered variable word $v$ is \emph{\sectionallymonochromatic} for
a coloring $f: A^{<\omega}\rightarrow\ell$ iff
for every $u_0, u_1 \in \WSub^\star(v)$ such that $|u_0| = |u_1|$, we have $f(u_0) = f(u_1)$.
\end{definition}

This weaker notion of monochromaticity yields the following statement, that we shall refer to as \emph{Level Ordered Variable Word} theorem:

\begin{definition}
Let $\LOVW^0$ be the $\Pi^1_2$-problem whose instances are colorings $f : A^{<\omega} \to \ell$ for a finite alphabet~$A$ and some number of colors $\ell \geq 1$. A \emph{$\LOVW^0$-solution} to~$f$ is an ordered $\omega$-variable word~$W$ which is level-homogeneous for~$f$.
\end{definition}

$\LOVW^0$ can be thought of as $\OVW^0$ without one last application of the infinite pigeonhole principle ($\RT^1$). Because of this, it satisfies a stronger notion of preservation:

\begin{theorem}\label[theorem]{strong-hyp-lovw}
$\LOVW^0$ admits strong preservation of $k$ hyperimmunities for every~$k \in \NN$.
\end{theorem}

Note that $\OVW^0$ does not admit strong preservation of 2 hyperimmunities. Indeed, $\RT^1$ is strongly computably reducible to $\OVW^0$, and it is known that $\RT^1$ even for 2-colorings does not admit strong preservation of 2 hyperimmunities (by \cite[Theorem 5.2]{jockusch1968semirecursive}, there exists a coloring $f : \NN \to 2$ such that $f^{-1}(0)$ and $f^{-1}(1)$ are both hyperimmune).
Before proving \Cref{strong-hyp-lovw}, let us assume it holds, and deduce \Cref{sca-th1}.

\begin{proof}[Proof of \Cref{sca-th1}]
As usual, we prove it in its unrelativized form.
Let $f : A^{<\omega} \to \ell$ be an arbitrary coloring for some finite alphabet~$A$ and $\ell \geq 1$. Let $h : \NN \to \NN$ be a hyperimmune function.
By \Cref{strong-hyp-lovw}, there is an ordered $\omega$-variable word~$W$ which is level-homogeneous for~$f$, and such that $h$ is $W$-hyperimmune.
Let $X = \{ t_0 < t_1 < \dots \}$ be such that $t_n$ is the position of the first occurrence of~$x_n$ in~$W$.
Let $g : \NN \to \ell$ be such that $g(n)$ is the color of $f$-monochromaticity of $W[A^n]$. By Patey~\cite[Theorem 28]{patey2017iterative}, $\RT^1$ admits strong preservation of 1 hyperimmunity, so there is an infinite $g$-homogeneous set~$H \subseteq \NN$ for some color~$c < \ell$ such that $h$ is $H \oplus W$-hyperimmune.
Let $\hat W \leq_T H \oplus W$ be obtained from~$W$ by merging variable kinds so that for every~$u \in \WSub^\star(\hat W)$, $|u| \in \{ t_n : n \in H \}$. In particular, $f(u) = g(n)$ for some~$n \in H$ such that $|u| = t_n$, so $f(u) = c$.
\end{proof}

Through this section, fix a coloring $f: A^{<\omega}\rightarrow\ell$ and a family $\vec{h} = (h_i)_{i < k}$ of hyperimmune functions for some~$k \in \NN$.
We shall construct $G$ through the following notion of forcing:

\begin{definition}
A \emph{$\QQ$-condition} is a pair $(v, W)$ such that
\begin{itemize}
    \item[(1)] $v$ is a finite ordered variable word and $W$ an ordered $\omega$-variable word;
    \item[(2)] $v$ is level-homogeneous for $f$;
    \item[(3)] each $h \in \vec{h}$ is $W$-hyperimmune.
\end{itemize}
\end{definition}

In other words, a $\QQ$-condition corresponds with a $\PP$-branch with some color constraints on~$v$ and some computability-theoretic weakness on~$W$. Note that this notion of forcing is parameterized by both the coloring~$f$, and the family of hyperimmune function~$\vec{h}$. The denotation $[v, W]$ of a $\QQ$-condition is exactly the same as for a $\PP$-branch.
In particular, do not impose any $G \in [v,W]$ to be level-homogeneous for $f$.
The notion of $\QQ$-extension is also the same as $\PP$-branch extension. We state it again for the sake of completeness:

\begin{definition}[$\QQ$-condition extension]\label[definition]{scaOVW-def-conditionextension}
A $\QQ$-condition\ $(\h v,\h W)$ \emph{extends}
a $\QQ$-condition\ $(v,W)$
(written $(\h v,\h W)\leq (v,W)$)
iff: $\h v\succeq v$ and $\h v\cdot \h W\in \OSub^\omega(v\cdot W)$.
\end{definition}


\subsection{Requirement and progress}

For a Turing functional $\Psi$ and some hyperimmune function~$h \in \vec{h}$, let $\R_\Psi^h$ be the requirement:
\begin{quote}
$\R_\Psi^h$: Either $\Psi^G$ is non-total or $\Psi^G $ does not dominate $h$.
\end{quote}
We now define what it means for a $\QQ$-condition to satisfy a requirement $\R^h_\Psi$.
Instead of asking $\R_\Psi^h$ to hold for every $G \in [v, W]$, we ask for a stronger property:

\begin{definition}
A $\QQ$-condition\ $(v,W)$ \emph{satisfies} $\mcal R_\Psi^h$
iff:
\begin{itemize}
\item  either for every $G\in [v,W]$,
$\Psi^G$ is non total
\item or
$\Psi^v(n)\downarrow < h(n)$ for some $n$.
\end{itemize}
\end{definition}

This notion of satisfaction is more uniform, in the sense that every $G \in [v, W]$ satisfies $\R_\Psi^h$ by the same disjunct. This stronger definition will be used in the proof of \Cref{sca-th1}.
In general, even for some fixed~$h \in \vec{h}$ and $\Psi$, the set of $\QQ$-conditions satisfying an $\R_\Psi^h$-requirement is not dense.


\begin{definition}[Non-progressable]
A $\QQ$-condition\ $(v,W)$ is \emph{non-progressable for~$h$}
iff
for some Turing functional $\Psi$,
there is no extension $(\h v,\h W)$ of $(v,W)$ such that
$(\h v,\h W)$ satisfies $\mcal R_\Psi^h$.
Otherwise, we say $(v,W)$ is \emph{progressable for $h$}
(i.e., it can be extended to satisfy $\R_\Psi^h$ for every~$\Psi$).
\end{definition}

The whole difficulty of the proof will be to show the existence of a $\QQ$-condition $(\epsilon, W^*)$ below which every $\QQ$-extension is progressable for each~$h \in \vec{h}$. In other words, the set of $\QQ$-conditions satisfying any $\R_\Psi^h$-requirement is dense below~$(\epsilon, W^*)$.

\subsection{Witness for progress}\label[section]{sec:sca-ovw-witness}

We now introduce the core concept of witness for progress which will be used to construct the initial condition $(\epsilon, W^*)$.

\begin{definition}
Let $W$ be an ordered $\omega$-variable word such that each $h \in \vec{h}$ is $W$-hyperimmune.
A \emph{$W$-witness for progress for~$h$} is a coloring $g : A^{<\omega} \to \ell$
such that for every finite ordered variable word~$v$ such that $\WSub^=(v)$ is $g$-monochromatic and level-homogeneous for~$f$, the $\QQ$-condition $(v, W \uh_{|v|}^\omega)$ is progressable for~$h$.
\end{definition}

Given such an ordered $\omega$-variable word $W$ and $h \in \vec{h}$, we write $\wc_h(W)$ for the class of all $W$-witnesses for progress for~$h$.
It is of course not clear at all that $\wc_h(W) \neq \emptyset$. Before proving the existence of such witnesses for progress, we sketch how they will be used in the proof.

\begin{remark}[How to use witnesses for progress]
\label{scaOVW-rem1}
Suppose we found some family of colorings $\vec{g} = (g_i)_{i < k}$ of type $A^{<\omega} \to \ell$ such that $g_i$ is not only a $W$-witness for progress for~$h_i$, but also simultaneously a $U$-witness for progress for~$h_i$ for every $U \in \OSub^\omega(W)$ such that each $h \in \vec{h}$ is $U$-hyperimmune.

Then, by letting $v$ staying in the monochromatic-for-$g$ world for each~$g \in \vec{g}$,
$(v,U)$ will always be progressable for each~$h \in \vec{h}$.
Thus, one may want to construct an ordered $\omega$-variable word
$W^*\in \wsamelong(W)$ that is somewhat $g$-monochromatic for each~$g \in \vec{g}$,
and consider $\QQ$-conditions of the form $(v,U)$
where $(v, U)\leq (\varepsilon, W^*)$.
Here, by \qt{somewhat $g$-monochromatic}, we shall see that it is sufficient for~$W^*$ to be a \wHJ\  solution to~$g$. We will therefore use strong preservation of $\omega$-hyperimmunity of \wHJ\ proven in \Cref{scaOVW-subsec1}.
\end{remark}

The remainder of \Cref{sec:sca-ovw-witness} is devoted to the proof of existence of such witnesses.
Consider the following notion of matching:

\begin{definition}[\hjsufficient]
A set $I$ of ordered variable words, each of length $t$,
is \emph{\hjsufficient} iff
for every coloring $g: A^t\rightarrow \ell$,
there is some  $v\in I$ such that
$\wsamelong(v)$ is monochromatic for $g$.
\end{definition}

A starting point but also a key observation
is that the set of initial segments which are non-progressable for~$h$ cannot form a \hjsufficient\ set.
\begin{lemma}
[Set of \nonprogressable\ initial segment is not \hjsufficient]
\label{scaOVW-lemkey}
 Let $W$ be an ordered $\omega$-variable word such that each $h \in \vec{h}$ is $W$-hyperimmune and fix some~$h \in \vec{h}$.
 Let $I$ be a \hjsufficient\
set of ordered variable words, with each element of $I$ of length
$t$, such that $\wshorter(v)$ is \sectionallymonochromatic\ for
$f$ for all $v\in I$.
Then there is some $v\in I$ such that $(v,W)$ is a $\QQ$-condition which is progressable for~$h$.

\end{lemma}
\begin{proof}
Suppose otherwise. For each~$v \in I$, let $\Psi_v$ be the witness  of
$(v,W)$ being \nonprogressable\ for~$h$.
To derive a contradiction,
we firstly try to  entail $\Psi_v^G$ to be non-total.
To this end, for every $n\in\NN$,
consider the set $Q_n$ of colorings $g:A^{<\omega}\rightarrow\ell$
with $g\uhr A^{\leq t} = f\uhr A^{\leq t}$ such that
\begin{align}\label{scaOVW-defQ}
&\text{ for every $v\in I$,
every $\h v\in [v,W]$ with
}\\ \nonumber
&\text{$\h v$ \sectionallymonochromatic\ for
$g$, we have } \Psi^{\h v}(n)\uparrow.
\end{align}
Intuitively, $Q_n\ne\emptyset$ is seen as it is possible to entail $\Psi_v^G(n)\uparrow$
for some $v\in I$.
Clearly $Q_n$ is a $\Pi_1^{0,W}$ class uniformly in $n$.
We divide into two cases and will show that
in each case there is some $v\in I$ so that $(v,W)$ can be extended to satisfy
$\mcal R_{\Psi_v}$, thus a contradiction to the otherwise hypothesis.

\ \\

\noindent\textbf{Case 1:} For every $n\in\NN$, $Q_n=\emptyset$.

By  compactness, for every $n\in\NN$, there  exists  a finite set $I_n$
of ordered variable words witnessing $Q_n=\emptyset$. That is:
for every coloring $g:A^{<\omega}\rightarrow\ell$ with
 $g\uhr A^{\leq t} = f\uhr A^{\leq t}$,
 there exists   $\h v\in I_n$ such that:
 \begin{align}\label{sphypOVW-eq0}
 &\text{ for some
 $v\in I$,  we have $\h v\in [v,W]$;}\\ \nonumber
 &\text{
 $\h v$ \sectionallymonochromatic\ for
$g$ and  $\Psi^{\h n}(n)\downarrow$.}
\end{align}
Moreover, $(I_n: n\in\NN)$ could be chosen so that the map
$n \mapsto I_n$ is $W$-computable.
We may also assume $I_n$ is not redundant in the sense that
for each $\h v\in I_n$, there exists a $v\in I$ and coloring $g$ so that
$\h v,v,g$ together satisfies (\ref{sphypOVW-eq0}),
in particular $\Psi_v^{\h v}(n)\downarrow$.
Therefore, consider the following function:
\begin{align}\label{sphypOVW-eq2}
\h h: \NN\ni n\mapsto \max\{\Psi^{\h v}_v(n):\h v\in I_n, v\in I, \h v\succeq v\}.
\end{align}
Note that $\h h$ is $W$-computable, but $h$ is $W$-hyperimmune, therefore there is some
$n^*\in\NN$ such that
\begin{align}
\label{sphypOVW-eq3}
\h h(n^*)<h(n^*).
\end{align}

Since $Q_{n^*}=\emptyset$, in particular, $f\notin Q_{n^*}$,
by replacing $g$ with $f$ in  (\ref{sphypOVW-eq0}),
there exists a $v\in I$ and $\h v\in [v,W]$
 such that:
\begin{align}\nonumber
 \text{
 $\h v$ \sectionallymonochromatic\ for
$f$ and  $\Psi^{\h n}(n^*)\downarrow<h(n^*)$.}
\end{align}
The last inequality follows since
we could chose $\h v$ in $ I_{n^*}$
(see (\ref{sphypOVW-eq2})(\ref{sphypOVW-eq3})).
But now,
$(\h v,W\uhr_{|\h v|-|v|}^\omega)$ is a $\QQ$-condition\ extending
$(v,W) $ that satisfies $\mcal R^h_{\Psi_v}$.
This is a contradiction to the otherwise hypothesis.

\ \\

\noindent\textbf{Case 2:}
For some $n$, $Q_n\ne\emptyset$.

Since $Q_n$ is a $\Pi_1^{0,W}$ class and $h$ is $W$-hyperimmune,
by the computably dominated basis theorem (\cite[Theorem 2.4]{jockusch1972pi}), there exists a $g\in Q_n$ such that each $h \in \vec{h}$ is  $g\oplus W$-hyperimmune.
Let $g_w$ denotes the $w$-shift of $g$, namely the coloring $g_w(\sigma):= g(\sigma\cdot w)$.
Note that $g$ induces a coloring $\h g : A^{<\omega} \to A^t \to \ell$ defined by
$$
\h g: A^{<\omega}\ni\sigma\mapsto g_{W[\sigma]}\uhr A^t.
$$
By preservation of $\omega$ hyperimmunities of $\msf{OVW}^0$
(Theorem \ref{thm:ovw0-omega-hyp}),
let $U$ be an $\msf{OVW}^0$ solution to $\h g$
such that $h$ is $g\oplus U\oplus W$-hyperimmune;
and suppose
\begin{align}\label{scaOVW-eq22}
\text{$\wshorter(U)$ is monochromatic for $\h g$
for color $g^*$.}
\end{align}
Note that $g^*$ is an $\ell$-coloring on $A^t$.
 By \hjsufficient\ of $I$,
there exists $v^*\in I$ such that
\begin{align}\label{scaOVW-eq23}
\wsamelong(v^*)\text{ is monochromatic\ for $g^*$,
say for color $\ell^*$}.
\end{align}
Unfolding (\ref{scaOVW-eq22}),
\begin{align}\label{scaOVW-eq24}
\text{for every $u\in \wshorter(U)$,
$g_{W[u]}\uhr A^t = g^*$
}.
\end{align}
Combine with (\ref{scaOVW-eq23}),
for every $u\in \wshorter(U)$ and $w\in \wsamelong(v^*)$,
we have:
\begin{align}
g(w\cdot W[u]) &= g_{W[u]}(w)\\ \nonumber
&\overset{(\ref{scaOVW-eq24})}=
g^*(w) \overset{(\ref{scaOVW-eq23})}{ =}\ell^*.
\end{align}
In particular (and combining with $v^*$ \sectionallymonochromatic\ for $f$ and therefore
\sectionallymonochromatic\ for $g$ since $Q_n$ requires $g\uhr A^{\leq t} = f\uhr A^{\leq t}$),
for every ordered variable word $\h v\in [v^*,W[U]]$,
\begin{align}
\nonumber
\h v\text{ is \sectionallymonochromatic\ for $g$.}
\end{align}
By definition (\ref{scaOVW-defQ}) of $Q_n$
and the fact $g \in Q_n$, and since $[v^*,W[U]]\subseteq [v^*,W]$),
we have:
\begin{align}
\label{scaOVW-eq25}
\text{for every ordered variable word $\h v\in [v^*,W[U]]$, }\Psi_{v^*}^{\h v}(n)\uparrow.
\end{align}
Clearly this means
$(v^*,W[U])$ satisfies $\mcal R^h_{\Psi_{v^*}}$,
 a contradiction to the otherwise hypothesis:  non-progression of $(v^*,W)$ for~$h$.

\ \\
In either case, we derived a contradiction. This completes the proof of \Cref{scaOVW-lemkey}.
\end{proof}

Since the set of initial segments which are \nonprogressable\ for~$h$ 
are not \hjsufficient,
there is a witnessing coloring  by unfolding the definition of \hjsufficient:
\begin{lemma}[Lemma \ref{scaOVW-lemkey} rephrased]
\label{scaOVW-lemkey2}
Let~$W$ be an ordered $\omega$-variable word
such that each $h \in \vec{h}$ is $W$-hyperimmune and fix some $h \in \vec{h}$.
For every $t\in\NN$, there is a coloring
$g:A^t\rightarrow \ell$ such that
for every $\QQ$-condition\
$(v,W)$  with $|v|=t$
and with $\wsamelong(v)$ monochromatic for
 $g$, $(v,W)$ is \progressable\ for~$h$.

\end{lemma}


Rephrasing Lemma \ref{scaOVW-lemkey2}:
Let~$W$ be an ordered $\omega$-variable word
$W$ such that each $h \in \vec{h}$ is $W$-hyperimmune. For every~$h \in \vec{h}$, we have
\begin{align}\label{scaOVW-eqclaim1}
 \wc_h(W)\ne\emptyset.
\end{align}

\subsection{Maximizing non-progressable initial segments}

In Remark \ref{scaOVW-rem1}, we assumed the existence of a $W$-witness for progress in a very strong sense:
a simultaneous witness for every $U \in \OSub^\omega(W)$ such that $h$ is $U$-hyperimmune. It is not clear that such witnesses exist. If we simply take $g \in \wc(W)$, then even if $v$ stays in the monochromatic-for-$g$ world,
we only have that the $\QQ$-condition $(v,W^*\uhr_{|v|}^\omega) $ is \progressable.
But
its extension $(v, U)$ may not be \progressable.
To settle this, we make sure by shrinking the reservoir that
the set of initial segment $v$, with $|v|=t$
and $(v,W^*\uhr_{|v|}^\omega)$ being a \nonprogressable\ $\QQ$-condition, is
somewhat stabilized at infinitely many levels $t$, in the sense
that this set can no longer be changed via shrinking the reservoir $W^*\uhr_{|v|}^\omega$.


\begin{definition}
[Maximizing the set of non-progressable initial segments]
Given an ordered $\omega$-variable word $W$
such that each $h\in\vec{h}$ is $W$-hyperimmune,
we say $W$ \emph{maximizes the set of non-progressable initial segments at level $t$} iff 
for every $h\in\vec{h}$, and every $\QQ$-condition of the form $(v,W\uhr_t^\omega)$ with $|v| = t$, 
if some extension $(v,U)$ of $(v,W\uhr_t^\omega)$
is \nonprogressable\ for $h $, then   $(v,W\uhr_t^\omega)$
is already \nonprogressable\ for
 $h $.

\end{definition}


In particular, if an ordered $\omega$-variable word $W$ maximizes the set of non-progress\-able initial segments at level
$t$, then for any $h\in\vec{h}$ if $(v,W)$ is a  $\QQ$-condition \progressable\ for $h$ with $|v| = t$,
then for any extension $(v,U) $ of $(v,W)$,
$(v,U)$ is also \progressable\ for $h$.

Next we note that there does exist an ordered $\omega$-variable word
that maximizes the set of non-progressable initial segments at infinitely many levels.

\begin{proposition}\label[proposition]{scaOVW-prop30}
There exists an ordered $\omega$-variable word $W $
with each $h\in\vec{h} $ being $W$-hyperimmune
and an infinite set $T\subseteq\omega$ with
each $t\in T$ in between variables of $W$
such that $W$ maximizes the set of non-progressable initial segments at every level $t\in T$.
\end{proposition}
\begin{proof}
This is done by a finite extension argument.
In this proof, we will use a variant of $\QQ$-forcing, where the stems are not required to be level-homogeneous for~$f$. In other words, a condition is a pair $(v, U)$ such that
\begin{itemize}
    \item[(1)] $v$ is a finite ordered variable word and $U$ is an ordered $\omega$-variable word;
    \item[(2)] each $h\in\vec{h}$ is  $U$-hyperimmune.
\end{itemize}
To make sure each $h\in\vec{h}$ is $W$-hyperimmune, given a condition $(v,U)$
and a requirement $\mcal R_\Psi^h$,
we simply find an arbitrary extension $(\h v,\h U)$
satisfying $\mcal R_\Psi^h$
Unlike $\QQ$-forcing, this is now easy since the initial segment need not
be \sectionallymonochromatic\ for a given coloring.
One can similarly ensure that $W$ contains infinitely many variable kinds.

Also note that for any condition $(v,U)$ there is an extension
$(v,\h U)$ so that for each $h\in\vec{h}$, we can no longer increase the set of
ordered variable word  of length $|v|$
that are, together with a reservoir in $\OSub^\omega(\h U)$, \nonprogressable\
for $h$
via shrinking $\h U$.
In other words,
$v\cdot\h U$ maximizes the set of non-progressable initial segments at level $|v|$.
This implies that for every ordered $\omega$-variable word $\t U\in [v,\h U]$,
$\t U$ maximizes the set of non-progressable initial segments at level $|v|$ as well.
We can therefore build a sequence of conditions 
$$
(v_0,U_0)\geq (v_1,U_1)\geq\cdots
$$
such that
$W:= \cup_n v_n$ maximizes the set of non-progressable initial segments at infinitely many levels and each $h \in \vec{h}$ is $W$-hyperimmune.
This completes the proof of \Cref{scaOVW-prop30}.
\end{proof}

We are now ready to prove Theorem \ref{sca-th1}.

\subsection{Proof of \Cref{sca-th1}}

Let $f : A^{<\omega} \to \ell$ be an arbitrary coloring, for some finite alphabet~$A$ and some~$\ell \geq 1$.
Let $\vec{h} = (h_i)_{i < k}$ be hyperimmune functions for some~$k \in \NN$.
By \Cref{scaOVW-prop30}, there is an $\omega$-variable word~$W$ (with variable set $(x_n)_{n \in \NN}$) such that 
each $h\in\vec{h}$ is $W$-hyperimmune and
\begin{align}\label{scaOVW-eqmaximalnonprogressable}
&\text{$W$  maximizes the set of non-progressable initial segments }\\ \nonumber
&\text{at levels
in an infinite set $T\subseteq\NN$ where each $t\in T $  is}\\ \nonumber
&\text{in between variables of $W$.}
\end{align}

By (\ref{scaOVW-eqclaim1}), for each~$i < k$, there is some $g_i\in\wc_{h_i}(W)$. Let $g : A^{<\omega} \to \ell^k$ be the product of $g_i$.
To stay in the monochromatic-for-$g$ world, we use
strong preservation of $\omega$-hyperimmunity for \wHJ.
\begin{claim}\label[claim]{scaOVW-claim3}
There exists an ordered $\omega$-variable word $W^*\in \vwsamelong(W)$
with each $h\in\vec{h}$ being $W^*$-hyperimmune
such that for infinitely many $t\in T$,

\noindent$\wsamelong(W^*\uhr t)$ is monochromatic for $g$.
\end{claim}
\begin{proof}
Let $t_\sigma$ be the largest number in $T$
such that $t_\sigma\leq |W[\sigma]|$, if it exists. Otherwise, $t_\sigma = 0$.
Let $\h g : A^{<\omega} \to \ell^k$ be defined by
$$
\h g(\sigma):= g(W[\sigma]\uhr t_\sigma).
$$
By strong preservation of $\omega$ hyperimmunity
for \wHJ\ (see  \Cref{scaOVW-prop3}),
there exists an ordered $\omega$-variable word $U$ with
each $h\in\vec{h}$ being $ U\oplus W $-hyperimmune such that
for infinitely many $t\in\NN$,
\begin{align}\nonumber
\wsamelong(U\uhr t)\text{ is monochromatic for $\h g$}
\end{align}
Let $W^\star = W[U]$. 
Clearly $W^*\leq_T W\oplus U$ so $h$ is $W^*$-hyperimmune.
It remains to find arbitrary large $t\in T$
such that
\begin{align}
\label{scaOVW-eq7}
\wsamelong(W^*\uhr t)
\text{ is monochromatic for $g$.}
\end{align}
Let $\h t\in\NN$ be such that there exists $t\in T$ large enough
such that
\begin{align}\label{scaOVW-eq6}
t\leq \big|W[A^{\h t}]\big|\text{ and }\wsamelong(U\uhr \h t)\text{ is monochromatic for $\h g$}
\end{align}
Fix such $\h t$ and  assume $t$
is the largest number in $T$ satisfying  $t\leq |W[A^{\h t}]|$.
By definition of $\h g$,
$$
\h g:A^{  \h t  }\ni\sigma\mapsto g(W[\sigma]\uhr t).
$$
Combine with (\ref{scaOVW-eq6})
we have
\begin{align}\nonumber
\wsamelong(W[U\uhr \h t])\uhr t  \text{ is monochromatic for $g$.}
\end{align}
Since $W^*  = W[U]$, one can equivalently say
$$
\wsamelong(W^*\uhr t)\text{ is monochromatic for $g$.}
$$
This verifies (\ref{scaOVW-eq7}). Thus, the proof of \Cref{scaOVW-claim3} is completed.
\end{proof}

 As long as the initial segment stay in the
 somewhat monochromatic-for-$g$ world,
 the $\QQ$-condition is always \progressable\
 for all $h \in\vec{h}$:

\begin{claim}
[Key Claim]
\label{scaOVW-claim0}
For every $\QQ$-condition $(v,U)\leq (\varepsilon, W^*)$,
every $h\in\vec{h}$,
$(v,U)$ is \progressable\ for $h$.
\end{claim}
\begin{proof}
Without loss of generality, suppose $h$ is $h_0$.
By definition of a $\QQ$-extension, $v\cdot U\in \vwsamelong(W^*)$.
By \Cref{scaOVW-claim3}, there exists $t\in T$
with $t\geq |v|$
such that
\begin{align}\nonumber
\wsamelong((v\cdot U)\uhr t)
\text{
is monochromatic for   $g $.
}
\end{align}
Note that we do not require $t$ to be in between variables of~$U$.
Since $g$ is the product of $(g_i: i<k)$,
we have 
\begin{align}\nonumber
\wsamelong((v\cdot U)\uhr t)
\text{
is monochromatic for   $g_0 $.
}
\end{align}
In particular, choose
$\sigma\in A^{<\omega}$ long enough so that
$|v\cdot U[\sigma]|\geq t$,
we have
\begin{align}\nonumber
\wsamelong((v\cdot U[\sigma])\uhr t)
\text{
is monochromatic for  $g_0 $.
}
\end{align}
Combining with $g_0\in\wc_{h_0}(W)$, and since $(v,U)$ is $\QQ$-condition,
we have
\begin{align}\label{scaOVW-eq5}
\text{$\big((v\cdot U[\sigma])\uhr t, W\uhr_t^\omega\big)$
is a $\QQ$-condition\ \progressable\ for $h_0$.}
\end{align}
Say that the variable set of~$v$ is $\{ x_i : i < r \}$.
Let $$U_\sigma:= U[\sigma\cdot x_r x_{r+1}\cdots].$$
In particular, $U_\sigma$ is an ordered $\omega$-variable word over~$A$ with variable set $\{ x_i : i \geq r \}$ so $v \cdot U_\sigma$ is an ordered $\omega$-variable word over~$A$ with variable set $\{ x_i : i \in \NN \}$.
Note that $(v\cdot U_\sigma)\uhr t = (v\cdot U[\sigma])\uhr t$.

\begin{subclaim}\label{scaOVW-subclaim0}
$\big((v\cdot U_\sigma)\uhr t, U_\sigma\uhr_{t-|v|}^\omega\big)\leq
\big((v\cdot U_\sigma)\uhr t, W \uhr_t^\omega\big).$
\end{subclaim}
\begin{proof}
By \Cref{scaOVW-def-conditionextension} of $\QQ$-condition extension,
it suffices to show that
\begin{align}\label{scaOVW-eq28}
U_\sigma\uhr_{t-|v|}^\omega\in
\vwsamelong(W\uhr_t^\omega).
\end{align}
Because
$$v\cdot U\in \vwsamelong(W^*)\subseteq \vwsamelong(W).$$
Which means $$v\cdot U_\sigma
 \in \vwsamelong(W).$$
Which means
\begin{align}\nonumber
U_\sigma\uhr_{t-|v|}^\omega
&= (v\cdot U_\sigma)
\uhr_t^\omega\\ \nonumber
&\in \vwsamelong(W )\uhr_t^\omega \\ \nonumber
&= \vwsamelong(W\uhr_t^\omega).
\end{align}
Thus, verifying (\ref{scaOVW-eq28}) and we are done.

\end{proof}

But $W$ maximizes the set of non-progressable initial segments at level $t$
by (\ref{scaOVW-eqmaximalnonprogressable}),
therefore, combining with Subclaim \ref{scaOVW-subclaim0}
and (\ref{scaOVW-eq5}),
we have
\begin{align}\nonumber
\text{$\big((v\cdot U_\sigma)\uhr t, U_\sigma\uhr_{t-|v|}^\omega\big)$
is a  $\QQ$-condition \progressable\ for $h_0$.}
\end{align}
But $$\big((v\cdot U_\sigma)\uhr t, U_\sigma\uhr_{t-|v|}^\omega\big)\leq(v,U),$$
thus, $(v,U)$ is \progressable\ for $h_0$. Thus, we are done
proving Claim \ref{scaOVW-claim0}.

\end{proof}
By Claim \ref{scaOVW-claim0},
we can build a sequence of $\QQ$-conditions
$$(\varepsilon,W^*)\geq (v_0,W_0)\geq (v_1,W_1)\geq\cdots$$
so that every requirement is satisfied by some $\QQ$-condition in the sequence.
Therefore, $G:= \cup_n v_n$ is an ordered variable word
satisfying all requirements,
meaning each $h\in\vec{h}$ is not dominated by $\Psi^G$ for all Turing functionals $\Psi$.
In other words, each $h\in\vec{h}$ is $G$-hyperimmune.
Clearly $G$ is \sectionallymonochromatic\ for $f$
since it is the case for $v_n$ for each~$n \in \NN$.

To see $G$ contains infinitely many variable kinds, note that
for every $n\in\NN$,
there is such Turing functional $\Psi$ so that
$\Psi^G(m)\downarrow$ (for all $m$) iff $G$ contains more than $n$
variable kinds.
Clearly, for any $\QQ$-condition $(v,W)$,
it cannot be the case $\Psi^G$ is non total for all $G\in [v,W]$.
Thus, by definition of $(v,W)  $ satisfying $\mcal R^{h}_\Psi$,
suppose $(v_N,W_N)$ satisfies $\mcal R^h_\Psi$, then $\Psi^{v_N}(m)\downarrow$.
Which means $v_N $ contains more than $n$ many variable kinds.
Thus $G$ is an $\LOVW^0$-solution to~$f$ such that each~$h\in\vec{h}$ is $G$-hyperimmune.
This completes the proof of \Cref{sca-th1}.

\section{Carlson-Simpson Lemma in $\ACA_0$}\label[section]{sec:csl-aca}

Carlson and Simpson~\cite{carlson1984dual} used the Ordered Variable Word theorem as a pigeonhole principle
to prove inductively what is now known as the Carlson-Simpson Lemma.
As for Ramsey's theorem, the proof consists in reducing the dimension $n+1$ to the dimension~$n$ through the notion of \emph{pre-homogeneous} space, which is a sub-space over which the coloring does not depend on the last dimension. In the case of the Carlson-Simpson Lemma, one can obtain a pre-homogeneous unordered $\omega$-variable word by sequentially applying the Ordered Variable Word. The previous best known upper bound for $\OVW^0$ being any PA over $\emptyset'$, a pre-homogeneous $\omega$-variable word was computed using $\emptyset^{(\omega)}$ and the proof was formalizable over~$\ACA_0^+$. In our case, thanks to the low${}_2$ basis theorem for $\OVW^0$, pre-homogeneous $\omega$-variable words can be obtained over~$\ACA_0$.

The following proof is in the spirit of the original one~\cite{carlson1984dual} and uses the same definitions as~\cite[Section 5]{angles2023carlson}. Only the meta-analysis based on \Cref{ovw0-low2-formalized} differs.

\subsection{Notation}\label{sec:notation-csl}
Given two $\omega$-variable words $W$ and $\hat W$, and $m \in \NN$, we write $\hat W\leq_m  W$ iff $\hat W = W[V]$
for some $\omega$-variable word $V$ such that $x_0\cdots x_{m-1}\prec V$.

Given $\beta \leq \alpha \leq \omega$ and some $\alpha$-variable word~$w$, we write $w \uuh_\beta$ for the $\beta$-variable word obtained by cutting $w$ at the first occurrence of the variable kind~$x_\beta$ (or the $\beta$th variable kind if the variable set is different).

Recall that given $\beta \leq \alpha \leq \omega$, we write $A^{\alpha, \beta}$ and $A^{\leq\alpha, \beta}$ for the sets of all unordered $\beta$-variable word of length~$\alpha$, or at most~$\alpha$, respectively. One defines the unordered counterpart $\USub$ of $\OSub$ in a straightforward way. For instance, given $\gamma \leq \beta \leq \omega$ and an unordered $\beta$-variable word~$w$, we let $\USub^\gamma_{A}(w)$ be the set of all unordered $\gamma$-variable words generated by~$w$, that is,
$$
\USub^\gamma_{A}(w) =  \{ w[u] : u \in A^{\leq\beta, \gamma} \}
$$
The notations $\USub_A(w)$, $\USub^{=}_A(w)$ and $\USub^{\star}_A(w)$ are defined accordingly. 

\subsection{Inductive proof of $\CSL$}

\begin{definition}\label{def:csl-prehomogeneous}
Fix a finite alphabet~$A$ and some~$\ell \geq 1$.
An unordered $\omega$-variable word $W$ over $A$ is \emph{pre-homogeneous} for a coloring $f : A^{<\omega,n+1} \to \ell$ if for every $u \in \USub^{n,\star}_A(W)$, there is a color~$c < \ell$ such that for every~$v \in \USub^{n+1,\star}_A(W)$, if $v \uuh_n = u$, then $f(v) = c$.
\end{definition}

Every pre-homogeneous set~$W$ induces a coloring $\hat f : \USub^{n,\star}_A(W) \to \ell$ defined by $\hat f(u) = c$ for the witness color~$c$. Note that for any unordered $\omega$-variable word~$U \in \USub^\omega_A(W)$ such that $\USub^{n,\star}_A(U)$ is $\hat f$-monochromatic, $\USub^{n+1,\star}_A(U)$ is $f$-monochromatic.


The novelty will be to prove the existence of pre-homogeneous unordered $\omega$-variable words over~$\ACA_0$.
Such $\omega$-variable words can be obtained by iterating the following lemma, which corresponds to \cite[Lemma 5.3]{angles2023carlson}, except that the $\omega$-variable word is provably of low${}_2$ degree.

\begin{lemma}[$\ACA_0$]\label[lemma]{lem:csl-prehomogeneous-step}
Fix a set~$Z$, a finite alphabet~$A$ and $\ell \geq 1$.
Let $W$ be a $Z$-computable $\omega$-variable word, $f : A^{<\omega,n+1} \to \ell$ be a $Z$-computable coloring and $u\in A^{<\omega,n}$. There exists an $\omega$-variable word $\hat W\leq_{|u|+1}W$ and a color $c < \ell$ such that
\begin{itemize}
    \item[(1)] for every $v\in A^{<\omega,n+1}$ with $v \uuh_n = u$, we have $f(\hat W[v])=c$.
    \item[(2)] $\hat W$ is of low${}_2$ degree over~$Z$, that is, $(\hat W \oplus Z)'' \leq_T Z''$.
\end{itemize}
\end{lemma}
\begin{proof}
Let $A_{n+1} = A \sqcup \{x_0, \dots, x_n\}$.
Let $g : A_{n+1}^{<\omega}\to \ell$ be defined as follows: for every $w \in A_{n+1}^{<\omega}$, $g(w) = f(W[u \cdot x_n \cdot w])$. Note that the variables $x_0, \dots, x_n$ are part of the alphabet~$A_{n+1}$, and therefore considered as constants from the viewpoint of $g$. In particular, $g$ is an instance of $\OVW^0$ for $\ell$-coloring of the alphabet~$A_{n+1}$.

By \Cref{ovw0-low2-formalized}, there is an (ordered) $\omega$-variable word~$U$ over $A_{n+1}$ such that the set $\WSub^{\star}_{A_{n+1}}(U) = \{ U[w] : w \in A_{n+1}^{<\omega} \}$ is $g$-monochromatic for some color~$c < \ell$, and $(U \oplus Z)'' \leq_T Z''$.
For simplicity, we can assume that the variable set of $U$ is $\{ x_{|u|+1}, x_{|u|+2}, \dots \}$. Then, $x_0 \cdots x_{|s|} \cdot U$ forms an unordered $\omega$-variable over~$A$. Indeed, $U$ may contain infinitely many occurrences of~$x_s$ for some $s \leq n$.

Let $\hat W = W[x_0 \cdots x_n \cdot U]$. By construction, $\hat W \leq_{|u|+1} W$. Moreover, $\hat W \leq_T U$, so $(\hat W \oplus Z)'' \leq_T Z''$. We claim that (1) holds. Fix some~$v \in A^{<\omega, n+1}$ with $v \uuh_n = u$. In particular, $u \cdot x_n\preceq v$.
Let $w \in A_{n+1}^{<\omega}$ be such that $v = u \cdot x_n \cdot w$. Then $f(\hat W[v]) = f(\hat W[u \cdot x_n \cdot w]) = f(W[u \cdot x_n \cdot U[w]]) = g(U[w]) = c$. This completes the proof of \Cref{lem:csl-prehomogeneous-step}.
\end{proof}

Iterating the previous lemma, we obtain the existence of a pre-homogeneous unordered $\omega$-variable word over~$\ACA_0$.

\begin{lemma}[$\ACA_0$]\label[lemma]{lem:csl-prehomogeneous}
Fix a finite alphabet~$A$ and $\ell \geq 1$.
For every coloring $f : A^{<\omega,n+1} \to \ell$, there is an unordered $\omega$-variable word~$W$ pre-homogeneous for~$f$.
\end{lemma}
\begin{proof}
Let $u_0, u_1, \dots$ be an enumeration of $A^{<\omega, n}$ in shortlex order. In particular, $|u_n| \leq |u_{n+1}|$.
Apply successively \Cref{lem:csl-prehomogeneous-step} to obtain an sequence of unordered $\omega$-variable words $W_0, W_1, \dots$ with $W_0 = x_0x_1\cdots$, such that for every~$n \in \NN$,
\begin{itemize}
    \item[(1)] $W_{n+1} \leq_{|u_n|+1} W_n$;
    \item[(2)] there is a color~$c_n < \ell$ such that for every~$v \in A^{<\omega, n+1}$ with $v \uuh_n = u_n$, we have $f(W_{n+1}[t])=c_n$;
    \item[(3)] $W_n$ is low${}_2$ over~$f$, that is, $(W_n \oplus f)'' \leq_T f$.
\end{itemize}
By (3), the sequence $W_0, W_1, \dots$ is arithmetic in~$f$, so exists by the arithmetic comprehension axiom.
By (1) and the fact that the enumeration $(u_n)_{n \in \NN}$ is of non-decreasing length order, for every~$n \in \NN$ and every~$m \geq n$, $W_m \leq_{|u_n|+1} W_n$. It follows in particular that: $\lim_n W_n$ exists, call it $W$. In particular, for every~$n \in \NN$, $W \leq_{|u_n|+1} W_n$.
By (2) and the previous fact, $W$ is pre-homogeneous for~$f$.
\end{proof}

We are now ready to prove over $\ACA_0$ that the base case and the induction step hold for the statement $\forall n \CSL^n$. However, note that $\ACA_0$ does not prove $\Pi^1_2$-induction, so we cannot deduce that $\forall n \CSL^n$ holds over~$\ACA_0$. We shall see in \Cref{thm:acap-equiv} that the statement $\forall n \CSL^n$ is equivalent to $\ACA_0'$ over~$\RCA_0$, where $\ACA_0'$ is $\RCA_0$ augmented with the statement $\forall n \forall X \exists Y (Y = X^{(n)})$.

\begin{proposition}\label[proposition]{prop:aca0-csl-induction}
$\ACA_0 \vdash \CSL^0 \wedge \forall n(\CSL^n \to \CSL^{n+1})$.
\end{proposition}
\begin{proof}
For the base case, Anglès d’Auriac, Mignoty, Liu and Patey~\cite{angles2023carlson} proved that $\ACA_0 \vdash \OVW^0$ (see also \Cref{sec:towsner-ovw0} for a simple proof). Since $\CSL^0$ is a weakening of~$\OVW^0$, $\ACA_0 \vdash \CSL^0$.
Suppose now that $\CSL^n$ holds for some~$n \in \NN$.
Let $f : A^{<\omega, n+1} \to \ell$ be a coloring.
By \Cref{lem:csl-prehomogeneous}, there is an unordered $\omega$-variable word~$W$ which is pre-homogeneous for~$f$.
Let $\hat f : \USub^{n,\star}_A(W) \to \ell$ be the coloring induced by~$W$ (see after \Cref{def:csl-prehomogeneous}), and let $g : A^{<\omega, n} \to \ell$ be defined by $g(u) = \hat f(W[u])$.
By $\CSL^n$, there is an unordered $\omega$-variable word~$U$ such that $\USub^{n,\star}_A(U)$ is $g$-monochromatic.
In particular, $\USub^{n,\star}_A(W[U])$ is $\hat f$-monochromatic, so $\USub^{n+1,\star}_A(W[U])$ is $f$-monochromatic.
So $\CSL^{n+1}$ holds.
\end{proof}

\subsection{Open and Combinatorial Dual Ramsey theorem}

Given $n, k \in \NN$, we write $\CSL^n(k)$ for the restriction of $\CSL^n$ to alphabets of size~$k$.
Dzhafarov, Flood, Solomon and Westrick~\cite{dzhafarov2021effectiveness} introduced a the \emph{Combinatorial Dual Ramsey theorem} ($\CDRT^n$) representing the combinatorial core of the Open Dual Ramsey theorem ($\ODRT^n$). Formulated in terms of unordered variable words theorems, $\CDRT^n$ corresponds to the statement $\CSL^{n-1}(0)$. They proved that $\RCA_0 \vdash \forall n (\ODRT^n \leftrightarrow \CDRT^n)$. Anglès d’Auriac, Mignoty, Liu and Patey~\cite{angles2023carlson} proved that the statements $\forall n \CDRT^n$ and $\forall n \CSL^n$ are equivalent over~$\RCA_0$. We give a slight improvement:

\begin{lemma}[$\RCA_0$]\label[lemma]{lem:csln-cdrtnp2}
For every~$n \in \NN$, $\CSL^n(1)$ implies both $\CDRT^{n+2}$ and $\ODRT^{n+2}$ over~$\RCA_0$.
\end{lemma}
\begin{proof}
Fix~$n \in \NN$. By Dzhafarov, Flood, Solomon and Westrick~\cite[Theorem 3.9]{dzhafarov2021effectiveness}, $\CDRT^{n+2}$ and $\ODRT^{n+2}$ are equivalent over~$\RCA_0$. Furthermore, $\CDRT^{n+2}$ is $\CSL^{n+1}(0)$, so it suffices to prove $\RCA_0 \vdash \CSL^n(1) \to \CSL^{n+1}(0)$.

Let $f : \emptyset^{<\omega, n+1} \to \ell$ be an instance of $\CSL^{n+1}(0)$.
It induces an instance $g : \{0\}^{<\omega, n} \to \ell$ of $\CSL^n(1)$ defined by $g(u) = f(x_0 \cdot u[x_i \mapsto x_{i+1}, 0 \mapsto x_0])$, where $u[x_i \mapsto x_{i+1}, 0 \mapsto x_0]$ means that all the occurrences of $x_i$ are replaced with $x_{i+1}$ and all the occurrences of the unique alphabet symbol~$0$ are replaced with $x_0$. By $\CSL^n$, there is an unordered $\omega$-variable word~$U$ over $\{0\}$ such that $\USub^{n, \star}_{\{0\}}(U)$ is $g$-monochromatic for some color~$c < \ell$. Let $W = x_0 \cdot U[x_i \mapsto x_{i+1}, 0 \mapsto x_0]$. Then $W$ is an $\omega$-variable word over $\emptyset$ such that $\USub^{n+1, \star}_{\emptyset}(W)$ is $f$-monochromatic for color~$c$. Indeed, given $v \in \USub^{n+1, \star}_{\emptyset}(W)$, since the alphabet is empty, $v = x_0 \cdot v'$ for some $v' \in \{x_0\}^{<\omega, n}$
where the variable set is $\{x_1, x_2, \dots \}$. Then, letting $u = v'[x_0 \mapsto 0, x_{i+1} \mapsto x_i]$, we have $f(v) = g(u) = c$.
\end{proof}

We can now prove our main theorem:

\begin{repmaintheorem}{mthm:csl-aca}
For every~$n \in \omega$, $\ACA_0 \vdash \CSL^n$ and $\ACA_0 \vdash \ODRT^{n+2}$.
Moreover, a reversal holds for every~$n \geq 2$.
\end{repmaintheorem}
\begin{proof}
By external induction on~$n$, using \Cref{prop:aca0-csl-induction}, $\ACA_0 \vdash \CSL^n$ for every~$n \in \omega$.
By \Cref{lem:csln-cdrtnp2}, it follows that $\ACA_0 \vdash \ODRT^{n+2}$ for every~$n \in \omega$.
Conversely, Miller and Solomon~\cite[Theorem 2.3]{miller2004effectiveness} proved that $\RCA_0 \vdash \forall n(\ODRT^{n+1} \to \RT^n)$, hence that $\RCA_0 \vdash \ODRT^4 \to \ACA_0$. By \Cref{lem:csln-cdrtnp2}, $\RCA_0 \vdash \CSL^2 \to \ODRT^4$, so $\RCA_0 \vdash \CSL^2 \to \ACA_0$.
\end{proof}

We conclude this section by proving that the full statement $\forall n \CSL^n$ is equivalent to $\ACA_0'$ over~$\RCA_0$.
Ko{\l}odziejczyk (personal communication) noticed the following consequence of Yokoyama~\cite{yokoyama2023paris}. Note that $\ACA_0'$ is not strong enough to prove $\Pi^1_2$-induction, so we cannot replace $\ACA_0$ with $\ACA_0'$ in the following proposition.

\begin{proposition}\label[proposition]{prop:aca0p-induction}
Let $\Psf(n)$ be a $\Pi^1_2$-formula with free integer variable~$n$.
If $\ACA_0 \vdash \Psf(0) \wedge \forall n(\Psf(n) \to \Psf(n+1))$,
then $\ACA_0' \vdash \forall n \Psf(n)$.
\end{proposition}
\begin{proof}[Proof by Ko{\l}odziejczyk]
Since provability over $\ACA_0$ is provable over~$\ACA_0$, $\ACA_0 \vdash \qt{\ACA_0 \vdash \Psf(0) \wedge \forall n(\Psf(n) \to \Psf(n+1))}$.
Therefore, by $\Sigma^0_1$-induction, $\ACA_0 \vdash \qt{\forall n \ACA_0 \vdash \Psf(n)}$.
By Yokoyama~\cite{yokoyama2023paris} (see also~\cite{pacheco2022determinacy}), $\ACA_0'$ is equivalent to the $\Pi^1_2$-reflection principle for $\ACA_0$, so $\ACA_0' \vdash \forall n \Psf(n)$.
\end{proof}

We now have all the necessary tools to prove the following equivalences:

\begin{theorem}\label[theorem]{thm:acap-equiv}
The following are equivalent over~$\RCA_0$:
\vspace{-8pt}
\begin{multicols}{3}
\begin{itemize}
    \item[(1)] $\ACA_0'$
    \item[(2)] $\forall n \CSL^n$
    \item[(3)] $\forall n \CDRT^n$
    \item[(4)] $\forall n \ODRT^n$
    \item[(5)] $\forall n \RT^n$
\end{itemize}
\end{multicols}
\end{theorem}
\begin{proof}
$(1) \to (2)$ is \Cref{prop:aca0-csl-induction} combined with \Cref{prop:aca0p-induction}.
$(2) \leftrightarrow (3)$ is Anglès d’Auriac, Mignoty, Liu and Patey~\cite[Theorem 6.3]{angles2023carlson}. $(3) \leftrightarrow (4)$ is Dzhafarov, Flood, Solomon and Westrick~\cite[Theorem 3.9]{dzhafarov2021effectiveness}. $(4) \to (5)$ is Miller and Solomon~\cite[Theorem 2.3]{miller2004effectiveness}.
$(1) \leftrightarrow (5)$ is McAloon~\cite{mcaloon1985paris} (see \cite[Theorem 6.27]{hirschfeldt2015slicing} or \cite[Theorem 8.1.6]{dzhafarov2022reverse}).
\end{proof}

\section{The weakness of $\CSL^1$}\label[section]{sec:weakness-csl1}

By \Cref{mthm:csl-aca}, $\CSL^n$ is equivalent to $\ACA_0$ over~$\RCA_0$ for every $n \geq 2$.
On the other hand, by \Cref{thm:ovw0-omega-hyp}, $\OVW^0$ (and a fortiori $\CSL^0$) admits preservation of $\omega$ hyperimmunities, and therefore does not even imply $\RT^2$ over~$\RCA_0$.
The goal of this section is to close the gap, and prove the following theorem:

\begin{theorem}\label[theorem]{thm:csl1-1-hyperimmunity}
$\CSL^1$ admits preservation of 1 hyperimmunity.
\end{theorem}

\Cref{mthm:csl1-avoidance} is an immediate consequence of \Cref{thm:csl1-1-hyperimmunity}:

\begin{repmaintheorem}{mthm:csl1-avoidance}
Neither $\CSL^1$, nor $\ODRT^3$ implies $\ACA_0$ over~$\RCA_0$.
\end{repmaintheorem}
\begin{proof}
There exists an arithmetical hyperimmune function, so $\ACA_0$ does not admit preservation of 1 hyperimmunity.
By \Cref{thm:csl1-1-hyperimmunity} and \Cref{lem:pres-hyp-sep}, $\CSL^1$ does not imply $\ACA_0$ over~$\RCA_0$.
By \Cref{lem:csln-cdrtnp2}, $\RCA_0 \vdash \CSL^1 \to \ODRT^3$, so $\ODRT^3$ does not imply $\ACA_0$ over~$\RCA_0$ either.
\end{proof}

As mentioned, by Downey et al~\cite{downey2022relationships}, \Cref{thm:csl1-1-hyperimmunity} is equivalent to saying that $\CSL^1$ admits cone avoidance.
First, note that \Cref{thm:csl1-1-hyperimmunity} is optimal, in the following sense:

\begin{lemma}
$\CSL^1(1)$ does not admit preservation of 2 hyperimmunities.
\end{lemma}
\begin{proof}
By \Cref{lem:csln-cdrtnp2}, $\RCA_0 \vdash \CSL^1(1) \to \ODRT^3$.
By Miller and Solomon~\cite[Theorem 2.3]{miller2004effectiveness}, $\RCA_0 \vdash \ODRT^3 \to \RT^2$,
and by \Cref{lem:rt22-not-2-hyp}, $\RT^2_2$ does not admit preservation of 2 hyperimmunities.
\end{proof}

\subsection{Level Carlson-Simpson lemma}

As in \Cref{scaOVW-sec3}, instead of directly proving that $\CSL^1$ admits preservation of 1 hyperimmunity, we shall actually consider a level version of the Carlson-Simpson lemma and prove that it admits preservation of $k$ hyperimmunities simultaneously. Note that one could have proven \Cref{thm:csl1-1-hyperimmunity} without this level version, but it actually gives more information about the relation between the combinatorial features of $\CSL^1$ and its computability-theoretic strength.

\begin{definition}
The \emph{projection} $\pi(w)$ of an unordered $n$-variable-word $w$ over an alphabet~$A$, is the set $\{ t_0 < \cdots t_n \}$ such that for $i < n$, $t_i$ is the position of the first occurrence of the variable~$x_i$ in~$w$, and $t_n = |w|$. Given a coloring $f : A^{<\omega, n} \to \ell$, an unordered $\omega$-variable word~$W$ is \emph{level-homogeneous} for~$f$ if for every $u_0, u_1 \in \USub^{n,\star}_A(W)$ such that $\pi(u_0) = \pi(u_1)$, we have $f(u_0) = f(u_1)$.
\end{definition}

This notion of level-homogeneity is a natural generalization of \Cref{def:ovw-level-homogeneous} for unordered variable words. It induces the following Level Carlson-Simpson lemma:

\begin{definition}
$\LCSL^n$ is the $\Pi^1_2$-problem whose instances are colorings $f : A^{<\omega,n} \to \ell$ for a finite alphabet~$A$ and some $\ell \geq 1$. A \emph{$\LCSL^n$-solution} to~$f$ is an unordered $\omega$-variable word~$W$ which is level-homogeneous for~$f$.
\end{definition}

It is useful to think of $\LCSL^n$ as a \qt{Ramsey-free} version of $\CSL^n$, in the sense that $\CSL^n$ is equivalent to $\RT^{n+1} \wedge \LCSL^n$ over~$\RCA_0$, and we shall see that $\LCSL^1$ does not imply $\RT^2$, even for 2-colorings. In what follows, given an unordered $\omega$-variable word~$W$ over~$A$, we write $\pi(W)$ for the set of all $t_i \in \NN$ such that $t_i$ is the position of the first occurrence of~$x_i$ in $W$.

\begin{lemma}\label[lemma]{lem:csl-lcsl-rt}
$\RCA_0 \vdash \forall n(\CSL^n \leftrightarrow \LCSL^n \wedge \RT^{n+1})$.
\end{lemma}
\begin{proof}
Fix~$n$. Suppose first $\CSL^n$ holds. Then $\LCSL^n$ clearly also holds as it is a weakening of $\CSL^n$. Moreover, by \Cref{lem:csln-cdrtnp2} $\ODRT^{n+2}$ holds, hence by Miller and Solomon~\cite[Theorem 2.3]{miller2004effectiveness}, $\RT^{n+1}$ holds.
Suppose now that $\LCSL^n$ and $\RT^{n+1}$ both hold. Let $f : A^{<\omega, n} \to \ell$ be a coloring, for some alphabet~$A$ and some $\ell \geq 1$.
By $\LCSL^n$, there is an infinite $\omega$-variable word~$W$ such that for every~$u_0, u_1 \in \USub^{n,\star}_A(W)$, if $\pi(u_0) = \pi(u_1)$, then $f(u_0) = f(u_1)$. Let $X = \pi(W)$ and $g : [X]^{n+1} \to \ell$ be defined by $g(F) = f(u)$, for any $u \in \USub^{n,\star}_A(W)$ such that $\pi(u) = F$. Such a coloring is well-defined by choice of~$W$. By $\RT^{n+1}$, there is an infinite $g$-homogeneous subset~$Y \subseteq X$, for some color~$c < \ell$. Let $\h W$ be an unordered $\omega$-variable word in $\USub^\omega_A(W)$ such that $\pi(\h W) = Y$. Note that $\h W$ can be obtained from~$W$ by replacing any variable $x_n$ for $n \in Y - X$ with any fixed constant in~$A$. Then, for any $u \in \USub^{n,\star}_A(\h W)$, $\pi(u) \in [Y]^{n+1}$, so $f(u) = g(\pi(u)) = c$.
\end{proof}

We will prove the following theorem:

\begin{theorem}\label[theorem]{thm:lcsl1-k-hyperimmunity}
$\LCSL^1$ admits preservation of $k$ hyperimmunities for every~$k \in \NN$.
\end{theorem}

Before proving \Cref{thm:lcsl1-k-hyperimmunity}, we assume it holds, and deduce \Cref{thm:csl1-1-hyperimmunity}:

\begin{proof}[Proof of \Cref{thm:csl1-1-hyperimmunity}]
By \Cref{thm:lcsl1-k-hyperimmunity}, $\LCSL^1$ admits preservation of 1 hyperimmunity.
By Patey~\cite[Theorem 23]{patey2017iterative}, so does $\RT^2$. It follows that $\LCSL^1 \wedge \RT^2$
admits preservation of 1 hyperimmunity. By \Cref{lem:csl-lcsl-rt}, $\LCSL^1 \wedge \RT^2$ implies $\CSL^1$ over~$\RCA_0$, and a fortiori over $\omega$-models, so by \Cref{lem:pres-hyp-sep}, $\CSL^1$ admits preservation of 1 hyperimmunity.
\end{proof}

\subsection{Cohesive Carlson-Simpson lemma}

In order to prove \Cref{thm:lcsl1-k-hyperimmunity}, we shall define a \emph{cohesive} version of $\CSL^1$.
Recall that given $\beta \leq \alpha \leq \omega$ and some $\alpha$-variable word~$w$, we write $w \uuh_\beta$ for the $\beta$-variable word obtained by cutting $w$ at the first occurrence of the variable kind~$x_\beta$ (or the $\beta$th variable kind if the variable set is different).

\begin{definition}
Fix~$A$ and $\ell \geq 1$. Let $W$ be an unordered $\omega$-variable word over~$A$.
A coloring $f : \USub^{1,\star}_A(W) \to \ell$ is \emph{stable}
if for every $u \in \WSub^{\star}_A(W)$, there is a threshold $t \in \NN$ and a color $c < \ell$ such that for every $v \in \USub^{1,\star}_A(W)$, if $v \uuh_0 = u$ and $|v| \geq t$, then $f(v) = c$.
\end{definition}

Note that any stable coloring $f : \USub^{1,\star}_A(W) \to \ell$ induces a coloring $\hat f : \WSub^{\star}_A(W) \to \ell$ defined by
$$\hat f(u) = \lim_{v \in \USub^{1,\star}_A(W), v \uuh_0 = u} f(v)$$
Such a coloring is called the \emph{limit coloring} of~$f$.
Let $\CCSL^1$, standing for \emph{Cohesive $\CSL^1$}, be the statement \qt{For every finite alphabet~$A$, every $\ell \geq 1$ and every coloring $f : A^{<\omega, 1} \to \ell$, there is an unordered $\omega$-variable word over which $f$ is stable}.

\begin{theorem}\label[theorem]{thm:ccsl1-omega-hyperimmunity}
$\CCSL^1$ admits preservation of $\omega$ hyperimmunities.
\end{theorem}


Once again, before proving \Cref{thm:ccsl1-omega-hyperimmunity}, let us suppose it holds, and deduce \Cref{thm:lcsl1-k-hyperimmunity}.

\begin{proof}[Proof of \Cref{thm:lcsl1-k-hyperimmunity}]
As usual, we prove the unrelativized version. Let $(h_i)_{i < k}$ be a family of hyperimmune functions for some~$k \in \NN$. Let $f : A^{<\omega, 1} \to \ell$ be a computable coloring, for some finite alphabet~$A$ and some~$\ell \geq 1$.
Since $\CCSL^1$ admits preservation of $\omega$ hyperimmunities (\Cref{thm:ccsl1-omega-hyperimmunity}), there is an unordered $\omega$-variable word~$W$ on which $f$ is stable, and such that each $h_i$ is $W$-hyperimmune. Let $\hat f : \WSub^{\star}_A(W) \to \ell$ be the limit coloring of~$f$.
Since $\LOVW^0$ admits strong preservation of $k$ hyperimmunities (\Cref{strong-hyp-lovw}), there is some ordered $\omega$-variable word~$U$ such that $W[U]$ is level-homogeneous for~$\hat f$, and such that each $h_i$ is $U \oplus W$-hyperimmune. By a greedy algorithm, $U \oplus W$ computes an unordered $\omega$-variable word $\hat W \in \USub^\omega_A(W[U])$ which is level-homogeneous for~$f$.
\end{proof}

The remainder of the section is devoted to the proof of \Cref{thm:ccsl1-omega-hyperimmunity}.

\subsection{Notion of forcing}


In what follows, fix a computable coloring $f : A^{<\omega, 1} \to \ell$ for a finite alphabet~$A$ and some $\ell \geq 1$. Let $\vec{h} = (h_i)_{i \in \NN}$ be a countable collection of hyperimmune functions.
For every~$n \in \NN$, let $A_n = A \sqcup \{ x_0, \cdots, x_{n-1} \}$.

\begin{remark}\label[remark]{rem:unordered-concatenate}
Fix~$n \in \NN$ and $\alpha \leq \omega$.
Given an unordered $n$-variable word~$w$ over~$A$ with variable set $\{ x_i : i < n\}$ and an unordered $\alpha$-variable $u$ over $A_n$ with variable set $\{ x_i : n \leq i < \alpha+n \}$, the concatenation $w \cdot u$ forms an unordered $(\alpha+n)$-variable word over~$A$ with variable set $\{ x_i : i < \alpha+n \}$.
In this case, any element of $\WSub^{=}_A(w \cdot u)$ is of the form 
$$w[v_0] \cdot u[v_1][x_i \mapsto v_0(i)]$$ 
for some $v_0 \in A^n$ and $v_1 \in A_n^{\alpha}$.
Here, we write $u[x_i \mapsto v_0(i)]$ for the word over~$A$ obtained from $u$ by replacing every occurrence of $x_i$ for $i < n$ with $v_0(i)$.
\end{remark}


We now define the notion of forcing for $\CCSL^1$:

\begin{definition}\label[definition]{def:ccsl1-forcing-condition}
An \emph{$\RR$-precondition} is a pair $(w, W)$ such that, for some~$n \in \NN$,
\begin{itemize}
    \item[(1)] $w$ is an unordered $n$-variable word over~$A$ with variable set $\{ x_i : i < n \}$;
    \item[(2)] $W$ is an unordered $\omega$-variable word over~$A_n = A \sqcup \{ x_i : i < n \}$, with variable set $\{ x_i : i \geq n \}$;
    \item[(3)] each $h \in \vec{h}$ is $W$-hyperimmune.
\end{itemize}
An $\RR$-precondition $(w, W)$ is \emph{$f$-stabilizing} if there is a function $\hat f : \WSub^{\star}_A(w) \to \ell$ such that for every $v \in \USub^{1,\star}_A(w \cdot W)$, if $v \uuh_0 \in \WSub^{\star}_A(w)$ and $|v| \geq |w|$, then $f(v) = \hat f(v \uuh_0)$.
An \emph{$\RR$-condition} is an $f$-stabilizing $\RR$-precondition.
\end{definition}

Note that $x_i$ for $i < n$ plays the role of a variable in~$w$, but of an alphabet symbol in~$W$. By \Cref{rem:unordered-concatenate}, $w \cdot W$ is seen as an unordered $\omega$-variable word over~$A$, in which case all occurrences of~$x_i$ in~$w \cdot W$ are now considered as variables.
Given an $\RR$-precondition $(w, W)$, $w$ is called the \emph{stem}, and $W$ is the \emph{reservoir}.
The set of $\RR$-conditions is non-empty, as witnessed by the pair $(\epsilon, W)$, where $W = x_0x_1x_2\cdots$

\begin{definition}
An $\RR$-precondition $(w', U)$ \emph{extends} $(w, W)$ (written $(w', U) \leq (w, W)$)
if $w \preceq w'$ and $w' \cdot U \in \USub^\omega_A(w \cdot W)$.
\end{definition}

It is useful to think of an $\RR$-precondition $(w, W)$ as the compact collection of its candidate solutions
$$
[w, W] = \left\{ U \in \USub_A(w \cdot W) : w \preceq U \right\}
$$
The following lemma states that, as expected, the class of candidate solutions is refined when considering extensions.

\begin{lemma}\label[lemma]{lem:ccsl0-forcing-approx}
Let $(w, W)$ be an $\RR$-precondition. For every $(w', U) \leq (w, W)$, $[w', U] \subseteq [w, W]$.
\end{lemma}
\begin{proof}
Let $v \in [w', U]$. First, $v \in \USub_{A}(w' \cdot U)$, and $w' \cdot U \in \USub^\omega_{A}(w \cdot W)$, so $v \in \USub_{A}(w \cdot W)$. Second, $w' \preceq v$, and $w \preceq w'$, so $w \preceq v$. It follows that $v \in [w, W]$.
\end{proof}

\subsection{Stability tree}

As in \Cref{sec:weakly-low}, a condition is defined as a precondition satisfying some universality property ($f$-matching for $\PP$-forcing and $f$-stabilizing for $\RR$-forcing). Given a condition, the difficulty is to find a precondition extension while maintaining this universality property. This is usually done by considering \qt{sufficiently many} extensions of the stem.
In \Cref{sec:weakly-low}, the appropriate combinatorial object is the Graham-Rothschild tree which, at every level, contains sufficiently many variable words to maintain the $f$-matching property. We now define another tree, called the \emph{stability tree}, which will play the same role for the $f$-stabilizing property.
\smallskip

For every~$n \in \NN$, let $B_n = A \sqcup \{ s_0, \cdots, s_{n-1} \}$, where $s_i$ are fresh constant symbols.
Let $(t_n)_{n \in \NN}$ be inductively defined by
$$t_0 = 0 \mbox{ ; } t_1 = 1 \mbox{ and for } n \geq 1 \mbox{, let } t_{n+1} = t_n + HJ(B_{t_n}, \ell^{\card A^{t_n,1}})$$

The \emph{stability tree for $\ell$-colorings of $A^{<\omega, 1}$} is the tree $\Sc_{A,\ell}$ whose level~$n$ consists of all unordered $n$-variable words~$w$ of length~$t_n$, such that for every~$s < n$, the first occurrence of the variable~$x_s$ appears within the positions $[t_s, t_{s+1})$. This tree is partially ordered by the prefix relation. Note that this tree is computably bounded, and every infinite path is an unordered $\omega$-variable word over~$A$. The fundamental property of this stability tree is the following:

\begin{lemma}[$\RCA_0$]\label[lemma]{lem:stability-tree}
Fix~$A$ and $\ell \geq 1$.
For every~$n \in \NN$ and every coloring $f : A^{t_n,1} \to \ell$, there is some unordered $n$-variable word~$w$ at level~$n$ in~$\Sc_{A,\ell}$ and some function $\hat f : \WSub^{\star}_A(w) \to \ell$ such that for every~$v \in \USub^{1,=}_{A}(w)$, $f(v) = \hat f(v \uuh_0)$.
\end{lemma}
\begin{proof}
By $\Delta^0_0$-induction over~$n$.
The case $n = 0$ is vacuously true as $\USub^{1,=}_{A}(\epsilon) = \emptyset$.
The case $n = 1$ is trivially true as $\USub^{1,=}_{A}(x_0) = \{x_0\}$.
Suppose it holds for $n \geq 1$, we now show that it holds for~$n+1$.
Fix a coloring $f : A^{t_{n+1},1} \to \ell$, and let $b = HJ(B_{t_n}, \ell^{\card A^{t_n,1}})$.

Let $g : B_{t_n}^b \to \ell^{A^{t_n,1}}$ be defined by
$$
g(v) := A^{t_n,1} \ni u \mapsto f(u \cdot v[s_i \mapsto u(i)])
$$
In other words, the color of $g(v)$ is a function of type $A^{t_n,1} \to \ell$.
By choice of~$b = HJ(B_{t_n}, \ell^{\card A^{t_n,1}})$ applied to~$g$, there is some 1-variable word~$w_n \in B_{t_n}^{b,1}$ (with some fresh variable kind, say $\star$, to avoid name clash with~$x_0$) and some color $h : A^{t_n,1} \to \ell$ such that for every~$a \in B_{t_n}$, $g(w_n[a]) = h$.
In other words, for every $u \in A^{t_n,1}$ and every~$a \in B_{t_n}$, 
\begin{align}\label{stability-eq1}
    f(u \cdot (w_n[a])[s_i \mapsto u(i)]) = h(u).
\end{align}

By induction hypothesis applied to $h$, there is some unordered $n$-variable word~$w$ at level~$n$ in~$\Sc_{A,\ell}$ and some coloring $\hat h : \WSub^{\star}_A(w) \to \ell$ such that for every~$u \in \USub^{1,=}_{A}(w)$, $h(u) = \hat h(u \uuh_0)$.
Let $w' = w \cdot w_n[s_i \mapsto w(i)]$. First, note that $w'$ is an unordered $(n+1)$-variable word at level~$n+1$ in~$\Sc_{A,\ell}$. Let $\hat f : \WSub^{\star}_A(w') \to \ell$ be defined by $\hat f(u) = \hat h(u)$ if $|u| < t_n$, and $\hat f(u) = f(v)$ for the unique $v \in \USub^{1,=}_A(w')$ such that $v \uuh_0 = u$ otherwise.
\smallskip

\textbf{Claim 1}: For every~$v \in \USub^{1,=}_{A}(w)$, $f(v) = \hat f(v \uuh_0)$.
If $|v \uuh_0| \geq t_n$, then $v$ is the unique 1-variable word in $\USub^{1,=}_A(w')$ extending $v \uuh_0$, so by definition of~$\hat f$, $f(v) = \hat f(v \uuh_0)$. Suppose now that $|v \uuh_0| < t_n$. Let $u = v \uh_{t_n}$, that is, the initial segment of~$v$ of length~$t_n$ (which is different from $v \uuh_{t_n}$). In particular, $u \in A^{t_n,1}$ and $u \uuh_0 = v \uuh_0$. Say that the unique variable kind of~$v$ is $\star$. Since $u \in A^{t_n,1}$, there is at least one occurrence of~$\star$ in~$u$, say at some position~$j < t_n$. Let $v'$ be such that $v = u \cdot v'$.  Note that $v'$ is a word over the alphabet $B_{t_n} \sqcup \{\star\}$. Let $v''$ be the word over the alphabet~$B_{t_n}$ obtained from~$v'$ by replacing every occurrence of~$\star$ with the symbol $s_j$ from the alphabet $B_{t_n}$. We then have $v' = v''[s_i \mapsto u(i)]$.
In particular, $v'' \in \WSub^=_{B_{t_n}}(w_n)$, so there is some $a \in B_{t_n}$ such that $v'' = w_n[a]$. Then 
\begin{align*}
    f(v) &= f(u \cdot v')\\
        &= f(u \cdot v''[s_i \mapsto u(i)])\\
        &= f(u \cdot (w_n[a])[s_i \mapsto u(i)])\\
        &\overset{(\ref{stability-eq1})}= h(u) = \hat h(u \uuh_0) = \hat h(v \uuh_0)
\end{align*}
This completes the proof of Claim~1, and \Cref{lem:stability-tree}.
\end{proof}

\subsection{First-jump control}

We now turn to the definition of the forcing relation for $\Sigma^0_1$ and $\Pi^0_1$-formulas,
and of a forcing question with the good definitional properties.

\begin{definition}
Let $(w, W)$ be an $\RR$-precondition and let $\varphi(G) \equiv \exists x \psi(G \uh_x)$ be a $\Sigma^0_1$-formula.
\begin{itemize}
    \item[(1)] $(w, W) \Vdash \varphi(G)$ if $\psi(w)$ holds.
    \item[(2)] $(w, W) \Vdash \neg \varphi(G)$ if
    for every $\h v\in [v,W]$, every $n\in\omega$, $\neg\psi(\h v\uhr n)$ holds.
\end{itemize}
\end{definition}

As usual, it is easy to check that if $(w, W) \Vdash \varphi(G)$ for a $\Sigma^0_1$ or $\Pi^0_1$-formula $\varphi(G)$,
then for every $U \in [w, W]$, $\varphi(U)$ holds. On can also prove that the forcing relation for $\Sigma^0_1$-formulas and $\Pi^0_1$-formulas is closed under extensions.

\begin{definition}\label{csl1-avoidance-defforcingquestion}
Let $(w, W)$ be an $\RR$-precondition with variable kinds of $w$ being $\{x_i:i<n\}$;
and let $\varphi(G) \equiv \exists x \psi(G \uh_x)$ be a $\Sigma^0_1$-formula.
We define $(w, W) \qvdash \varphi(G)$ to hold if there is a level~$r \in \NN$ in the stability tree $\Sc_{A,\ell^{A^n}}$
for $\ell^{A^n}$-colorings of~$A^{<\omega, 1}$ such that for every node $v$ at level~$r$ in~$\Sc_{A,\ell^{A^n}}$,
there is some $v_1 \in \USub^=_A(v)$ such that $\psi(w \cdot W[v_1])$ holds.
\end{definition}

We now prove that the forcing question satisfies its abstract specifications, that is, if $(w, W) \qvdash \varphi(G)$ for some $\RR$-condition $(w, W)$, then there is an extension forcing $\varphi(G)$, and if $(w, W) \nqvdash \varphi(G)$, there is an extension forcing $\neg \varphi(G)$. As in \Cref{sec:ovw0-forcing-first-jump}, we start with the negative answer, which is the simplest case.

\begin{lemma}[$\Pi^0_1$-extension]\label[lemma]{lem:ccsl1-forcing-question-pi}
Let $(w, W)$ be an $\RR$-condition, and let $\varphi(G)$ be a $\Sigma^0_1$-formula.
If $(w, W) \nqvdash \varphi(G)$, then there is an $\RR$-condition $(w, U) \leq (w, W)$ such that $(w, U) \Vdash \neg \varphi(G)$.
\end{lemma}
\begin{proof}
Say $\varphi(G) \equiv \exists x \psi(G \uh_x)$ and that $(w, W)$ is an $\RR$-condition with variable kinds of $w$ being $\{x_i:i<n\}$.
Suppose $(w, W) \nqvdash \varphi(G)$ holds. Then, for every level~$r$ there is some node $v$ at level~$r$ in $\Sc_{A,\ell^{A^n}}$ such that for every $v_1 \in \USub^=_A(v)$, $\psi(w \cdot W[v_1])$ does not hold. Let $\Sc$ be the infinite sub-tree of~$\Sc_{A,\ell^{A^n}}$ of all such nodes. Note that $\Sc \leq_T f \oplus W$. By the computably dominated basis theorem (\cite[Theorem 2.4]{jockusch1972pi}), there is an infinite path~$U$ through~$\Sc$ such that each $h \in \vec{h}$ is $U \oplus W$-hyperimmune. Then for every $v_1 \in \USub^{\star}_A(U)$, $\psi(w \cdot W[v_1])$ does not hold. It follows that $(w, W[U])$ is a valid $\RR$-condition extending $(w, W)$ such that $(w, W[U]) \Vdash \neg \varphi(G)$.
\end{proof}

\begin{lemma}[$\Sigma^0_1$-extension]\label[lemma]{lem:ccsl1-forcing-question-sigma}
Let $(w, W)$ be an $\RR$-condition, and let $\varphi(G)$ be a $\Sigma^0_1$-formula.
If $(w, W) \qvdash \varphi(G)$, then there is an $\RR$-condition $(\h w, \h W) \leq (w, W)$ such that $(\h w, \h W) \Vdash \varphi(G)$.
\end{lemma}
\begin{proof}
Say $\varphi(G) \equiv \exists x \psi(G \uh_x)$ and that $(w, W)$ is an $\RR$-condition with variable kinds of $w$ being $\{x_i:i<n\}$.
Note that for every $u'\in A^{<\omega,1}$ with variable kind $\{y\}$,
$w\cdot W[u']$ is an unordered variable word (over $A$) with variable kinds $\{x_i:i<n\}\cup\{y\}$.
Thus, for any $u \in A^n$, $(w\cdot W[u'])[x_i \mapsto u(i)]$ is a 1-variable word over~$A$ with variable kind~$y$. 
Let $\bar f : A^{<\omega,1} \to \ell^{A^n}$
be defined for each~$u' \in A^{<\omega,1}$ by
\begin{align}\label{ccsl1-eq2}
\bar f(u') &:=  A^n \ni u \mapsto f\bigg((w \cdot W[u'])[x_i \mapsto u(i)]\bigg)
\end{align}
Note that that the colors of~$f$ are themselves colorings: $f(u')$ is a coloring of type $A^n \to \ell$. We shall therefore write $f(u')(u)$ for $f(u')$ applied to $u \in A^n$.
Also note that if $(w, W)$ is the initial $\RR$-condition $(\epsilon, x_0x_1\cdots)$, then $\bar f = f$. One can think of~$\bar f$ as the instance of $\CCSL^1$ obtained by transforming the $\RR$-condition $(w, W)$ into the canonical $\RR$-condition $(\epsilon, x_0x_1\cdots)$.

Suppose $(w, W) \qvdash \varphi(G)$ holds. Let $r$ be the level in $\Sc_{A,\ell^{A^n}}$ witnessing that it holds.
By definition, all the variable words at level~$r$ in $\Sc_{A,\ell^{A^n}}$ have length~$t_r$. 
We first stabilize $\bar f$ above 
level $t_r$ with respect to all 1-variable word of length $t_r$
\footnote{We actually only need to stabilize it with respect to a part of such variable words.
But it is notationally convenient to do this for all such variable words.}
via shrinking $W \uuh_{t_r}^\omega$.
To this end,
let $B_r = A \sqcup \{ s_i : i < t_r \}$ where $s_i$ are fresh constant symbols;
let $$g : B_r^{<\omega} \to (A^{t_r,1} \to \ell^{A^n})$$ be defined by
\begin{align}\label{ccsl1-eq3}
g(\sigma) &:= A^{t_r,1} \ni v \mapsto \bar f(v \cdot \sigma[s_i \mapsto v(i)])
\end{align}
Here, we see $A^{t_r,1} \to \ell^{A^n}$ as the space of $\ell^{A^n}$-colorings of $A^{t_r,1}$, and therefore $g$ maps any $\sigma$ to such a coloring. Note that $A^{t_r,1} \to \ell^{A^n}$ is finite as the domain and the codomain are finite.
One can therefore see $g$ as a $W$-computable instance of $\CSL^0$ (and therefore of $\OVW^0$). Since $\OVW^0$ admits preservation of $\omega$ hyperimmunities (\Cref{thm:ovw0-omega-hyp}), there is an (ordered, and a fortiori unordered) $\omega$-variable word~$U$
over alphabet $B_r$ such that
\begin{align}\label{ccsl1-eq6}
\text{$\WSub^{\star}_{B_r}(U)$ is $g$-monochromatic for some color~$\mu : A^{t_r,1} \to \ell^{A^n}$}
\end{align}
and such that each $h \in \vec{h}$ is $U \oplus W$-hyperimmune. Here, $\mu$ is a function, but it is seen as a color of~$g$. 

By the fundamental property of the stability tree (\Cref{lem:stability-tree}) applied to $\mu : A^{t_r,1} \to \ell^{A^n}$, there is some unordered $r$-variable word~$\bar w$ at level $r$ in~$\Sc_{A,\ell^{A^n}}$ and some function $\hat \mu : \WSub^\star_A(\bar w) \to \ell^{A^n}$ such that
\begin{align}\label{ccsl1-eq4}
   \text{for every $\hat v \in \USub^{1,=}_A(\bar w)$, $\mu(\hat v) = \hat \mu(\hat v \uuh_0)$}
\end{align}
By definition of the forcing question, there is some~$w_1 \in \USub^=_A(\bar w)$ such that $\psi(w \cdot W[w_1])$ holds. 
Since $w_1$ is generated by $\bar w$, $w_1$ also satisfy (\ref{ccsl1-eq4}). i.e.,
\begin{align}\label{ccsl1-eq40}
   \text{for every $\hat v \in \USub^{1,=}_A(w_1)$, $\mu(\hat v) = \hat \mu(\hat v \uuh_0)$}
\end{align}
Let $r_1 \leq r$ be the number of variable kinds of~$w_1$. We can assume without loss of generality that the variable set of~$w_1$ is $\{ x_i : n \leq i < n+r_1 \}$ and that the variable set of~$U$ is $\{ x_i : i \geq n+r_1 \}$.
Let $$\hat W = (W \uuh_{t_r}^\omega)[U][s_i \mapsto w_1(i)]$$ and $\hat w = w \cdot W[w_1]$.
\begin{claim}\label[claim]{ccsl1-claim1}
$(\hat w, \hat W)$ is an $\RR$-precondition extending $(w, W)$.
\end{claim}
\begin{proof}
Note that $\hat W \leq_T U \oplus W$, hence for every~$h \in \vec{h}$, $h$ is $\hat W$-hyperimmune.
Moreover, $\hat w = w \cdot W[w_1]$ is an unordered $(n+r_1)$-variable word over~$A$ with variable set $\{x_i : i < n+r_1 \}$ and $\hat W = (W \uuh_{t_r}^\omega)[U][s_i \mapsto w_1(i)]$ is an unordered $\omega$-variable word over~$A_{n+r_1} = A \sqcup \{ x_i : i < n + r_1 \}$ with variable set $\{ x_i : i \geq n+r_1 \}$. It follows that $(\hat w, \hat W)$ is an $\RR$-precondition.

We now show that $(\hat w, \hat W) \leq (w, W)$. We have
\begin{align*}
\hat w \cdot \hat W &= w \cdot W[w_1] \cdot (W \uuh_{t_r}^\omega)[U][s_i \mapsto w_1(i)]\\
 &= w \cdot W\big[w_1 \cdot U[s_i \mapsto w_1(i)]\big]
\end{align*}
It follows that $\hat w \cdot \hat W \in \USub^\omega_A(w \cdot W)$.
This proves \Cref{ccsl1-claim1}.
\end{proof}

Let $\hat f : \WSub^{\star}_A(w) \to \ell$ witness $f$-stabilization of $(w, W)$, that is, for every $v \in \USub^{1,\star}_A(w \cdot W)$, if $v \uuh_0 \in \WSub^{\star}_A(w)$ and $|v| \geq |w|$, then $f(v) = \hat f(v \uuh_0)$
Let $\hat f_1 : \WSub^\star_A(\hat w) \to \ell$ be defined for $u \in \WSub^\star_A(\hat w)$ by
\begin{align}\label{ccsl1-eq5}
\hat f_1(u) = \begin{cases}
    \hat f(u) & \mbox{if } |u| < |w|\\
    \hat \mu(u_1)(u_0) & \mbox{if } u_0\in A^n, u_1\in A^{<t_r}, u = w[u_0] \cdot  (W[x_i \mapsto u_0(i)])[u_1] 
\end{cases}
\end{align}
Note that any $u \in \WSub^\star_A(\hat w)$ falls into one of the two cases above.

\begin{claim}\label[claim]{ccsl1-claim2}
$(\hat w, \hat W)$ is $f$-stabilizing, with witness $\hat f_1$.
\end{claim}
\begin{proof}
Fix some $v \in \USub^{1,\star}_A(\hat w \cdot \hat W)$ such that $v \uuh_0 \in \WSub^{\star}_A(\hat w)$ and $|v| \geq |\hat w|$. Suppose first $|v \uuh_0| < |w|$. Then
$$
f(v) = \hat f(v \uuh_0)  \overset{(\ref{ccsl1-eq5})}= \hat f_1(v \uuh_0)
$$
The first equality holds since $\hat f$ witnesses $f$-stabilization of $(w, W)$.

Suppose now that $|v \uuh_0| \geq |w|$. Since $|v| \geq |\hat w|$, we can write $v$ of the form
$$
v = w[u_0] \cdot W[x_i \mapsto u_0(i)][v_0] \cdot (W \uuh_{t_r}^\omega)\big[\sigma[s_i \mapsto v_0(i)]\big]\big[x_i \mapsto u_0(i)\big]
$$
for some~$u_0 \in A^n$, $v_0 \in \USub^{1,=}(W[w_1])$ and $\sigma \in \WSub^{\star}_{B_r}(U)$.
We firstly show that $f(v)$ does not depend on $\sigma$.
This is due to $\WSub^{\star}_{B_r}(U)$ being monochromatic for $g$ (see (\ref{ccsl1-eq6})).
More precisely, by (\ref{ccsl1-eq6}), $g(\sigma) = \mu$, so
\begin{align*}
    \mu(v_0)(u_0)  &  \overset{(\ref{ccsl1-eq6})}= g(\sigma)(v_0)(u_0)\\
    & \overset{(\ref{ccsl1-eq3})}= \bar f(v_0 \cdot \sigma[s_i \mapsto v_0(i)])(u_0)\\
    & \overset{(\ref{ccsl1-eq2})}= f\bigg(w[u_0] \cdot W\big[v_0 \cdot \sigma[s_i \mapsto v_0(i)]\big][x_i \mapsto u_0(i)]\bigg)\\
    & = f(w[u_0] \cdot W[x_i \mapsto u_0(i)][v_0] \cdot (W \uuh_{t_r}^\omega)[\sigma[s_i \mapsto v_0(i)]][x_i \mapsto u_0(i)])\\
    & = f(v)
\end{align*}
i.e., $\mu$ captures $f$ on a prefix of $v$. Moreover, by choice of $w_1$ (and some segment of $v$ is generated by $w_1$), $\hat \mu$ (thus  $\hat f_1$) captures $\mu$ on the word prefix of $v$.
More precisely, we have
$$
v \uuh_0 = w[u_0] \cdot W[x_i \mapsto u_0(i)][v_0 \uuh_0]
$$
so
$$
\hat f_1(v \uuh_0) \overset{(\ref{ccsl1-eq5})}= \hat \mu(v_0 \uuh_0)(u_0) \overset{(\ref{ccsl1-eq40})}= \mu(v_0)(u_0)
$$
Putting all together, we  have $f(v) = \mu(v_0)(u_0) = \hat f_1(v \uuh_0)$.
This proves our claim.
\end{proof}

Last, note that, by choice of~$w_1$, $\psi(w \cdot W[w_1])$ holds, so $\psi(\hat w)$ holds, hence $(\hat w, \hat W) \Vdash \varphi(G)$. This completes the proof of \Cref{lem:ccsl1-forcing-question-sigma}.
\end{proof}

\subsection{Forcing requirements}

We need to force two families of requirements.
The first one is structural, and ensures that the resulting set will have infinitely many variable kinds.
For this, given~$t \in \NN$, consider the following $\Sigma_1^0$-formula $\varphi_t(G)$:
\begin{align}\label{ccsl1-defvarphi}
&\text{ $G$ contains at least~$t$ variable kinds.}
\end{align}

\begin{lemma}\label[lemma]{lem:ccsl1-inf-variables}
For every $\RR$-condition $(w, W)$ and $t \in \NN$, there is an $\RR$-condition $(\hat w, \hat W) \leq (w, W)$
such that $(\hat w, \hat W) \Vdash \varphi_t(G)$.
\end{lemma}
\begin{proof}
By definition \ref{csl1-avoidance-defforcingquestion} of forcing question, we obviously have
  $(w, W) \qvdash \varphi_t(G)$. Then by the $\Sigma^0_1$-Extension \Cref{lem:ccsl1-forcing-question-sigma}, there is an $\RR$-condition $(\hat w, \hat W) \leq (w, W)$
such that $(\hat w, \hat W) \Vdash \varphi_t(G)$ and we are done.
\end{proof}

The second family of requirements concern preservation of $\omega$ hyperimmunities.
Given a function $h \in \vec{h}$ and some Turing functional $\Psi$, let $\R^h_\Psi$ be the requirement
\begin{quote}
Either $\Psi^G$ is non-total or $\Psi^G $ does not dominate $h$.
\end{quote}

\begin{lemma}\label[lemma]{lem:ccsl1-omega-hyp-diag}
Fix $h \in \vec{h}$ and a Turing functional~$\Psi$.
For every $\RR$-condition $(w, W)$, there is an $\RR$-condition $(\h w,\h W) \leq (w, W)$ such that
\begin{enumerate}[label=(\arabic*)]
\item either $(v, \h W)\forces \Psi^G(n)\uparrow$ for some~$n \in \NN$;
\item or $(v, \h W)\forces \Psi^G(n) < h(n)$ for some~$n \in \NN$.
\end{enumerate}
\end{lemma}
\begin{proof}
Suppose first that there is some~$n \in \NN$ such that $(w, W) \nqvdash \Psi^G(n)\downarrow$.
Then by the $\Pi^0_1$-Extension \Cref{lem:ccsl1-forcing-question-pi}, there is an $\RR$-condition $(\hat w, \hat W) \leq (w, W)$ such that $(\hat w, \hat W) \Vdash \Phi^G(n)\uparrow$, and we are done.

Suppose now that for every~$n \in \NN$, $(w, W) \qvdash \Psi^G(n)\downarrow$.
Unfolding the definition of the forcing question, for every $n \in \NN$,
there exists some $b_n\in\NN$ such that $(\hat w, \hat W) \Vdash \Phi^G(n)\uparrow < b_n$.
Moreover, since the forcing question is $\Sigma^0_1(W)$, the sequence $(b_n)_{n\in\NN}$ can be chosen to be
$W$-computable. Since $h$ is $W$-hyperimmune, there is some $n\in\omega$ such that $b_n<h(n)$.
By the $\Sigma^0_1$-Extension \Cref{lem:ccsl1-forcing-question-sigma},
ther eis an $\RR$-condition $(\hat w, \hat W) \leq (w, W)$ such that $(\hat w, \hat W) \Vdash \Psi^G(n)\downarrow<b_n$.
Either case, we are done.
\end{proof}

We are now ready to prove \Cref{thm:ccsl1-omega-hyperimmunity}.

\subsection{Proof of \Cref{thm:ccsl1-omega-hyperimmunity}}

By \Cref{lem:ccsl1-inf-variables} and \Cref{lem:ccsl1-omega-hyp-diag},
there is an infinite decreasing sequence of $\RR$-conditions
$(w_0, W_0) \geq (w_1, W_1) \geq \cdots$
such that for every~$s \in \NN$, $(w_s, W_s) \Vdash \varphi_s(G)$, and for every~$h \in \vec{h}$ and every Turing functional~$\Psi$, there is some~$s \in \NN$ such that
\begin{itemize}
    \item[(1)] either $(w_s, W_s) \Vdash \Psi^G(n)\uparrow$ for some~$n \in \NN$;
    \item[(2)] or $(w_s, W_s) \Vdash \Psi^G(n)\downarrow < h(n)$ for some~$n \in \NN$.
\end{itemize}
Let $G = \bigcup_s w_s$. Note that $G \in \bigcap_s [w_s, W_s]$.
In particular, for every~$s \in \NN$, $\varphi_s(G)$ holds, so $G$ is an unordered $\omega$-variable word.
Moreover, by item (1-2), every $h \in \vec{h}$ is $G$-hyperimmune.
Last, by the $f$-stabilizing property of each $\RR$-condition, $f$ is stable on~$G$. This completes the proof of \Cref{thm:ccsl1-omega-hyperimmunity}.

\section{Applications and conclusion}\label[section]{sec:applications}

Besides the proof of the Dual Ramsey theorem, the main application of infinite variable word theorems comes from Structural Ramsey Theory, which studies partition theorems of mathematical structures. Given two structures~$\mathbf F, \mathbf G$, we write ${\mathbf G \choose \mathbf F}$ for the set of all embeddings from~$\mathbf F$ to~$\mathbf G$.
We write $\mathbf F\cong \mathbf G$ iff $\mathbf F$ is isomorphic to $\mathbf G$.
Given integers $k, \ell \in \omega$, structures $\mathbf A, \mathbf B$ and $\mathbf C$ we write $\mathbf C \longrightarrow (\mathbf A)^{\mathbf B}_{k,\ell}$ for the following statement:

\begin{definition}
$\mathbf C \longrightarrow (\mathbf A)^{\mathbf B}_{<\omega,\ell}$: For any $k \geq 1$ and any coloring $\chi : {\mathbf A \choose \mathbf B} \to k$, there is an embedding $g \in {\mathbf A \choose \mathbf C}$ such that ${g(\mathbf C) \choose \mathbf B}$ uses at most~$\ell$ many colors.
\end{definition}

For instance, Ramsey's theorem can be reformulated as 
$$(\NN, <) \to (\NN, <)^{\mathbf{n}}_{<\omega, 1}$$ 
for every $n \in \NN$, where $\mathbf{n}$ is the ordered set $(\{0, \dots, n-1\}, <)$.

\begin{definition}
An infinite structure $\mathbf A$ has the \emph{Ramsey property} if for every finite substructure $\mathbf B$, the relation $\mathbf A \to (\mathbf A)^{\mathbf B}_{<\omega,1}$ holds.
\end{definition}

All the infinite structures~$\mathbf A$ that we shall consider in this section have a unique sub-structure~$\mathbf 1$ of size~1, in which case any coloring of ${\mathbf A \choose \mathbf 1}$ corresponds to a coloring of the elements of $\mathbf A$. The Ramsey property of~$\mathbf A$ restricted to $\mathbf 1$ is called \emph{indivisibility} of~$\mathbf A$. 

\begin{definition}
An infinite structure $\mathbf A$ is \emph{indivisible} if $\mathbf A \to (\mathbf A)^{\mathbf 1}_{<\omega, 1}$ holds, where $\mathbf 1$ is its unique sub-structure of size~1. 
\end{definition}

Many infinite structures do not admit the Ramsey property, but still enjoy a weaker partition theorem:

\begin{definition}
Let $\mathbf A$ be an infinite structure.
A finite sub-structure $\mathbf B$ has \emph{big Ramsey degree}~$\ell \in \NN$ if $\ell$ is the least number such that   the relation $\mathbf A \to (\mathbf A)^{\mathbf B}_{<\omega, \ell}$ holds, if it exists. We say that $\mathbf A$ has \emph{finite big Ramsey degree} if every finite sub-structure admits a big Ramsey degree.
\end{definition}

Structural Ramsey theory gives a particular focus on Fraissé limits, such as the dense linear order with no endpoints $(\QQ, <)$, the Rado (or random) graph $\mathcal R$, and in our case, the universal $k$-clique-free Henson graph $\HH_k$.
Each of the previous examples are known to admit a finite big Ramsey degree.
Such partition theorems are usually proven by coding the mathematical structures into trees, and applying a tree partition theorem. For instance, Devlin~\cite{devlin1980some} and Sauer~\cite{sauer2006coloring} used Milliken's tree partition theorem to prove that $(\QQ, <)$ and $\mathcal R$ admit finite Big Ramsey degree, respectively. The exact big Ramsey degrees for the Rado graph were then characterized by Laflamme, Sauer and Vuksanovic~\cite{laflamme2006canonical} 

\subsection{Big Ramsey degree of the triangle-free Henson graph}

The big Ramsey degree of the Henson graphs were a longstanding open question.
In 1986, Komj\'{a}th and R\"{o}dl~\cite{komjath1986coloring} proved that~$\HH_3$ is indivisible
and in 1989, El-Zahar and Sauer~\cite{elzahar1989indivisibility} generalized it to all Henson graphs.
In 1998, Sauer~\cite{sauer1998edge} proved that the edges of~$\HH_3$ have big Ramsey degree~2.
In 2017, Dobrinen~\cite{dobrinen2020ramsey} proved that $\HH_3$ admits finite big Ramsey degree using the method of strong coding trees. She then generalized it to all Henson graphs~\cite{dobrinen2023ramsey} with an elaboration of the technique, and characterized~\cite{dobrinen2020ramseypart2,balko2024exact} the exact big Ramsey degrees of $\HH_3$.
Hubi\v{c}ka~\cite{hubivcka2020big} gave an arguably simpler proof of the finite big Ramsey degree of~$\HH_3$ using the Carlson-Simpson lemma. However, his proof is less precise than the coding tree method, in that it only yields upper bounds to the Ramsey degrees. More recently, in a collective work, Balko, Chodounsk\'y, Dobrinen, Hubi\v{c}ka, Kone\v{c}n\'y,  Vena and Zucker~\cite{balko2021exact} characterized the big Ramsey degrees of free amalgamation classes in finite binary
languages defined by finitely many forbidden irreducible substructures.

From a computability-theoretic and reverse-mathematical viewpoint, Anglès d’Auriac et al.~\cite{angles2023carlson} formalized the proof of Hubi\v{c}ka~\cite{hubivcka2020big} and noticed that for every~$n \in \NN$, $\RCA_0 + \CSL^{2^n+n-1}(1)$ proves that every sub-structure of~$\HH_3$ of size~$n$ admits finite big Ramsey degree, and deduced that over~$\ACA_0^+$, $\HH_3$ has finite big Ramsey degree for~$\HH_3$. Gill~\cite{gill2023note} proved that the indivisibility of $\HH_n$ is not computably true for $n \geq 3$. More recently, Cholak, Dobrinen and McCoy~\cite{cholak2026henson} proved that the indivisibility of~$\HH_3$ holds over~$\ACA_0$ and that the existence of big Ramsey degrees for sub-structures of~$\HH_3$ of size~2 implies $\ACA_0$. Finally, Cholak, Dobrinen, and Towsner~\cite{cholak2026hensonupper} proved that for every standard~$n, m$, the existence of big Ramsey degrees for sub-structures of~$\HH_m$ of size~$n$ holds over~$\ACA_0$ using a Milliken-style proof for coding trees. Note that their proof over~$\ACA_0$ yields exact bounds for finite big Ramsey degree.

We now combine the new bounds on the Carlson-Simpson lemma with the analysis of Anglès d’Auriac et al.~\cite{angles2023carlson} of the proof of Hubi\v{c}ka~\cite{hubivcka2020big} to obtain similar bounds as Cholak, Dobrinen, and Towsner.

\begin{definition}
Let $\GG = (\{0\}^{<\omega, 1}, E)$ be the graph where $E$ is the symmetric binary relation defined as follows:
for $v, w \in \{0\}^{<\omega, 1}$, $v E w$ if $|v| \neq |w|$ and, assuming $|v| < |w|$, the following holds:
\begin{itemize}
    \item[(1)] Passing number property: $w(|v|) = x_0$;
    \item[(2)] Triangle-freeness condition: there is no $i < |v|$ with $v(i) = w(i) = x_0$.
\end{itemize}
\end{definition}

Hubi\v{c}ka~\cite{hubivcka2020big} proved that for every unordered $\omega$-variable word $W$ over~$\{0\}$,
$(\USub^{1,\star}(W), E)$ is computably isomorphic to the triangle-free Henson graph~$\HH_3$. In particular, $\GG \cong \HH_3$. Moreover, the proof holds over~$\RCA_0$. Combined with \Cref{thm:csl1-1-hyperimmunity}, we obtain

\begin{theorem}
The indivisibility of~$\HH_3$ admits preservation of 1 hyperimmunity.
\end{theorem}

Since by Downey et al~\cite{downey2022relationships}, preservation of 1 hyperimmunity and cone avoidance coincide, it follows that the indivisibility of $\HH_3$ admits cone avoidance and does not imply $\ACA_0$ over $\omega$-models, answering a question of Gill~\cite{gill2023note} and Cholak, Dobrinen and McCoy~\cite[Question 48]{cholak2026henson}.
\smallskip

Combining \Cref{mthm:csl-aca} with Anglès d’Auriac et al.~\cite{angles2023carlson}, we obtain:

\begin{theorem}[\cite{cholak2026hensonupper}]
For every $n \in \omega$, $\ACA_0$ proves that every finite triangle-free graph of size~$n$ has finite big Ramsey degree.
Moreover, $\ACA_0'$ proves that $\HH_3$ admits finite big Ramsey degree.
\end{theorem}

\subsection{Summary diagram}

We conclude with a summary diagram (\Cref{fig:summary}):

\begin{figure}[htbp]
\begin{center}
\begin{tikzpicture}[x=2.3cm, y=1.5cm, 
	node/.style={minimum size=2em},
	impl/.style={draw,very thick,-latex},
	strict/.style={draw, thick, -latex, double distance=2pt},
	nonimpl/.style={draw, very thick, dotted, -latex},
    equiv/.style={draw, very thick, latex-latex}
]
	
	\node (ACAp) at (0, 0) {$\ACA_0'$};
	\node (ACA) at (0, -1) {$\ACA_0$};
	\node (WKL) at (0, -4) {$\WKL_0$};

    \node (RT) at (1, 0) {$\RT$};
    \node (RTn) at (1, -1) {$\RT^{n+1}$};
    \node (RT2) at (1, -3) {$\RT^2$};
    \node (RT1) at (1, -5) {$\RT^1$};

    \node (A) at (5, -1) {for $n \geq 2$};

    \node (ODRT) at (2, 0) {$\ODRT$};
    \node (ODRTn) at (2, -1) {$\ODRT^{n+2}$};
    \node (ODRT3) at (2, -2.5) {$\ODRT^3$};
    \node (ODRT2) at (2, -4.5) {$\ODRT^2$};

    \node (CSL) at (4, 0) {$\CSL$};
    \node (CSLn) at (4, -1) {$\CSL^n$};
    \node (CSL1) at (4, -2) {$\CSL^1$};

    \node (LCSL1+RT2) at (5,-2) {$\LCSL^1 + \RT^2$};
    \node (LCSL1) at (4,-3) {$\LCSL^1$};
    \node (CSL0) at (4,-4) {$\CSL^0$};

    \node (OVW0) at (3,-3.5) {$\OVW^0$};

    \node (LCSL0+RT1) at (5,-4) {$\LCSL^0 + \RT^1$};

	\draw[strict] (ACAp) -- (ACA);
    \draw[strict] (ACA) --  (WKL);

    \draw[impl] (ODRT3) -- (RT2);
    \draw[impl] (CSL1) -- (ODRT3);
    \draw[impl] (OVW0) -- (CSL0);
    \draw[impl] (LCSL1) -- (CSL0);
    \draw[impl] (CSL0) -- (ODRT2);
    \draw[impl] (ODRT2) -- (RT1);

    \draw[strict] (CSL) --  (CSLn);
    \draw[strict] (CSL1) --  (LCSL1);
    \draw[strict] (RT) --  (RTn);
    \draw[strict] (RTn) --  (RT2);
    \draw[strict] (RT2) --  (RT1);
    \draw[strict] (ODRT) --  (ODRTn);
    \draw[strict] (ODRTn) --  (ODRT3);
    \draw[strict] (ODRT3) --  (ODRT2);

    \draw[equiv] (ACAp) -- (RT);
    \draw[equiv] (RT) -- (ODRT);
    \draw[equiv] (ODRT) -- (CSL);

    \draw[equiv] (ACA) -- (RTn);
    \draw[equiv] (RTn) -- (ODRTn);
    \draw[equiv] (ODRTn) -- (CSLn);

    \draw[equiv] (LCSL1+RT2) -- (CSL1);
    \draw[equiv] (LCSL0+RT1) -- (CSL0);

    \draw[strict] (ACA) --  (WKL);
    \draw[strict] (CSLn) --  (CSL1);
    
    \draw[nonimpl] (RT2) -- node[pos=0.5, sloped] {/} (WKL);
    \draw[nonimpl] (LCSL1) -- node[pos=0.5, sloped] {/} (RT2);
    \draw[nonimpl] (OVW0) -- node[pos=0.25, sloped] {/} (RT2);
    \draw[nonimpl] (WKL) -- node[pos=0.5, sloped] {/} (RT1);
\end{tikzpicture}

\caption{\label{fig:summary} Summary diagram of implications and separations over~$\RCA_0$. A double arrow denotes a strict implication, and a dotted arrow is a non-implication. Some arrows which can be deduced by the transitive closure of the implication are missing.}
\end{center}

\end{figure}
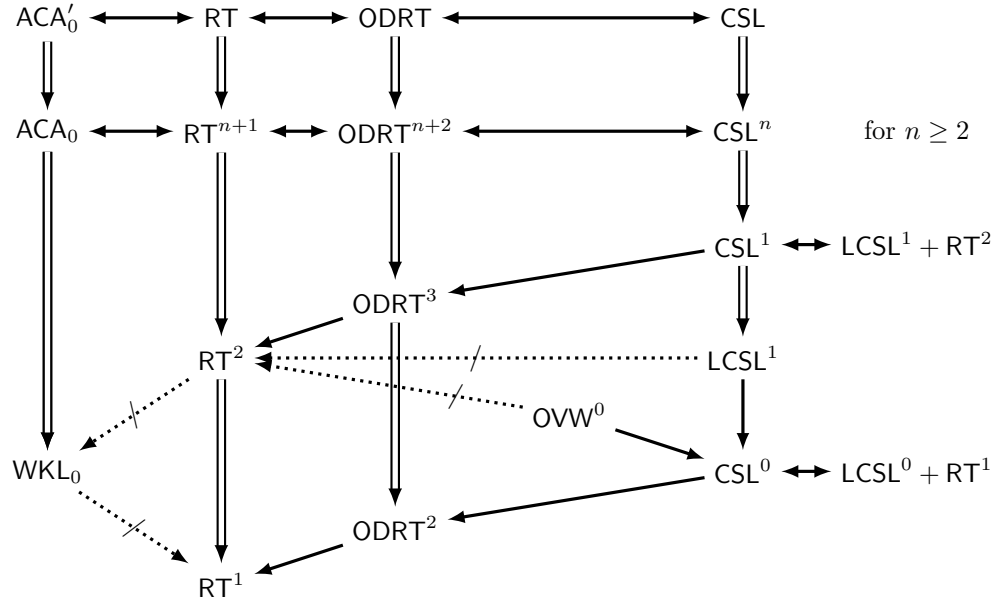

\begin{center}
\textbf{Acknowledgement}
\end{center}

The authors are thankful to Natasha Dobrinen, Leszek Ko\l odziejczyk and Keita Yokoyama for insightful comments and discussions.

\bibliographystyle{plain}
\bibliography{biblio}

\end{document}